\documentclass[12pt,a4paper]{amsart}
\usepackage{amsfonts}
\usepackage{amsthm}
\usepackage{amsmath}
\usepackage{amscd}
\usepackage[latin2]{inputenc}
\usepackage{t1enc}
\usepackage[mathscr]{eucal}
\usepackage{indentfirst}
\usepackage{graphicx}
\usepackage{graphics}
\usepackage{pict2e}
\usepackage{epic}
\numberwithin{equation}{section}
\usepackage[margin=2cm]{geometry}
\usepackage{epstopdf}

\usepackage{textcomp}
\usepackage{mathrsfs}
\usepackage{amsthm}
\usepackage{fancyhdr}
\usepackage{indentfirst}
\usepackage{hyperref}
\usepackage{graphicx}
\usepackage{newlfont}
\usepackage{amssymb}
\usepackage{amsmath}
\usepackage{latexsym}
\usepackage{amsthm}
\usepackage{amscd}
\usepackage{varioref}
\usepackage{enumitem}

\theoremstyle{plain}
\newtheorem{thm}{Theorem}[section]
\newtheorem{lem}[thm]{Lemma}
\newtheorem{cor}[thm]{Corollary}
\newtheorem{prop}[thm]{Proposition}
\newtheorem{oss}[thm]{Observation}
\newtheorem{es}[thm]{Example}
\newtheorem{stm}[thm]{Statement}

 \theoremstyle{definition}
\newtheorem{defi}[thm]{Definition}
\newtheorem{conj}[thm]{Conjecture}

\newtheorem{?}[thm]{Problem}

\begin{document}

\title{Thurston's metric on Teichm\"uller space of semi-translation surfaces}

\author[F. Wolenski]{Federico Wolenski}

\address{Dipartimento di Matematica G. Castelnuovo, Universit\'a Sapienza di
Roma, Piazzale A. Moro 5, 00185, Roma, Italy} 

\email{wolenski@mat.uniroma1.it}



\begin{abstract} 
The present paper is composed of two parts. In the first one we define two pseudo-metrics $L_F$ and $K_F$ on the Teichm\"uller space of semi-translation surfaces $\mathcal{TQ}_g(\underline k,\epsilon)$, which are the symmetric counterparts to the metrics defined by William Thurston on $\mathcal{T}_g^n$. We prove some nice properties of $L_F$ and $K_F$, most notably that they are complete pseudo-metrics. In the second part we define their asymmetric analogues $L_F^a$ and $K_F^a$ on $\mathcal{TQ}_g^{(1)}(\underline k,\epsilon)$ and prove that their equality depends on two statements regarding 1-Lipschitz maps between polygons. We are able to prove the first statement, but the second one remains a conjecture: nonetheless, we explain why we believe it is true.
\end{abstract}

\maketitle

\tableofcontents

\bigskip
\section{Introduction} 
\bigskip

Denote by $\mathcal{T}_g^n$ the Teichm\"uller space of Riemann surfaces of genus $g\ge 2$ and $n\ge 0$ punctures. William Thurston in \cite{Th} defined the following asymmetric metric $L$ on $\mathcal{T}_g^n$:
given any two hyperbolic surfaces $X,X'\in \mathcal{T}_g^n$, their distance with respect to $L$ defined as 
$$L(X,X')=\inf\limits_{\varphi\in Diff^+_0(S_g^n)}\log(Lip(\varphi)_X^{X'}),$$
where $Diff^+_0(S_g^n)$ is the group of diffeomorphisms of $S_g^n$ homotopic to the identity and 
$$Lip(\varphi)_X^{X'}=\sup\limits_{x,y\in S_g^n}\frac{d_{X'}(\varphi(x),\varphi(y))}{d_X(x,y)}$$ 
is the Lipschitz constant of $\varphi$ computed with respect to the hyperbolic metrics of $X$ and $X'$.\\
The main result of \cite{Th} is that for every $X,X'\in \mathcal{T}_g^n$ it results 
\begin{equation}
L(X,X')=K(X,X'),
\end{equation}
where $K$ is another asymmetric metric on $\mathcal{T}_g^n$ defined as
$$K(X,X')=\sup\limits_{\alpha\in \mathcal{S}}\log\left(\frac{\hat l_{X'}(\alpha)}{\hat l_X(\alpha)}\right)$$ 
with $\mathcal{S}$ being the set of homotopy classes of simple closed curves on $S_g^n$ and $\hat l_X(\alpha)$ being the length of the geodesic representative for $X$ of the homotopy class of $\alpha$.\\
The equality (1.1) has been proved by Thurston using the properties of measured laminations on $S_g^n$. Roughly, one could say that the idea of the proof is to triangulate the surface with hyperbolic triangles and then use the fact that for any $c > 1$ there is a $c$-Lipschitz homeomorphism of a filled hyperbolic triangle to itself which maps each side to itself, multiplying arc length on the side by $c$.\\
The Teichm\"uller space endowed with the Thurston's metric is a geodesic space; A.Papadopoulos and G.Th\'eret proved that it is also a complete asymmetric space (\cite{PT}).\\

Every semi-translation surface defines a singular flat metric on $S_g$: the idea of the present paper is to investigate how the definitions of Thurston's metric $L$ and $K$ on $\mathcal{T}_g^n$ could be adapted to the case of flat singular metrics. 

W.A. Veech already did something similar in \cite{Ve2} defining a complex-valued distance map $D_0$ on the Teichmuller space $\mathcal{TQ}_g^n(\underline k,\epsilon)$ of semi-translation structures on $S_g^n$.\\ 
We copy the definition of $D_0$ maintaining the original notation of Veech:
$$D_0(q_1,q_2):=\inf\limits_{\varphi\in Diff_0^+(S_g^n)}\alpha(\varphi^*q_1,q_2),$$
$$ \alpha(\varphi^*q_1,q_2):=\sup\limits_{x\in S_g^n}\left( \sup\limits_{(U_i,f_i)\in q_i, x\in U_1\cap U_2}\left(\limsup_{x'\rightarrow x}Log\left(\left(\frac{f_1(\varphi(x'))-f_1(\varphi(x))}{f_2(x')-f_2(x)}\right)^2\right)\right)\right),$$\\
where $q_i$, $i=1,2$ is regarded as a semi-translation structures and $f_i:U_i\rightarrow \mathbb{C}$, $U_i\subset S_g^n$, are natural charts of $q_i$. The map $Log$ is a branch of the complex logarithm.\\
The real part of $\alpha(\varphi^*q_1,q_2)$ is the Lipschitz constant of $\varphi$ computed with respect to the metrics $|q_1|$ and $|q_2|$ and consequently the real part of the distance function $D_0$ is asymmetric.\\
Veech claimed that the map $D_0$ is a complete pseudo-metric on $\mathcal{TQ}_g^n(\underline k,\epsilon)$ (the proof should be contained in unpublished preprints \cite{Ve3}).\\

We defined the pseudo-metric $L_F$ on $\mathcal{TQ}_g(\underline k,\epsilon)$ which is the symmetric analogue to the Thurston's metric:
$$L_F(q_1,q_2):=\inf\limits_{\varphi\in Diff^+_0(S_g,\Sigma)}\mathcal{L}_{q_1}^{q_2}(\varphi),$$$$ \mathcal{L}_{q_1}^{q_2}(\varphi):=\sup\limits_{p\in S_g\setminus \Sigma }\left(\sup\limits_{v\in T_pS_g,||v||_{q_1}=1}\left|\log(||d\varphi_pv||_{q_2})\right|\right).$$
One should notice that $L_F$ is different from the real part of Veech's distance function $D_0$.\\
A first, notable, inequality regarding $L_F$ is given by proposition 2.4: it results 
$$L_F(q_1,q_2)\ge d_{\mathcal{T}}(X_1,X_2),$$
where $X_1,X_2$ are the points in $\mathcal{T}_g$ corresponding to the conformal structures underlying the quadratic differentials.\\
The metric $L_F$ endows $\mathcal{TQ}_g(\underline k,\epsilon)$ with the structure of proper and complete space (propositions 2.8 and 2.10) and the standard topology of $\mathcal{TQ}_g(\underline k,\epsilon)$ (which is the one induced by its structure of complex manifold) is finer than the topology induced by $L_F$ (proposition 2.6). Furthermore, there is a metric $\mathbb{P}L_F$ on $\mathbb{P}\mathcal{TQ}_g(\underline k,\epsilon)$ induced by $L_F$, and the topology it induces is equal to the standard topology of the projectification of $\mathcal{TQ}_g(\underline k,\epsilon)$.\\
Motivated by Thurston's work, we defined another metric $K_F$ on $\mathcal{TQ}_g(\underline k,\epsilon)$ through ratios of lengths of saddle connections:
$$K_F(q_1,q_2):=\max\{K_F^a(q_1,q_2),K_F^a(q_2,q_1)\},$$
$$K^a_F(q_1,q_2):=\sup\limits_{\gamma\in SC(q_1)}\log\left(\frac{\hat l_{q_2}(\gamma)}{\hat l_{q_1}(\gamma)}\right),$$
where $SC(q_1)$ is the set of saddle connections of $q_1$ (geodesics for the flat metric meeting singular points only at their extremities), and $\hat l_{q_i}(\gamma)$ is the length of the geodesic representative for the metric $|q_i|$ of the homotopy class of $\gamma$ with fixed endpoints.\\
While it is possible to prove $L_F(q_1,q_2)=K_F(q_1,q_2)$ if $q_1$ and $q_2$ are on the same orbit of the action of $GL(2,\mathbb{R})^+$ (proposition 2.13), in the general case we were not able to adapt Thurston's proof of $L=K$. This is mainly because, as it is explained in the end of section 2, we believe it is not possible to find a flat analogue to the large class of geodesics of $L$ which Thurston uses in the proof of $L=K$.\\

In section 3 we introduced an asymmetric analogue $L_F^a$ to $L_F$ on $\mathcal{TQ}^{(1)}_g(\underline k,\epsilon)$ (which is the subset of $\mathcal{TQ}_g(\underline k,\epsilon)$ corresponding to surfaces of unitary area) defined as 
$$L^a_F(q_1,q_2):=\inf\limits_{\varphi\in \mathcal{D}} \log(Lip(\varphi)_{q_1}^{q_2}),$$
$$Lip(\varphi)_{q_1}^{q_2}=\sup\limits_{p\in S_g\setminus \Sigma}\left(\sup\limits_{v\in T_pS_g, ||v||_{q_1}=1} ||d\phi_pv||_{q_2}\right) ,$$
with $\mathcal{D}$ being the set of functions $\varphi:S_g\rightarrow S_g$ which are homotopic to the identity, differentiable almost everywhere and which fix the points of $\Sigma$. \\

We are able to reduce the proof of the equality of $L_F^a$ and $K_F^a$ on $\mathcal{TQ}_g^{(1)}(\underline{k},\epsilon)$ to the proof of two statements (corresponding to following theorem 1.1 and conjecture 1.2) about 1-Lischitz maps between planar polygons. In order to give the reader an idea of the reasonings involved, we briefly state them in a slightly simplified version.\\

Consider two planar polygons $\Delta$ and $\Delta'$ such that there is an injective function 
$$\iota:Vertices(\Delta)\rightarrow Vertices(\Delta')$$
which to every vertex $v$ of $\Delta$ associates a unique vertex $\iota(v)=v'$. Suppose both $\Delta$ and $\Delta'$ have exactly three vertices with strictly convex internal angle, which we denote $x_i$ and $x_i'$, $i=1,2,3$ respectively. \\
Suppose furthermore that for every $x,y\in Vertices(\Delta)$ it results
$$d_\Delta(x,y)\ge d_{\Delta'}(x',y'),$$
where $d_\Delta$ (resp. $d_{\Delta'}$) is the intrinsic Euclidean metric inside $\Delta$ (resp. $\Delta'$): $d_\Delta(x,y)$ (resp. $d_{\Delta'}(x',y')$) is defined as the infimum of the lengths, computed with respect to the Euclidean metric, of all paths from $x$ to $y$ (resp. from $x'$ to $y'$) entirely contained in $\Delta$ (resp. in $\Delta'$). \\
We say that vertices of $\Delta$ and of $\iota(Vertices(\Delta))$ are \textit{disposed in the same order} if it is possible to choose two parametrizations $\gamma:[0,1]\rightarrow \partial \Delta$ and $\gamma_1:[0,1]\rightarrow \partial \Delta'$ such that $\gamma(0)=x_1$, $\gamma_1(0)=x_1'$ and $\gamma,\gamma_1$ meet respectively vertices of $\Delta$ and of $\Delta'$ in the same order.

\begin{thm}
If $Vertices(\Delta)$ and $\iota(Vertices(\Delta))$ are disposed in the same order, then there is a 1-Lipschitz map $f:\Delta\rightarrow \Delta'$ (with respect to the intrinsic Euclidean metrics of the polygons) which sends vertices to corresponding vertices.
\end{thm}
\begin{conj}
If $Vertices(\Delta)$ and $\iota(Vertices(\Delta))$ are not disposed in the same order, then for every point $p\in \Delta$ there is a point $p'\in \Delta'$ such that $$d_\Delta(p,x_i)\ge d_{\Delta'}(p',x_i'), \quad i=1,2,3.$$
\end{conj}
We were able to prove theorem 1.1, which corresponds to theorem 3.21 of section 3, but not conjecture 1.2, which corresponds to conjecture 5.31 of section 3: we will still explain why we believe it must be true.\\
We proved the following theorem, which is the main result of this paper.
\begin{thm}
If conjecture 1.2 is true, then for every $q_1,q_2\in \mathcal{TQ}_g^{(1)}(\underline{k},\epsilon)$, it results
$$L_F^a(q_1,q_2)= K^a_F(q_1,q_2).$$
\end{thm}

Instead of following Thurston's approach, we proved theorem 1.3 adapting the idea of F. A. Valentine's proof (\cite{Va}) of Kirszbraun's theorem for $\mathbb{R}^2$.
\begin{thm} 
Let $S\subset \mathbb{R}^2$ be any subset and $f:S\rightarrow \mathbb{R}^2$ a 1-Lipschitz map.\\
Given any set $T$ which contains $S$, it is possible to extend $f$ to a 1-Lipschitz map $\hat f:T\rightarrow \mathbb{R}^2$ such that $\hat f(T)$ is contained in the convex hull of $f(S)$.
\end{thm}

\bigskip
\section{Symmetric pseudo-metrics $L_F$ and $K_F$}
\bigskip
\subsection{Teichm\"uller space of semi-translation surfaces} In this preliminary part we introduce semi-translation surfaces and their Teichm\"uller spaces, underlining some of their major properties.

\begin{defi}
A semi-translation surface is a closed topological surface $S_g$ of genus $g\ge 2$ endowed with a semi-translation structure, that is:
\begin{enumerate}[label=(\roman*)]
\item a finite set of points $\Sigma\subset S_g$ and an atlas of charts on $S_g\setminus \Sigma$ to $\mathbb{C}$ such that \ transition maps are of the form $z\mapsto \pm z+c$ with $c\in \mathbb{C}$,
\item a flat singular metric on $S_g$ such that for each point $p\in \Sigma$ there is a homeomorphism of a neighborhood of $p$ with a neighborhood of a cone angle of $\pi(k+2)$ for some $k>0$, which is an isometry away from $p$ (we call such point a singular point of order $k$). Furthermore, charts of the atlas of $(i)$ are isometries for the flat singular metric.
\end{enumerate}
\end{defi}

Equivalently, a semi-translation surface can be defined as a closed Riemann surface $X$ endowed with a non-vanishing holomorphic quadratic differential $q$. Indeed, natural coordinates for $q$ and the metric $|q|$ endow $S_g$ with a semi-translation structure. Conversely, given a semi-translation structure one can obtain a quadratic differential by setting $q=dz^2$ on $S_g\setminus \Sigma$ (where $z$ is a coordinate of the charts of the semi-translation structure) and $q=z^{k}dz^2$ in a neighborhood of a singular point. It is clear then that the sum of the orders of singular points is $4g-4$.\\
A semi-translation surface is naturally endowed with a locally $Cat(0)$ metric. Actually, one can extend the definition to allow the quadratic differentials to have at most simple poles (and consequently cone angles of $\pi$), but then the resulting metric will not be locally $Cat(0)$ anymore.\\

The flat singular metric $|q|$ can be nicely characterized (see \cite{St}) stating that its local geodesics are continuous maps $\gamma:\mathbb{R}\rightarrow S_g$ such that for every $t\in \mathbb{R}$:
\begin{itemize}
\item if $\gamma(t)\notin\Sigma$, then there is a neighborhood $U$ of $t$ in $\mathbb{R}$ such that $\gamma|_U$ is an Euclidean  segment,
\item if $\gamma(t)\in\Sigma$, then there is a small neighborhood $V$ of $\gamma(t)$ in $S_g$ and an $\epsilon>0$ small enough such that the angles defined by $\gamma([t,t+\epsilon))$ and $\gamma((t-\epsilon,t])$ in $V$ are both at least $\pi$.
\end{itemize}
We say that a \textit{saddle connection} on $(X,q)$ is a geodesic for the flat metric going from a singularity to a singularity, without any singularities in the interior of the segment. \\
Since the metric $|q|$ is locally $Cat(0)$, for any arc $\gamma$ with endpoints in $\Sigma$ there always is a unique geodesic representative in the homotopy class of $\gamma$ with fixed endpoints. This geodesic representative is a concatenation of saddle connections.\\
Finally, we define the $systole$ of a semi-translation surface $(X,q)$, and denote it with $sys(q)$, to be the length of the shortest saddle connection.\\

Given any $g\ge 2$ and $m\ge 1$, fix a finite set of points $\Sigma=\{p_1,\dots,p_m\}\subset S_g$ and an $m$-ple $\underline{k}=(k_1,k_2,\dots,k_m)\in \mathbb{N}^m$ such that $\sum_{l=1}^mk_l=4g-4$.\\
We denote by $\mathcal{S}\Omega(\underline k,\epsilon,\Sigma)$ the set of semi-translation surfaces on $S_g$ with singularities prescribed by $\underline k$ on the points of $\Sigma$ (i.e. it has a zero of order $k_i$ on $p_i$, $i=1,\dots,m$) and holonomy determined by $\epsilon\in \{\pm 1\}$ ($\epsilon=1$ in case of trivial holonomy and $\epsilon=-1$ otherwise). \\
Consider the group $Diff^+(S_g,\Sigma)$ of diffeomorphisms of $S_g$ which fix the points of $\Sigma$ and its subgroup $Diff^+_0(S_g,\Sigma)$ consisting of diffeomorphisms homotopic to the identity.\\
We define the Teichm\"uller and moduli space of semi-translation surfaces with singularities prescribed by $\underline k$ on $\Sigma$ and holonomy defined by $\epsilon$ in the following way:
$$\mathcal{TQ}(\underline k,\epsilon,\Sigma):=\mathcal{S}\Omega(\underline k,\epsilon,\Sigma)/Diff^+_0(S_g,\Sigma),\quad \mathcal{Q}(\underline k,\epsilon,\Sigma):=\mathcal{S}\Omega(\underline k,\epsilon,\Sigma)/Diff^+(S_g,\Sigma),$$
where the two groups of diffeomorphisms act by pullback.\\
We will denote $\mathcal{TQ}_g(\underline k,\epsilon,\Sigma)$ and $\mathcal{Q}(\underline k,\epsilon,\Sigma)$ simply as $\mathcal{TQ}_g(\underline k,\epsilon)$ and $\mathcal{Q}(\underline k,\epsilon,\Sigma)$ in order lighten the notation: one should keep in mind that in the definition is implicit the choice of $\Sigma$.\\
Furthermore, we denote simply by $q$ an element of $\mathcal{TQ}_g(\underline k,\epsilon)$ and $\mathcal{Q}(\underline k,\epsilon,\Sigma)$: the fact that it is an equivalence class will be clear from the context.\\
As it is explained in the following theorem, the spaces $\mathcal{TQ}_g(\underline k,\Sigma)$ have a nice structure of complex manifold.
\begin{thm}
Each space $\mathcal{TQ}_g(\underline k,1)$ has the structure of a complex manifold of dimension $2g+m-1$, while $\mathcal{TQ}_g(\underline k,-1)$ has the structure of a complex manifold of dimension $2g+m-2$.

\end{thm}
Unfortunately, the spaces $\mathcal{Q}(\underline k,\epsilon)$ have only the structure of complex orbifolds of the same dimension of $\mathcal{TQ}(\underline k,\epsilon)$.\\
There is a natural action of $GL(2,\mathbb{R})^+$ on $\mathcal{TQ}_g(\underline k,\epsilon)$ and $\mathcal{Q}_g(\underline k,\epsilon)$: for each \mbox{$A\in GL(2,\mathbb{R})^+$} and each quadratic differential $q$, the element $A\cdot q$ is the quadratic differential obtained post-composing the natural charts of $q$ with $A$. 

\bigskip
\subsection{Definitions of flat Thurston's metrics}

Fix any genus $g\ge 2$ and consider the Teichm\"uller space of semi-translation surfaces $\mathcal{TQ}_g(\underline k,\epsilon)$ with singularities on $\Sigma\subset S_g$ prescribed by the $m$-ple $\underline k=(k_1,\dots, k_m)\in\mathbb{N}^m$ such that $\sum_{i=1}^mk_i=4g-4$ and holonomy determined by $\epsilon\in\{+1,-1\}$.  \\
We will introduce now all flat analogues to Thurston's metrics.\\
First we define the following function $L_F$, which is a symmetric analogue to Thurston's metric $L$. 
$$L_F:\mathcal{TQ}_g(\underline k,\epsilon)\times \mathcal{TQ}_g(\underline k,\epsilon)\rightarrow \mathbb{R}, $$
$$L_F(q_1,q_2):=\inf\limits_{\varphi\in Diff^+_0(S_g,\Sigma)}\mathcal{L}_{q_1}^{q_2}(\varphi),$$ 
$$\mathcal{L}_{q_1}^{q_2}(\varphi):=\sup\limits_{p\in S_g\setminus \Sigma }\left(\sup\limits_{v\in T_pS_g,||v||_{q_1}=1}\left|\log(||d\varphi_pv||_{q_2})\right|\right).$$\\
The quantity $\mathcal{L}_{q_1}^{q_2}(\varphi)$ can be rewritten as
$$\mathcal{L}_{q_1}^{q_2}(\varphi)=\max\{\log(Lip_{q_1}^{q_2}(\varphi)),-\log(lip_{q_1}^{q_2}(\varphi))\},$$  
with $Lip_{q_1}^{q_2}(\varphi)$ being the \textit{upper Lipschitz constant} of $\varphi$:
$$Lip_{q_1}^{q_2}(\varphi):=\sup\limits_{p\in S_g\setminus \Sigma}\left(\sup\limits_{v\in T_pS_g,||v||_{q_1}=1}||d\varphi_pv||_{q_2}\right)$$
and $lip_{q_1}^{q_2}(\varphi)$ being the \textit{lower Lipschitz constant} of $\varphi$:
$$lip_{q_1}^{q_2}(\varphi):=\inf\limits_{p\in S_g\setminus \Sigma}\left(\inf\limits_{w\in T_pS_g,||w||_{q_1}=1}||d\varphi_pw||_{q_2}\right).$$
\\
We define also an asymmetric analogue to $L$ on $\mathcal{TQ}_g^{(1)}(\underline{k},\epsilon)$ 

$$L^a_F:\mathcal{TQ}_g^{(1)}(\underline{k},\epsilon)\times \mathcal{TQ}_g^{(1)}(\underline{k},\epsilon)\rightarrow \mathbb{R}$$
associating to any pair $q_1,q_2 \in\mathcal{TQ}_g^{(1)}(\underline{k},\epsilon)$ of semi-translation surfaces of unitary area the quantity

$$L^a_F(q_1,q_2):=\inf\limits_{\varphi\in \mathcal{D}} \log(Lip(\varphi)_{q_1}^{q_2}),$$
$$Lip(\varphi)_{q_1}^{q_2}=\sup\limits_{p\in S_g\setminus \Sigma}\left(\sup\limits_{v\in T_pS_g, ||v||_{q_1}=1} ||d\varphi_pv||_{q_2}\right),$$
where $\mathcal{D}$ is the set of functions $\varphi:S_g\rightarrow S_g$ which are homotopic to the identity, differentiable almost everywhere and which fix the points of $\Sigma$. \\
Since $Diff_0^+(S_g,\Sigma)\subset \mathcal{D}$, one can immediately deduce $L_F(q_1,q_2)\ge L_F^a(q_1,q_2)$ for every $q_1,q_2 \in\mathcal{TQ}_g^{(1)}(\underline{k},\epsilon)$.\\

We define two flat counterparts to the metric $K$, which are $K_F^a$ and $K_F$. The first one is asymmetric and the second one is its symmetrization.\\
In particular, for every $q_1,q_2\in \mathcal{TQ}_g(\underline k,\epsilon)$, we set
$$K^a_F(q_1,q_2):=\sup\limits_{\gamma\in SC(q_1)}\log\left(\frac{\hat l_{q_2}(\gamma)}{\hat l_{q_1}(\gamma)}\right),$$
where $SC(q_1)$ is the set of saddle connections of $q_1$, and $\hat l_{q_i}(\gamma)$ is the length of the geodesic representative for $|q_i|$ in the homotopy class of $\gamma$ with fixed endpoints. \\Finally the symmetric analogue to $K$ is defined as
$$K_F(q_1,q_2):=\max\{K_F^a(q_1,q_2),K_F^a(q_2,q_1)\}$$
for every $q_1,q_2\in  \mathcal{TQ}_g(\underline k,\epsilon)$.\\

In this section we will study the properties of $L_F$ and $K_F$ and explain the difficulties in trying to prove $L_F=K_F$ on $\mathcal{TQ}_g(\underline k,\epsilon)$.\\
These difficulties can be solved considering $L_F^a$ instead of $L_F$: the fact that $L_F^a$ is asymmetric and the infimum is taken over functions in $\mathcal{D}$ will play a crucial role. \\
Indeed, $L_F^a$ is defined specifically to get $L_F^a=K_F^a$: the next section will be completely devoted to the proof of such equality.\\

We now begin the study of the properties of $L_F$. 

\begin{prop}
The function $L_F$ is a symmetric pseudo-metric on $\mathcal{TQ}_g(\underline k,\epsilon)$.\\
\end{prop}
\begin{proof}
It is clear that $L(q,q)=|\log(Lip_q^q(Id))|=0$ for all $q\in\mathcal{TQ}_g(\underline k,\epsilon)$.\\
The equality
$$lip_{q_1}^{q_2}(\varphi)=\frac{1}{Lip_{q_2}^{q_1}(\varphi^{-1})}$$
grants 
$$\mathcal{L}_{q_1}^{q_2}(\varphi)=\mathcal{L}_{q_2}^{q_1}(\varphi^{-1})$$
and thus the symmetry of $L_F$.\\
The triangular inequality follows from the inequality 
$$\mathcal{L}_{q_1}^{q_3}(\varphi\circ \psi)\le \mathcal{L}_{q_2}^{q_3}(\varphi)+\mathcal{L}_{q_1}^{q_2}(\psi).$$
Finally, one could easily note that, given any $q_1\in \mathcal{TQ}_g(\underline k,\epsilon)$, it results \mbox{$L_F(q_1,q_2)=0$} exactly for all $q_2\in \mathcal{TQ}_g(\underline k,\epsilon)$ such that $q_2=e^{i\theta}q_1$.
\end{proof}

Since it results $L_F(q_1,q_2)=0$ if and only if $q_1$ and $q_2$ are in the same orbit of the action of the unitary group $U(1)\subset \mathbb{C}^*$, it follows that $L_F$ can be considered as a metric on the space of flat singular metrics with singularities prescribed by $\underline k$ and holonomy prescribed by $\epsilon$.\\
For the same reason, $L_F$ descends to a metric $\mathbb{P} L_F$ on the projectivization $\mathbb{P}\mathcal{TQ}_g(\underline k,\epsilon)=\mathcal{TQ}_g(\underline k,\epsilon)/\mathbb{C}^*=\mathcal{TQ}^{(1)}_g(\underline k,\epsilon)/U(1)$ by setting
$$\mathbb{P} L_F([q_1],[q_2]):=L_F\left(\frac{q_1}{Area(q_1)},\frac{q_2}{Area(q_2)}\right).$$

The first result we present on the pseudo-metric $L_F$ is an inequality concerning the Teichm\"uller metric $d_{\mathcal{T}}$.

\begin{prop}\label{tecl}
For any $q_1,q_2\in \mathcal{TQ}_g(\underline k,\epsilon)$, denote by $X_1,X_2\in \mathcal{T}_g$ the points in the Teichm\"uller space relative to the corresponding conformal structure. 
It results:
$$L_F(q_1,q_2)\ge d_{\mathcal{T}}(X_1,X_2)$$
In case there is a Teichm\"uller map between $X_1$ and $X2$ with respect to the differentials $q_1$ and $q_2$ the last inequality is an equality.
\end{prop}
\begin{proof}
For every $\varphi\in Diff^+_0(S_g,\Sigma)$ and $p\in S_g\setminus \Sigma$ we define the quantities 
$$Lip_{q_1}^{q_2}(\varphi)_p:=\sup\limits_{v\in T_pS_g,||v||_{q_1}=1}||d\varphi_pv||_{q_2},$$
$$lip_{q_1}^{q_2}(\varphi)_p:=\inf\limits_{w\in T_pS_g,||w||_{q_1}=1}||d\varphi_pw||_{q_2}.$$
Then, since the global dilatation $K(\varphi)$ is independent of the holomorphic charts and thus can be computed in the natural coordinates respectively of $q_1$ and $q_2$, we get the inequality
$$K(\varphi)=\sup\limits_{p\in S_g\setminus \Sigma}\frac{Lip_{q_1}^{q_2}(\varphi)_p}{lip_{q_1}^{q_2}(\varphi)_p}\le \frac{Lip_{q_1}^{q_2}(\varphi)}{lip_{q_1}^{q_2}(\varphi)}.$$
Since for every $\varphi\in Diff^+_0(S_g,\Sigma)$ it also results
$$K(h)\le K(\varphi),$$
where $K(h)$ is the global dilatation of a Teichm\"uller map $h$ such that $d_{\mathcal{T}}(X_1,X_2)=\frac 1 2\log(K(h))$, combining the last two inequalities we get that it can not be at the same time
$$Lip_{q_1}^{q_2}(\varphi)<\sqrt{K(h)}\text{ and } lip_{q_1}^{q_2}(\varphi)>\frac{1}{\sqrt{K(h)}}$$
and this implies the inequality $L(q_1,q_2)\ge d_{\mathcal{T}}(X_1,X_2)$.\\
Finally, in case $h$ is a Teichm\"uller map with respect to the quadratic differentials $q_1$ and $q_2$, then, since $h$ can be written in local coordinates as
$$h(x+iy)=\sqrt{K(h)}x+\frac{i}{\sqrt{K(h)}}y$$
it follows 
$$Lip_{q_1}^{q_2}(h)=\sqrt{K(h)},\quad lip_{q_1}^{q_2}(h)=\frac{1}{\sqrt{K(h)}}$$
and thus the equality of the claim.\\ 
\end{proof}

\begin{oss}
Notice that in the proof of proposition 2.4 the fact that the metric induced by the quadratic differential is locally $Cat(0)$ is never used. 
For this reason, one could allow the quadratic differentials to have simple poles on the marked points and define $L_F$ in the same way.\\
Then the same inequality $L_F(q_1,q_2)\ge d_{\mathcal{T}}(X_1,X_2)$ will be true for $X_1,X_2\in \mathcal{T}_g^n$.
\end{oss}
\bigskip
\subsection{Induced topology of $L_F$}

We define \textit{standard topology} on $\mathcal{TQ}_g(\underline k,\epsilon)$, and denote it by $\mathbb{T}_{std}$, the topology induced by the structure of complex manifold, that is, the topology induced by the period maps. Given a sequence $\{q_n\}_{n\in\mathbb{N}}\subset \mathcal{TQ}_g(\underline k,\epsilon)$, we write $q_n\rightarrow q$ to denote its convergence to $q\in \mathcal{TQ}_g(\underline k,\epsilon)$ with respect to the standard topology.\\ 
Similarly, we denote by $\mathbb{T}_{L_F}$ the topology on $\mathcal{TQ}_g(\underline k,\epsilon)$ induced by $L_F$.

\begin{prop}\label{finer}
The topology $\mathbb{T}_{std}$ is finer than $\mathbb{T}_{L_F}$.
\end{prop}

We will prove the equivalent claim that for every sequence \mbox{$\{q_n\}_{n\in\mathbb{N}}\subset \mathcal{TQ}_g(\underline k,\epsilon)$} the convergence $q_n\rightarrow q$ implies $\lim\limits_{n\rightarrow \infty}L_F(q_n,q)=0$.\\

To this end, we need to first make an observation concerning Euclidean triangles.\\ 
Denote by $\Xi$ the set of non-degenerate Euclidean triangles $T\subset \mathbb{R}^2$ with one vertex in the origin of $\mathbb{R}^2$: since every triangle $T\in \Xi$ can be identified by the coordinates of its two vertices different from the origin, $\Xi$ can be considered as a subset of $\mathbb{R}^4$.\\
Given any sequence $\{T_n\}_{n\in\mathbb{N}}$ in $\Xi$, we say that it converges to $T\in \Xi$, and write $T_n\rightarrow T$, if $\{T_n\}_{n\in\mathbb{N}}$ converges to $T$ as a sequence of $\mathbb{R}^4$ with respect to the standard Euclidean metric. For every $n\in\mathbb{N}$ consider the affine map $A_n$ which sends $T_n$ to $T$ and denote by $\sigma_{1}(A_n),\sigma_{2}(A_n)$ its eigenvalues. It is easy to verify that if $T_n\rightarrow T$ then $\lim\limits_{n\rightarrow \infty}\sigma_1(A_n)=1$, $\lim\limits_{n\rightarrow \infty}\sigma_2(A_n)=1$.
\begin{proof}
In order to prove the proposition, given any sequence $\{q_n\}_{n\in\mathbb{N}}\subset \mathcal{TQ}_g(\underline k,\epsilon)$ such that $q_n\rightarrow q$, we will find a sequence of maps $A_n\in Diff^+_0(S_g,\Sigma)$ with the property $\mathcal{L}_{q}^{q_n}(A_n)\rightarrow 0$. The claim then will follow from the inequality \mbox{$L(q_n,q)\le \mathcal{L}_{q}^{q_n}(A_n)$}.\\
If $q_n\rightarrow q$ then one could find a collection of arcs $\Gamma=\{\gamma_j\}_{j=1}^{3(m+2g-2)}$ with endpoints in $\Sigma$ which triangulate $S_g$ and an $n_0>0$ such that the geodesic representative of the homotopy class of every $\gamma_j$ for $|q|$ and $|q_n|$, $n>n_0$, is a saddle connection.\\
The geodesic representatives of the homotopy classes of the arcs in $\Gamma$ for $|q|$ (resp. $|q_n|$), provide us of a set of Euclidean triangles $\Xi_q=\{T_l\}_{l=1}^{2(k+2g-2)}$ (resp $\Xi_{q_n}=\{T_l^n\}_{l=1}^{2(k+2g-2)}$) which cover $S_g$. Using period coordinates of $\mathcal{TQ}_g(\underline k,\epsilon)$ one can indeed observe that $q_n\rightarrow q$ implies that every triangle $T^n_l$ converges to $T_l$ in the sense explained in the observation preceding this proof.\\
For every $n\in\mathbb{N}$, we define by $A_n\in Diff^+_0(S_g,\Sigma)$ the map which is piecewise affine in natural coordinates respectively of $q_n$ and $q$, and which on every triangle $T_l^n$ of $\Xi_{q_n}$ is the affine map $A_n^l$ which sends $T_l^n$ to the corresponding triangle $T_l$ of $\Xi_{q}$. \\
As before, we denote by $\sigma_1(A_n^l),\sigma_2(A_n^l)$ the eigenvalues of $A_n^l$. Since it results
$$Lip_{q}^{q_n}(A_n^l)=\max\limits_{l=1,\dots ,2(m+2g-2)}\left(\max\{\sigma_1(A_n^l),\sigma_2(A_n^l)\}\right)$$
$$lip_{q}^{q_n}(A_n^l)=\min\limits_{l=1,\dots ,2(m+2g-2)}\left(\min\{\sigma_1(A_n^l),\sigma_2(A_n^l)\}\right)$$
the claim of the proposition follows from the preceding observation about Euclidean triangles.
\end{proof}

From proposition \ref{finer}, it follows that compact sets of $\mathbb{T}_{std}$ are also compact sets of $\mathbb{T}_{L_F}$. It is thus useful to characterize them in a way which is similar to the statement of Mumford's compactness criterion.\\
Before doing so, let us fix once and for all some notation: for any arc $\gamma$ in $S_g$ with endpoints in $\Sigma$ and any quadratic differential $q\in\mathcal{TQ}_g(\underline k,\epsilon)$, we denote by $l_q(\gamma)$ the length of $\gamma$ with respect to the metric $|q|$ and by $\hat l_q(\gamma)$ the length of the geodesic representative for $|q|$ in the homotopy class of $\gamma$ with fixed endpoints.\\

The following proposition about compact sets of $\mathbb{T}_{std}$ is a consequence of proposition 1, section 3, of \cite{KMS}, which establishes the compactness of subsets of quadratic differentials with lower bound on the area.
\begin{prop}\label{compa}
Fix $\epsilon,L>0$ and a collection of arcs $\Gamma=\{\gamma_i\}_{i=1}^{3(m+2g-2)}$ with endpoints in $\Sigma$ which triangulates $S_g$.\\
Define the subset $\mathcal{K}_{\epsilon,L}\subset \mathcal{TQ}_g(\underline k,\epsilon)$ as the set of quadratic differentials $q$ which satisfy the following two conditions.
\begin{enumerate}[label=(\roman*)]
\item $sys(q)\ge \epsilon$,
\item $\sum\limits_{i=1}\limits^{3(m+2g-2)}\hat l_{q}(\gamma_i)\le L$.
\end{enumerate}
The set $\mathcal{K}_{\epsilon,L}$ is a compact set of $\mathbb{T}_{std}$.
\end{prop}

Using this characterization of compact sets we can prove the following proposition.

\begin{prop}\label{proper}
Each Teichm\"uller space $\mathcal{TQ}(\underline k,\epsilon)$ endowed with the pseudo-metric $L_F$ is a proper topological space.
\end{prop}

\begin{proof}
We prove that closed balls $B_{L_F}^R(q)$ of $L_F$,
$$B_{L_F}^R(q):=\{q'\in \mathcal{TQ}_g(\underline k,\epsilon)|L_F(q,q')\le R\}$$
are contained in a compact subset of $\mathbb{T}_{std}$: thanks to the result of proposition \ref{finer} they will be contained also in a compact set of $\mathbb{T}_{L_F}$.\\
Let $\gamma$ be any geodesic arc for $|q|$ with endpoints in $\Sigma$, and $\gamma_n$ the geodesic representative of its homotopy class for the metric $|q_n|$. Then it follows
$$\frac{l_q(\gamma)}{l_{q_n}(\gamma_n)}\le \frac{l_q(\varphi(\gamma_n))}{l_{q_n}(\gamma_n)}\le \sup\limits_{p\in S_g\setminus \Sigma}\left(\sup_{v\in T_pS_g,||v||_{q_n}=1}||d\varphi_pv||_q\right),$$ 

$$\frac{l_{q_n}(\gamma_n)}{l_{q}(\gamma)}\le \frac{l_{q_n}(\varphi(\gamma))}{l_{q}(\gamma)}\le \sup\limits_{p\in S_g\setminus \Sigma}\left(\sup_{v\in T_pS_g,||v||_{q}=1}||d\varphi_pv||_{q_n}\right)$$ 
and from the fact that $L_F(q,q_n)$ is bounded it follows that it can not happen $$\lim\limits_{n\rightarrow \infty}\hat l_{q_n}(\gamma)= 0 \text{ or } \lim\limits_{n\rightarrow \infty}\hat l_{q_n}(\gamma)= \infty.$$
\end{proof}

By abuse of notation we will denote by $\mathbb{T}_{std}$ and $\mathbb{T}_{L_F}$ the induced topologies on $\mathbb{P}\mathcal{TQ}_g(\underline k,\epsilon)$. 

\begin{prop}
$\mathbb{T}_{std}$ and $\mathbb{T}_{L_F}$ are the same topology on $\mathbb{P}\mathcal{TQ}_g(\underline k,\epsilon)$.
\end{prop}
\begin{proof}
It will be sufficient to prove that $\mathbb{T}_{L_F}$ is finer than $\mathbb{T}_{std}$ and thus that for every sequence \mbox{$\{q_n\}_{n\in\mathbb{N}}\subset \mathcal{TQ}^{(1)}_g(\underline k,\epsilon)$} such that $\lim\limits_{n\rightarrow \infty}L_F(q_n,q)= 0$ it follows that there exists $c\in U(1)$ with the property $q_n\rightarrow cq$.\\
Since $\lim\limits_{n\rightarrow \infty}L_F(q_n,q)= 0$ it follows that $\{q_n\}_{n\in\mathbb{N}}$ is contained in a closed ball of $L_F$ and thus in a compact set. Up to passing to a subsequence we can state that there is $q'\in \mathcal{TQ}_g(\underline k,\epsilon)$ such that $q_n\rightarrow q'$.
Since 
$$L_F(q,q')\le L_F(q,q_n)+L_F(q_n,q')$$
it follows $q'=e^{i\theta}q$.
\end{proof}

In the following theorem we establish another similarity between $L_F$ and Thurston's asymmetric metric $L$: $L_F$ is a complete pseudo-metric.\\
The notion of completeness makes sense also for pseudo-metrics: a pseudo-metric $d$ on a topological space $X$ is complete if every Cauchy sequence for $d$ admits at least one limit point for $d$. Thus in the proof of the following theorem we will prove that every Cauchy sequence for $L_F$ admits at least one limit point for $L_F$.

\begin{thm}
Every Teichm\"uller space $\mathcal{TQ}_g(\underline k,\epsilon)$ and its quotient $\mathbb{P}\mathcal{TQ}_g(\underline k,\epsilon)$, endowed respectively with the metrics $L_F$ and $\mathbb{P} L_F$, are complete pseudo-metric spaces.
\end{thm}

\begin{proof}
We prove that any Cauchy sequence $\{q_n\}_{n\in\mathbb{N}}$ for $L_F$ on $\mathcal{TQ}_g(\underline k,\epsilon)$ is contained in a compact set of $\mathbb{T}_{std}$: from proposition 2.6 it will follow that $\{q_n\}_{n\in \mathbb{N}}$ is contained in a compact set of $\mathbb{T}_{L_F}$ and therefore is convergent. We will use the same inequalities of the proof of proposition \ref{proper}.\\ 
Consider any Cauchy sequence $\{q_n\}_{n\in\mathbb{N}}$ for $L_F$ on $\mathcal{TQ}_g(\underline k,\epsilon)$, given any arc $\gamma$ on $S_g$ with endpoints in $\Sigma$ denote by $\gamma_{n}$ the geodesic representative of the homotopy class of $\gamma$ for the metric $|q_n|$. Then for every $\varphi\in Diff_0^+(S_g,\Sigma)$ it results:
$$\frac{ l_{q_m}(\gamma_m)}{l_{q_n}(\gamma_n)}\le \frac{l_{q_m}(\varphi(\gamma_n))}{l_{q_n}(\gamma_n)}\le Lip_{q_n}^{q_m}(\varphi)$$ 
and thus the sequence $\{\log(l_{q_n}(\gamma_n))\}_{n\in\mathbb{N}}$ is a Cauchy sequence and consequently bounded: this means that $\{q_n\}_{n\in \mathbb{N}}$ is contained in a set of the form described in proposition 2.7.\\
The completeness of $(\mathbb{P}\mathcal{TQ}_g(\underline k,\epsilon),\mathbb{P} L_F)$ follows from the same reasoning considering a Cauchy sequence $\{q_n\}_{n\in\mathbb{N}}\subset \mathcal{TQ}^{(1)}_g(\underline k,\epsilon)$.
\end{proof}

Finally, it is worth mentioning that the mapping class group $\Gamma(S_g,\Sigma)$ acts on $\mathcal{TQ}_g(\underline k,\epsilon)$ by isometries of $L_F$: in particular for every $q_1,q_2\in \mathcal{TQ}_g(\underline k,\epsilon)$ and \mbox{$\psi\in \Gamma(S_g,\Sigma)$} it results
$$L_F(q_1,q_2)=L_F(\psi\cdot q_1,\psi\cdot q_2),$$
where $\psi\cdot q$ is the pullback by $\psi^{-1}$ of the quadratic differential $q$. This result follows from the equality 
$$\mathcal{L}_{q_1}^{q_2}(\varphi)=\mathcal{L}_{\psi\cdot q_1}^{\psi\cdot q_2}(\psi\circ\varphi\circ \psi^{-1})$$
for every $\varphi\in Diff_0^+(S_g,\Sigma)$ and the fact that the conjugation of $Diff_0^+(S_g,\Sigma)$ by any element of $\Gamma(S_g,\Sigma)$ is an isomorphism of $Diff_0^+(S_g,\Sigma)$.\\
Since the action of the mapping class group $\Gamma(S_g,\Sigma)$ on $\mathcal{TQ}_g(\underline k,\epsilon)$ is also properly discontinuous, one gets that the metric $L_F$ descends also to a metric $\hat L_F$ on $\mathcal{Q}(\underline k,\epsilon)$,
$$\hat L_F(\hat q_1,\hat q_2)=\inf L_F(q_1,q_2),$$
where the infimum is taken over all liftings $q_1,q_2$ to $\mathcal{TQ}_g(\underline k,\epsilon)$ of $\hat q_1,\hat q_2\in \mathcal{Q}_g(\underline k,\epsilon)$.
\begin{prop}
The space $\mathcal{Q}_g(\underline k,\epsilon)$ endowed with the metric $\hat L_F$ is a complete pseudo-metric space.
\end{prop}
\begin{proof}
The proof is identical to the one of proposition 2.10.
\end{proof}

\bigskip
\subsection{Properties of the pseudo-metric $K_F$}


A first analogy with the metric $L_F$ is given by the fact that $K_F$ has all the properties we just proved for $L_F$ and in particular it follows:
\begin{thm}
The function $K_F$ is a complete and proper symmetric pseudo-metric on $\mathcal{TQ}_g(\underline k, \epsilon)$. 
\end{thm}
\begin{proof}
All the previous proofs for $L_F$ adapt to $K_F$ (in particular, the fact that $q_n\rightarrow q$ implies $K_F(q_n,q)\rightarrow q$ is a direct consequence of the definition of period maps), except for 
$$K_F(q_1,q_2)=0 \text{ if and only if } q_1=e^{i\theta}q_2$$
which can be proved as we now explain.\\
If $q_1=e^{i\theta}q_2$ then $q_1$ and $q_2$ induce the same flat metric on $S_g$ and consequently \mbox{$K_F(q_1,q_2)=0$}, so let us prove the other implication.\\
If $K_F(q_1,q_2)=0$, then consider any saddle connection $\sigma$ of $q_1$ and let $\tau$ be the geodesic representative for $|q_2|$ in the homotopy class of $\sigma$. The curve $\tau$ is a concatenation of saddle connections $\tau_1,\dots,\tau_k$ of $q_2$ and since $K_F^a(q_1,q_2)\le 0$ it results
$$l_{q_1}(\sigma)\ge l_{q_2}(\tau_1)+\dots l_{q_2}(\tau_k).$$
For each $i=1,\dots,k$ let $\sigma_i$ be the geodesic representative for the metric $|q_1|$ in the homotopy class of $\tau_i$. Since $K_F^a(q_2,q_1) \le 0$ it follows
$$l_{q_2}(\tau_1)+\dots l_{q_2}(\tau_k)\ge l_{q_1}(\sigma_1)+\dots+l_{q_1}(\sigma_k)$$
and, since the concatenation $\sigma_1*\dots *\sigma_k$ is in the same homotopy class of $\sigma$, it also results
$$l_{q_1}(\sigma_1)+\dots+l_{q_1}(\sigma_k)\ge l_{q_1}(\sigma).$$
These inequalities can be realized at the same time only if they are equalities, and since $\sigma$ is the only geodesic representative in its homotopy class it follows that $\tau$ must be a saddle connection of $q_2$: we have thus proved that if $K_F(q_1,q_2)=0$ then the geodesic representative for $|q_2|$ (resp. for $|q_1|$) of any saddle connection of $q_1$ (resp. of $q_2$) must be a saddle connection of the same length.\\
At this point the claim is basically already proved, since $q_1$ and $q_2$ give triangulations of $S_g$ by saddle connections of the same length.

\end{proof}

We can define on $\mathbb{P}\mathcal{TQ}_g(\underline k,\epsilon)$ the metric $\mathbb{P} K_F$ in the same way we defined $\mathbb{P} L_F$ and prove that its induced topology $\mathbb{T}_{K_F}$ coincides with the standard topology $\mathbb{T}_{std}$.\\

As for the metrics $L$ and $K$ on $\mathcal{T}_g$, the inequality 
$$L_F(q_1,q_2)\ge K_F(q_1,q_2), \quad \forall q_1,q_2\in\mathcal{TQ}_g(\underline k,\epsilon)$$
is straightforward, while proving the inverse inequality is a much harder problem, which could be solved finding a function $\varphi\in Diff_0^+(S_g,\Sigma)$ such that 
$$\mathcal{L}_{q_1}^{q_2}(\varphi)\le K_F(q_1,q_2).$$
Before studying the general case, let us first state a much simpler fact.
\begin{prop}
Given any $q\in\mathcal{TQ}_g(\underline k,\epsilon)$ and any $A\in GL(2,\mathbb{R})^+$, it results
$$L_F(q,A\cdot q)=K_F(q,A\cdot q)=\log(\sigma),$$
where $\sigma:=\max\{\sigma_1(A),\sigma_1(A)^{-1},\sigma_2(A),\sigma_2(A)^{-1}\}$ and $\sigma_1(A), \sigma_2(A)$ are the two eigenvalues of $A$.
\end{prop}
\begin{proof}
Without loss of generality, we can suppose $\sigma_1(A)$ is realized in the horizontal direction of $q$ and $\sigma_2(A)$ in the vertical direction. Notice furthermore that it results 
$$\log(\sigma)=\mathcal{L}_{q}^{A\cdot q}(Id).$$
If $\sigma=\sigma_1(A)$ or $\sigma=\sigma_1(A)^{-1}$, then a saddle connection in the horizontal direction will have stretch factor $\sigma$: although it is not always possible to suppose the existence of such geodesic, it is a consequence of theorem 2 of \cite{Ma2} that the directions of saddle connections of a quadratic differential are dense in $S^1$. Consequently, we can always consider a sequence $\{\gamma_n\}_{n\in \mathbb{N}}$ of saddle connections of $q$ asymptotic in the horizontal direction: this means that it results $\lim\limits_{n\rightarrow \infty}\theta(\gamma_n)=0$, where $\theta(\gamma_n)$ is the difference between the direction of $\gamma_n$ and the horizontal direction.\\
Then it follows
$$K_F(q,A\cdot q)\ge \lim\limits_{n\rightarrow \infty}\left|\log\left(\frac{\hat l_{A\cdot q}(\gamma_n)}{\hat l_{q}(\gamma_n)}\right)\right|=\log(\sigma)\ge L_F(q,A\cdot q)$$
and from $K_F(q,A\cdot q)\le L_F(q,A\cdot q)$ one gets $K_F(q,A\cdot q)= L_F(q,A\cdot q)=\log (\sigma)$.
If $\sigma=\sigma_2(A)$ or $\sigma=\sigma_2(A)^{-1}$, one can repeat the same reasoning for the vertical direction.
\end{proof}

Considering the general case, one could be tempted to adapt the ideas behind Thurston's proof in \cite{Th} to the case of $L_F$ and $K_F$. Specifically, one could try to build a flat analogue to Thurston's stretch maps. \\

We thought the more natural approach to try to do so was to triangulate $S_g$ by saddle connections: clearly this could work only locally on $\mathcal{TQ}_g(\underline k,\epsilon)$, since for quadratic differentials $q_1,q_2$ too far apart there will not be any triangulation $\Gamma=\{\gamma_i\}_{i=1}^{3(m+2g-2)}$ of $S_g$ by arcs and a continuous path $t\mapsto q_t$ which connects $q_1$ and $q_2$ and is such that the geodesic representative of the homotopy class of each $\gamma_i$ is a saddle connection for all $q_t$.\\
Another possibility concerned the use of a flat counterpart to geodesic laminations, called \textit{flat lamination} (for definitions and properties we refer the reader to \cite{Mo}) in order to obtain a triangulation of $S_g$.\\
Unfortunately, both approaches suffered of the same problem: instead of hyperbolic triangles, singular flat metrics require the use of Euclidean triangles. 
Indeed, one can triangulate a semi-translation surface $(X,q)$ with Euclidean triangles and stretch each side of each Euclidean triangle by the same factor $c>1$ as in proposition 2.2 of \cite{Th}, but then the resulting semi-translation surface will simply be $c\cdot q$.\\
The point is that in this case the sides of the triangles of the triangulation should be stretched by different factors. When trying to do so, one should notice that there are plenty of couples of Euclidean triangles $T_1,T_2$ with each side stretched by a factor lower or equal to $c>1$, and such that there could be no homeomorphism $f:T_1\rightarrow T_2$ which sends sides to corresponding sides and with $Lip(f)\le c$.

\begin{es}Consider the equilateral triangle $T_1$ with sides of length 1 and the isosceles triangle $T_2$ with base side of length 1 and height $\sqrt{3}$. Then clearly the maximal stretching of the sides of $T_1$ and $T_2$ is $\frac{\sqrt{13}}{2}$, while each homeomorphism \mbox{$f:T_1\rightarrow T_2$} which sends sides to corresponding sides must also send the arc parametrizing the height of $T_1$ to an arc of length at least $\sqrt{3}$. This implies that the Lipschitz constant of such $f$ must be at least $2>\frac{\sqrt{13}}{2}$.
\end{es}

The fundamental fact enlightened by the previous conter-example is that, if one tries to obtain a diffeomorphism $\varphi\in Diff_0^+(S_g,\Sigma)$ with $\mathcal{L}_{q_1}^{q_2}(\varphi)=K_F(q_1,q_2)$ by defining it first on the Euclidean triangles of a triangulation of $S_g$, then $\mathcal{L}_{q_1}^{q_2}(\varphi)$ should be attained along a curve of the triangulation. As a consequence, when searching for flat analogues to Thurston's stretch maps, one should impose strict conditions on the triangles considered.\\
As we made clear before, for $q_1$ and $q_2$ sufficiently close in $\mathcal{TQ}_g(\underline k,\epsilon)$, there is a triangulation $\Gamma=\{\gamma_i\}_{i=1}^{3(m+2g-2)}$ of $S_g$ by arcs with endpoints in $\Sigma$ such that the geodesic representative of each $\gamma_i$ for $|q_1|$ and $|q_2|$ is a saddle connection. This procedure provides us of a collection $\Xi^1=\{T^1_j\}_{j=1}^{2(m+2g-2)}$ of Euclidean triangles in the natural coordinates of $q_1$ and a collection $\Xi^2=\{T^2_j\}_{j=1}^{2(m+2g-2)}$ of Euclidean triangles in the natural coordinates of $q_2$.\\
Our problem is now to establish if there is a triangulation $\Gamma$ of $S_g$ such that it is possible to obtain a function $\varphi\in Diff_0^+(S_g,\Sigma)$ with $\mathcal{L}_{q_1}^{q_2}(\varphi)=K_F(q_1,q_2)$ by defining it first on each couple of corresponding triangles of $\Xi^1$ and $\Xi^2$.\\

To this end one should consider the following fact:\\

\textit{Given two Euclidean triangles $T_1,T_2$ with sides labeled, consider the set $L(T_1,T_2)$ of Lipschitz constants of diffeomorphisms $f:T_1\rightarrow T_2$ which send sides to corresponding sides in a linear way. \\
The minimum of $L(T_1,T_2)$ is the Lipschitz constant of the affine map $A$ which maps $T_1$ in $T_2$.}\\
 
Note that we considered functions which are linear on the sides of the triangles since we want the Lipschitz constant to be equal to ratio of lengths of a side. This suggests the fact that the function $\varphi$ we are trying to obtain should be affine on each triangle $T_j^1$ and that its greater eigenvalue should be attained on the most stretched side of $\Gamma$.\\

Finally, we see that this last condition imposes a very strong constrain on the collections $\Xi_1$ and $\Xi_2$ and consequently on the triangulation $\Gamma$. Since this problem is related to the nature of Euclidean triangles, it does not seem likely to be solved using flat laminations.\\

For the reasons we just explained, we were not able to prove the local equality $L_F=K_F$ trying to adapt Thurston's approach. In section 3 we will explain another approach we used to prove that the equality of two asymmetric pseudo-metrics $L_F^a$ and $K_F^a$ on $\mathcal{TQ}_g^{(1)}(\underline k,\epsilon)$ depends on two statements about 1-Lipschitz maps between polygons.
\bigskip
\subsection{Geodesics of $L_F$}

In the previous discussion we explained why we are not able to produce a flat counterpart to Thurston's stretch lines, but it is interesting nonetheless to investigate what do geodesics of $L_F$ look like.\\ 
We could only find geodesics of $L_F$ which are also geodesics of $K_F$: this is because the only feasible strategy to find geodesics $t\mapsto q_t$ of $L_F$ we could think of was to find functions $\varphi_t\in Diff_0^+(S_g,\Sigma)$ such that $\mathcal{L}_q^{q_t}(\varphi_t)=t=K_F(q,q_t)$ and then conclude from $K_F(q,q_t)\le L_F(q,q_t)$.\\

As one can easily notice, these geodesics of $L_F$ are very particular: as soon as some hypothesis are lighten, one can no longer be sure to find functions $\varphi_t$ such that $\mathcal{L}_q^{q_t}(\varphi_t)=t=K_F(q,q_t)$.\\

Let us explain first how to obtain geodesics of $L_F$ and $K_F$ entirely contained in one orbit of $GL(2,\mathbb{R})^+$.

\begin{prop}
Consider any $q\in \mathcal{TQ}_g(\underline k,\epsilon)$, and any pair of continuous functions $$\theta:[0,1]\rightarrow [0,2\pi), \quad f:[0,1]\rightarrow \mathbb{R}^+$$
such that for every $t_0,t_1\in [0,1]$, $t_0<t_1$ it results 
$$e^{t_1/t_0}\ge \max\Big\{\frac{f(t_1)}{f(t_0)},\frac{f(t_0)}{f(t_1)}\Big\}.$$
Using these data one can produce four geodesics for $K_F$ and $L_F$ starting at $q$ of the following form
$$\Phi^j:[0,1]\rightarrow \mathcal{TQ}_g(\underline k,\epsilon), \quad t\mapsto q^j_t:= e^{i\theta(t)}\cdot \Sigma_t^j\cdot q,\quad j=1,2,3,4,$$ 
where $\Sigma_t^j$ is one of the following four diagonal matrices
$$\Sigma_t^1:=\begin{pmatrix} e^t &0\\0&f(t)\end{pmatrix}\quad \Sigma_t^2:=\begin{pmatrix} e^{-t} &0\\0&f(t)\end{pmatrix}\quad \Sigma_t^3:=\begin{pmatrix} f(t) &0\\0&e^t\end{pmatrix}\quad \Sigma_t^4:=\begin{pmatrix} f(t) &0\\0&e^{-t}\end{pmatrix}.$$

\end{prop}

\begin{proof}

The proof is identical for all four geodesics, so we will just prove it for $\Phi^1$. \\
For any $t_0,t_1\in [0,1]$, $t_0<t_1$ it results $q_{t_1}^{1}=A\cdot q_{t_0}^1$, where $A=e^{i\theta(t_1)}\cdot  \Sigma \cdot e^{-i\theta(t_0)}$ and $\Sigma$ is the following diagonal matrix
$$\Sigma:=\begin{pmatrix} e^{t_1-t_0} & 0\\ 0&\frac{f(t_1)}{f(t_0)} \end{pmatrix}.$$
Since $\Phi^1$ is contained in a $GL(2,\mathbb{R})^+$-orbit, one can apply previous proposition 2.13 and get $K_F(q_{t_0}^1,q_{t_1}^1)=L_F(q_{t_0}^1,q_{t_1}^1)=t_1-t_0.$

\end{proof}

Given any Teichm\"uller geodesic $$\Psi:[0,1]\rightarrow \mathcal{T}_g, \quad t\mapsto [(X_t,h_t)]$$ with initial differential $q$ on $X$, we define its lifting on $\mathcal{TQ}_g(\underline k,\epsilon)$ to be $$\widetilde \Psi:[0,1]\rightarrow \mathcal{TQ}_g(\underline k,\epsilon),\quad t\mapsto q_t$$ where $q_t$ is the holomorphic quadratic differential on $X_t$ such that $h_t:X\rightarrow X_t$ is a Teichm\"uller map with respect to $q$ and $q_t$ and with dilatation $e^{2t}$. 
\begin{prop}
Liftings to $\mathcal{TQ}_g(\underline k,\epsilon)$ of Teichm\"uller geodesics are geodesics for $L_F$ and $K_F$.
\end{prop}

\begin{proof}
The claim follows immediately from the previous proposition: one just has to notice that the Teichm\"uller map $h_t$ can be locally written in natural coordinates of $q$ and $q_t$ as
$$h_t(x+iy)=e^{t}x+ie^{-t}y.$$
\end{proof}

At this point, one could be tempted to try to obtain other geodesics using the result of proposition 2.13. In particular, considering functions $\theta$ and $f$ as in proposition $2.15$, one may wonder if it could be possible to impose for example $q_t:=\Sigma^1_t\cdot e^{i\theta(t)}\cdot q$.\\ 
The answer is no: since the direction where the stretching $e^t$ is obtained varies, there is no hope to get $L_F(q_{t_0},q_{t_1})=t_1-t_0$ or $K_F(q_{t_0},q_{t_1})=t_1-t_0$.\\

It is possible however to obtain other kinds of geodesics modifying only one part of the semi-translation surface, as we will now explain.

\begin{prop}
Let $q\in\mathcal{TQ}_g(\underline k,\epsilon)$ be a semi-translation surface which contains a flat cylinder $C$ of height $h>0$ such that there is at least one saddle connection entirely contained in $C$ which realizes the height of the cylinder.\\
The arc $\Phi:[0,1]\rightarrow \mathcal{TQ}_g(\underline k,\epsilon)$, $t\mapsto q_t$, where $q_t$ is the semi-translation surface obtained from $q$ changing the height of the flat cylinder to $e^th$, is a geodesic for $L_F$ and $K_F$.
\end{prop}
\begin{proof}
Denote by $\gamma_1,\dots,\gamma_2$ the saddle connections entirely contained in $C$ which realize the height of the cylinder. Clearly, if $h$ is stretched by $e^t$ then the length of $\gamma_1,\dots,\gamma_k$ is stretched by the same factor.\\
For any $t_0,t_1\in [0,1]$, $t_0<t_1$, the semi-translation surface $q_{t_1}$ is obtained from $q_{t_0}$ stretching the height of the cylinder by the factor $e^{t_1-t_0}$. All saddle connections of $q_{t_0}$ different from $\gamma_1,\dots,\gamma_k$ are stretched by a factor which is smaller than $e^{t_1-t_0}$, and consequently one can conclude $$K_F(q_{t_1},q_{t_0})=\log\left(\frac{\hat l_{q_{t_1}}(\gamma_i)}{\hat l_{q_{t_0}}(\gamma_i)}\right)=t_1-t_0.$$
Without loss of generality, we can suppose the direction of the saddle connection $\gamma_1,\dots,\gamma_k$ is the vertical one. Consequently there is a function $\varphi\in Diff^+_0(S_g,\Sigma)$ which, in natural coordinates of $q_{t_0}$ and $q_{t_1}$, can be written as the affine function $\begin{pmatrix} 1& 0\\0& e^{t_1-t_0}\end{pmatrix} $ on the cylinder and as the  identity  on the complement of the cylinder.\\
From $L_F(q_{t_0},q_{t_1})\le\mathcal{L}_{q_{t_0}}^{q_{t_1}}(\varphi)=t_1-t_0=K_F(q_{t_0},q_{t_1})$  and  $L_F(q_{t_0},q_{t_1})\ge K_F(q_{t_0},q_{t_1})$ one gets the last desired equality $L_F(q_{t_1},q_{t_0})=t_1-t_0$.\\
\end{proof}

The idea behind the previous proposition can be applied also to the case of a semi-translation surface $q$ obtained \textit{gluing two semi-translation surfaces $q_1,q_2$ along a slit in the horizontal direction}. This means that one cuts two slits of the same length, one in $q_1$ and one in $q_2$, both in the horizontal direction. Each $q_i$ will then have boundary consisting of two segments: each segment of the boundary of $q_1$ will be glued with a segment of the boundary of $q_2$ and the resulting surface $q$ will have two singularities of total angle $4\pi$ at the extremities of the slit. 

\begin{prop}
Let $q\in\mathcal{TQ}_g(\underline k,\epsilon)$ be a semi-translation surface obtained gluing two semi-translation surfaces $q_1,q_2$ along a slit in the horizontal direction. Furthermore, suppose that $q_1$ is such that it contains a sequence of saddle connections $\{\gamma_n\}_{n\in \mathbb{N}}$ asymptotic in the vertical direction (i.e. the limit of the differences of their directions with the vertical direction is zero) such that no $\gamma_n$ intersects the slit. \\
Then one obtains the geodesic $\Phi:[0,1]\rightarrow \mathcal{TQ}_g(\underline k,\epsilon)$, $t\mapsto q_t$, where $q_t$ is the semi-translation surface obtained gluing $\begin{pmatrix}1&0 \\ 0&e^t \end{pmatrix}\cdot q_1$ and  $q_2$ along the same slit.

\end{prop}
\begin{proof}
The idea of the proof is very similar to the one of the previous proposition. \\
First of all notice that $q_t$ is a well-defined semi-translation surface since the slit is horizontal and $q_1$ is stretched only in the vertical direction. \\
Then, for every $t_0,t_1\in [0,1]$, $t_0<t_1$, from the fact that no $\gamma_n$ intersects the slit it follows 
$$K_F(q_{t_0},q_{t_1})=\lim_{n\rightarrow \infty}\log\left(\frac{\hat l_{q_{t_1}}(\gamma_n)}{\hat l_{q_{t_0}}(\gamma_n)}\right)=t_1-t_0.$$
One can then conclude noting $L_F(q_{t_0},q_{t_1})\le\mathcal{L}_{q_{t_0}}^{q_{t_1}}(Id)=t_1-t_0.$
\end{proof}

\bigskip
\section{Equality of asymmetric pseudo-metrics $L_F^a$ and $K_F^a$}
\bigskip
In this section we investigate the equality of two asymmetric pseudo-metrics $L_F^a$ and $K_F^a$ on each Teichm\"uller space $\mathcal{TQ}^{(1)}_g(\underline k,\epsilon)$ of holomorphic quadratic differentials of unitary area without simple poles. \\
In particular, using the method we develop in this section, the equality of $L_F^a$ and $K_F^a$ on whole $\mathcal{TQ}^{(1)}_g(\underline k,\epsilon)$ can be proved if two statements about 1-Lipschitz maps between planar polygons are true. We are able to prove the first statement, but the second one remains a conjecture: nonetheless, we explain why we believe it is true.

For any $g\ge 2$ and any Teichm\"uller space $\mathcal{TQ}_g^{(1)}(\underline{k},\epsilon)$ of holomorphic quadratic differentials of unitary area without simple poles, we define the function 

$$L^a_F:\mathcal{TQ}_g^{(1)}(\underline{k},\epsilon)\times \mathcal{TQ}_g^{(1)}(\underline{k},\epsilon)\rightarrow \mathbb{R}$$
associating to any pair $q_1,q_2 \in\mathcal{TQ}_g^{(1)}(\underline{k},\epsilon)$ of semi-translation surfaces of unitary area the quantity

$$L^a_F(q_1,q_2):=\inf\limits_{\varphi\in \mathcal{D}} \log(Lip(\varphi)_{q_1}^{q_2}),$$
$$Lip(\varphi)_{q_1}^{q_2}=\sup\limits_{p\in S_g\setminus \Sigma}\left(\sup\limits_{v\in T_pS_g, ||v||_{q_1}=1} ||d\varphi_pv||_{q_2}\right),$$
where $\mathcal{D}$ is the set of functions $\varphi:S_g\rightarrow S_g$ which are homotopic to the identity, differentiable almost everywhere and which fix the points of $\Sigma$.

\begin{prop}
The function $ L^a_F$ is an asymmetric pseudo-metric on $\mathcal{TQ}_g^{(1)}(\underline{k},\epsilon)$.
\end{prop}

\begin{proof}
It is clear that $L^a_F(q,q)=0$ for every $q\in \mathcal{TQ}_g(\underline k,\epsilon)$ and that $L^a_F$ is not symmetric.\\
Note that every function $\varphi\in\mathcal{D}$ must be surjective, since it has degree 1: from this fact it follows $Lip(\varphi)_{q_1}^{q_2}\ge 1$ and $Lip(\varphi)_{q_1}^{q_2}= 1$ if and only if $q_2=e^{i\theta}q_1$. \\
Finally, $L^a_F$ satisfies the triangular inequality since for every couple of functions \mbox{$\varphi,\phi\in\mathcal{D}$} it follows
$$Lip(\varphi\circ\phi)_{q_1}^{q_3}\le Lip(\varphi)_{q_2}^{q_3}Lip(\phi)_{q_1}^{q_2}.$$
\end{proof}
The other pseudo-metric we consider in the present section is $K_F^a$: for every $q_1,q_2\in \mathcal{TQ}^{(1)}_g(\underline k,\epsilon)$, $K_F^a(q_1,q_2)$ is defined as 
$$K^a_F(q_1,q_2):=\sup\limits_{\gamma\in SC(q_1)}\log\left(\frac{\hat l_{q_2}(\gamma)}{\hat l_{q_1}(\gamma)}\right).$$
For every $q_1,q_2\in \mathcal{TQ}^{(1)}_g(\underline k,\epsilon)$ it clearly results $$L_F^a(q_1,q_2)\ge K_F^a(q_1,q_2).$$

With the techniques exposed in the present section we are able to reduce the proof of the equality of $L_F^a$ and $K_F^a$ on the whole $\mathcal{TQ}_g^{(1)}(\underline{k},\epsilon)$ to the proof of two statements about 1-Lischitz maps between planar polygons. Given their importance, we feel it is necessary to briefly anticipate them now in a slightly simplified version.\\

Consider two planar polygons $\Delta$ and $\Delta'$ such that there is an injective function 
$$\iota:Vertices(\Delta)\rightarrow Vertices(\Delta')$$
which to every vertex $v$ associates a unique vertex $\iota(v)=v'$. Suppose both $\Delta$ and $\Delta'$ have exactly three vertices with strictly convex internal angle, which we denote $x_i$ and $x_i'$, $i=1,2,3$ respectively. \\
Suppose furthermore that for every $x,y\in Vertices(\Delta)$ it results
$$d_\Delta(x,y)\ge d_{\Delta'}(x',y'),$$
where $d_\Delta$ (resp. $d_{\Delta'}$) is the intrinsic Euclidean metric inside $\Delta$ (resp. $\Delta'$): $d_\Delta(x,y)$ (resp. $d_{\Delta'}(x',y')$) is defined as the infimum of the lengths, computed with respect to the Euclidean metric, of all paths from $x$ to $y$ (resp. from $x'$ to $y'$) entirely contained in $\Delta$ (resp. in $\Delta'$). \\
We say that vertices of $\Delta$ and of $\iota(Vertices(\Delta))$ are \textit{disposed in the same order} if it is possible to choose two parametrizations $\gamma:[0,1]\rightarrow \partial \Delta$ and $\gamma_1:[0,1]\rightarrow \partial \Delta'$ such that $\gamma(0)=x_1$, $\gamma_1(0)=x_1'$ and $\gamma,\gamma_1$ meet respectively vertices of $\Delta$ and of $\Delta'$ in the same order.

\begin{stm}(Theorem 3.21) 
If $Vertices(\Delta)$ and $\iota(Vertices(\Delta))$ are disposed in the same order, then there is a 1-Lipschitz map $f:\Delta\rightarrow \Delta'$ (with respect to the intrinsic Euclidean metrics of the polygons) which sends vertices to corresponding vertices.
\end{stm}
\begin{stm}(Conjecture 3.31) 
If $Vertices(\Delta)$ and $\iota(Vertices(\Delta))$ are not disposed in the same order, then for every point $p\in \Delta$ there is a point $p'\in \Delta'$ such that $$d_\Delta(p,x_i)\ge d_{\Delta'}(p',x_i'), \quad i=1,2,3.$$
\end{stm}
We were able to prove the first statement, which corresponds to following theorem 3.21, but not the second one, which from now on will be referred to as conjecture 3.31: we will still explain why we believe it must be true.\\
We state the following theorem, which is the main result of this paper.
\begin{thm}\label{bigthm}
If conjecture 3.31 is true, then for every $q_1,q_2\in \mathcal{TQ}_g^{(1)}(\underline{k},\epsilon)$, it results
$$L_F^a(q_1,q_2)= K^a_F(q_1,q_2).$$
\end{thm}

We proved theorem \ref{bigthm} using an approach similar to a proof by F.A. Valentine (which can be found in \cite{Va}) of Kirszbraun's theorem for $\mathbb{R}^2$ (firstly proved by M.D. Kirszbraun in \cite{Ki}).

\begin{thm} (Kirszbraun)\\\label{kirsz}
Let $S\subset \mathbb{R}^2$ be any subset and $f:S\rightarrow \mathbb{R}^2$ a 1-Lipschitz map.\\
Given any set $T$ which contains $S$, it is possible to extend $f$ to a 1-Lipschitz map $\hat f:T\rightarrow \mathbb{R}^2$ such that $\hat f(T)$ is contained in the convex hull of $f(S)$.
\end{thm}

The key ingredients of Valentine's proof of Kirszbraun theorem are the following two lemmas.

\begin{lem}\label{triang}
Fix two Euclidean triangles $\Delta(x_1,x_2,x_3)$ and $\Delta(x_1',x_2',x_3')$ in $\mathbb{R}^2$ such that 
$$|x_i'-x_j'|\le |x_i-x_j|\quad \textit{ for every }i,j=1,2,3.$$
Then for any $x_4\in \mathbb{R}^4$ there is a point $x_4'$ contained in $\Delta(x_1',x_2',x_3')$ such that 
$$|x_4'-x_i'|\le |x_4-x_i|\quad \textit{for every } i=1,2,3.$$
\end{lem}
The second lemma is often referred to as \textit{Helly's theorem} (firstly proved by E.Helly in \cite{He}).
\begin{lem}(Helly)\\
Let $F$ be any family of compact and convex subsets of $\mathbb{R}^n$. Suppose that for every $C_1,\dots, C_{n+1}\in F$ it results 
$$\bigcap_{i=1}^{n+1}C_i\neq \emptyset$$
then it also results
$$\bigcap_{C\in F}C\neq \emptyset.$$
\end{lem}

Together, these two lemmas imply the ensuing proposition, from which one easily deduces theorem \ref{kirsz}.
\begin{prop}
Given any two collections $\{B_{r_j}(x_j)\}_{j\in J}$ and $\{B_{r_j}(x_j')\}_{j\in J}$ of closed disks in $\mathbb{R}^2$ with the same radii and with centers such that 
$$|x_i'-x_j'|\le |x_i-x_j|.$$
Then, if 
$$\bigcap\limits_{j\in J}B_{r_j}(x_j)\neq \emptyset$$
it follows
$$\bigcap\limits_{j\in J}B_{r_j}(x_j')\neq \emptyset.$$
\end{prop}
\bigskip

We performed a similar reasoning in order to find a function $\phi\in \mathcal{D}$ such that $$Lip(\phi)_{q_1}^{\sigma_2}=1,$$
where $\sigma_2$ is the rescaled differential 
$$\sigma_2:=\frac{q_2}{e^{K^a_F(q_1,q_2)}}.$$
The existence of such function $\phi$ proves the equality 
\begin{equation}
e^{L_F^a(q_1,\sigma_2)}=e^{K_F^a(q_1,\sigma_2)}=1
\end{equation}
and consequently, since for every $c>0$ it follows 
$$c e^{L_F^a(q_1,q_2)}=e^{L_F^a(q_1,cq_2)},\quad ce^{K_F^a(q_1,q_2)}=e^{K_F^a(q_1,cq_2)}$$
multiplying both termes of equation $(5.1)$ by $e^{K^a_F(q_1,q_2)}$ and then composing with the logarithm, one gets the desired result
$$L_F^a(q_1,q_2)=K_F^a(q_1,q_2).$$

It is important to specify that in our proof we used the following version of Helly's lemma, which can be found in \cite{Iv}.

\begin{lem}\label{hellyiv}
Let $X$ be a uniquely geodesic space of compact topological dimension $n<\infty$. If $\{A_j\}_{j\in J}$ is any finite collection of convex sets in $X$ such that every subcollection of cardinality at most $n+1$ has a nonempty intersection, then 
$$\bigcap\limits_{j\in J} A_j\neq \emptyset.$$
\end{lem}
If $q$ is a holomorphic quadratic differential on a closed Riemann surface of genus $g\ge 2$ one can consider a universal cover $\pi:\widetilde S_g\rightarrow S_g$ and the pullback $\widetilde q$ of $q$ on $\widetilde S_g$. Then $|\widetilde q|$ induces a metric which is $Cat(0)$ and consequently uniquely geodesic. But if $q$ has poles then $|\widetilde q|$ does not induce an uniquely geodesic metric space: this is the reason why our proof could not be adapted to the Teichm\"uller space of quadratic differentials with poles.
\bigskip

One should notice that the equality $L_F^a=K_F^a$ could be implied by a version of Kirszbraun theorem which suits semi-translation surfaces (without simple poles). The generalization of theorem \ref{kirsz} which could be considered closer to semi-translation surfaces was proved by S.Alexander, V.Kapovitch and A.Petrunin in \cite{AKP} and applies to the case of functions from complete $CBB(k)$ spaces (spaces with curvature bounded below by $k$) to complete $Cat(k)$ spaces (spaces with curvature bounded above by $k$). Since semi-translation surfaces are only locally $Cat(0)$ spaces, unfortunately the theorem of \cite{AKP} does not apply to our case.\\

At this point it should be more clear why we decided to prove the equality of the two pseudo-metrics $L_F^a,K_F^a$ instead of the equality of the two pseudo-metrics $L_F,K_F$ studied in the preceding section.\\
Indeed, one reason is that it is more convenient to study asymmetric pseudo-metrics, since it is more complicated to control both Lipschitz constants (the lower and the upper one) at once: for an attempt in this direction in the simple case of the unit square see \cite{DP}.\\
The other reason is that using this kind of \textit{Kirszbraun approach} there is no hope to obtain an injective 1-Lipschitz function. This is the reason why we defined $L_F^a$ as the infimum of Lipschitz constants of functions in $\mathcal{D}$.\\
Finally one should notice that the condition of unitary area of the two semi-translation surfaces $q_1$ and $q_2$ will never be used in the proof. We could actually prove the equality of $L_F^a$ and $K_F^a$ on the whole $\mathcal{TQ}_g(\underline k,\epsilon)$, where the two pseudo-metrics are much more degenerate.  \\

The next section is devoted to the explanation of our proof of the construction of the function $\phi\in \mathcal{D}$ such that $Lip(\phi)_{q_1}^{\sigma_2}=1$.

\subsection{Proof of the equality}


Let $\pi:\widetilde S_g\rightarrow S_g$ be a universal cover. Lifting through $\pi$ the complex structure of $X_1$ and the differential $q_1$ one obtains the metric universal cover  $\pi:(\widetilde X_1,|\widetilde q_1|)\rightarrow (X_1,|q_1|)$ and doing the same thing to $X_2$ and $q_2$ one obtains the metric universal cover $\pi:(\widetilde X_2,|\widetilde \sigma_2|)\rightarrow (X_2,|\sigma_2|)$.\\
Denote by $d_{\widetilde {q_1}}$ the $Cat(0)$ metric induced by $|\widetilde q_1|$ and by $d_{\widetilde {\sigma_2}}$ the $Cat(0)$ metric induced by $|\widetilde \sigma_2|$. In order to avoid confusion, when we will want to underline that a point of $\widetilde S_g$ is regarded as a point of $\widetilde X_2$, we will denote it with an additional prime symbol: for example a point $\widetilde x\in \pi^{-1}(\Sigma)$ will be denoted as $\widetilde x$ if regarded as a point of $\widetilde X_1$ and $\widetilde x'$ if regarded as a point of $\widetilde X_2$.\\
For every couple of points $\widetilde x,\widetilde y\in \widetilde X_1$, $\overline{ \widetilde x\widetilde y}$ is the $d_{\widetilde q_1}$-geodesic from $\widetilde x$ to $\widetilde y$. Since there will be no ambiguity, we will denote geodesics of $d_{\widetilde \sigma_2}$ in the same way: for every couple of points $\widetilde x',\widetilde y'\in \widetilde X_2$, $\overline{ \widetilde x'\widetilde y'}$ is the $d_{\widetilde \sigma_2}$-geodesic from $\widetilde x'$ to $\widetilde y'$.\\

Fix a point $x_0\in\Sigma\subset S_g$ and $\widetilde x_0\in \pi^{-1}(x_0)$: as it is well known, the group $\pi_1(S_g,x_0)$ acts on $\widetilde S_g$ and for every $\gamma\in\pi_1(S_g,x_0)$, $\widetilde x\in \widetilde S_g$ it results
$$\gamma\cdot \widetilde x=\widetilde \tau(1),$$
where $\widetilde \tau$ is the lifting of $\gamma * \pi(\widetilde \sigma)$ ($\widetilde \sigma$ is any path in $\widetilde S_g$ from $\widetilde x_0$ to $\widetilde x$) such that $\widetilde \tau(0)=\widetilde x_0$.\\

Fix a fundamental domain $P\subset (\widetilde X_1,d_{\widetilde q_1})$ for the action of $\pi_1(S_g,x_0)$, suppose $\widetilde x_0\in P$.\\
We want to build a map $\hat\phi:\hat U\rightarrow (\widetilde X_2,d_{\widetilde \sigma_2})$ (where $\hat U$ is a dense countable subset of $P$ which includes the zeroes of $\widetilde q_1$ contained in $P$), such that for every couple of points $\widetilde x,\widetilde y\in \hat U$ (eventually equal) and every $\gamma\in \pi_1(S_g,x_0)$, it results 
\begin{equation}
d_{\widetilde \sigma_2}(\hat \phi(\widetilde x),\gamma\cdot \hat \phi(\widetilde y))\le d_{\widetilde q_1}(\widetilde x,\gamma\cdot \widetilde y)
\end{equation}
and for every zero $\widetilde z$ of $\widetilde q_1$ contained in $P$ it results $\hat \phi(\widetilde z)=\widetilde z'$ (notice that $\widetilde q_1$ and $\widetilde q_2$ have zeroes in the same points, which are the points of $\pi^{-1}(\Sigma)$).\\

Having done so, we define the dense subset $\widetilde U:=\pi_1(S_g,x_0)\cdot \hat U$ of $\widetilde X_1$ and extend the function $\hat \phi$ by equivariance to a function $\widetilde \phi^U:\widetilde U\rightarrow \widetilde X_2$, imposing
$$\widetilde \phi^U(\gamma\cdot \widetilde x):=\gamma\cdot\hat \phi(\widetilde x)$$
for every $\gamma\in \pi_1( X_1,x_0),\widetilde x\in \hat U$.\\
Notice that for every $\gamma_1\cdot \widetilde x_1,\gamma_2\cdot \widetilde x_2\in \widetilde U$ it results: 
$$d_{\widetilde \sigma_2}( \widetilde \phi^U(\gamma_1\cdot \widetilde x_1),\widetilde  \phi^U(\gamma_2\cdot \widetilde x_2))=d_{\widetilde \sigma_2}(\gamma_1\cdot \hat \phi(\widetilde x_1),\gamma_2\cdot \hat \phi(\widetilde x_2))=d_{\widetilde \sigma_2}( \hat \phi(\widetilde x_1),(\gamma_1^{-1} * \gamma_2)\cdot \hat \phi(\widetilde x_2))\le$$$$\le d_{\widetilde q_1}(\widetilde x_1,(\gamma_1^{-1}* \gamma_2)\cdot \widetilde x_2))=d_{\widetilde q_1}(\gamma_1\cdot \widetilde x_1, \gamma_2\cdot \widetilde x_2)$$ 
and consequently $\widetilde \phi^U$ can be extended to a function $\widetilde \phi:(\widetilde X_1,|\widetilde q_1|)\rightarrow (\widetilde X_2,|\widetilde \sigma_2|)$ which has Lipschitz constant 1. \\
In particular, for every point $\widetilde x\in \widetilde X\setminus \widetilde U$ we define $\widetilde \phi(\widetilde x)$ as 
$$\widetilde \phi(\widetilde x):=\lim\limits_{n\rightarrow \infty}\widetilde \phi^U(\widetilde x_n),$$
where $\{\widetilde x_n\}_{n\in \mathbb{N}}\subset \widetilde U$ is a sequence such that $\lim\limits_{n\rightarrow \infty}\widetilde x_n=x$: since $\widetilde \phi^U$ is 1-Lipschitz on $\widetilde U$, the limit in the definition of $\widetilde \phi(\widetilde x)$ exists and does not depend from the chosen sequence $\{\widetilde x_n\}_{n\in \mathbb{N}}$.\\
Notice furthermore that $\widetilde \phi$ is equivariant for the action of $\pi_1(S_g,x_0)$: for every $\widetilde x\in \widetilde X\setminus \widetilde U$ and $\gamma\in \pi_1(S_g,x_0)$ consider a sequence $\{\widetilde x_n\}_{n\in \mathbb{N}}\subset \widetilde U$ such that $\lim\limits_{n\rightarrow \infty}\widetilde x_n=\widetilde x$, then it results $\lim\limits_{n\rightarrow \infty}\gamma\cdot\widetilde x_n=\gamma\cdot \widetilde x$ and consequently
$$\widetilde \phi(\gamma\cdot \widetilde x)= \lim\limits_{n\rightarrow \infty}\widetilde \phi^U(\gamma\cdot \widetilde x_n)=\lim\limits_{n\rightarrow \infty}\gamma\cdot\widetilde \phi^U(\widetilde x_n)=\gamma\cdot\lim\limits_{n\rightarrow \infty}\widetilde \phi^U(\widetilde x_n)=\gamma\cdot\widetilde \phi(\widetilde x).$$

We have proved that $\widetilde \phi$ descends to a function $\phi:(X_1,q_1)\rightarrow (X_2,\sigma_2)$ which is $1-$Lipschitz and such that
$$(\phi)_*=Id:\pi_1(S_g,x_0)\rightarrow \pi_1(S_g,x_0)$$ 
which implies that $\phi$ is homotopic to the identity.\\
In the rest of the section we will explain how to obtain a function $\hat \phi$ which satisfies previous inequality $(5.2)$.\\

We have imposed $\hat\phi(\widetilde z)=\widetilde z'$ for every zero $\widetilde z$ of $\widetilde q_1$ which is contained in $P$, so we have to verify
$$d_{\widetilde \sigma_2}(\widetilde z_1',\gamma\cdot\widetilde z_2')\le d_{\widetilde q_1}(\widetilde z_1,\gamma\cdot \widetilde z_2)$$ 
for every pair of zeroes $\widetilde z_1,\widetilde z_2$ of $\widetilde q_1$ contained in $P$ and every $\gamma\in \pi_1(S_g,x_0)$. \\
Notice that it results $d_{\widetilde \sigma_2}(\widetilde z_1',\gamma\cdot\widetilde z_2')=\hat l_{\sigma_2}(\tau)$, where $\hat l_{\sigma_2}(\tau)$ is the length of the geodesic representative for $|\sigma_2|$ of the homotopy class (with fixed endpoints) of $\pi(\widetilde\tau)$ and $\widetilde \tau$ is any arc in $\widetilde X_2$ from $\widetilde z_1'$ to $\gamma\cdot \widetilde z_2'$. In the same way it results $d_{q_1}(\widetilde z_1,\gamma\cdot\widetilde z_2)=\hat l_{q_1}(\tau)$. \\
Let $\tau^{q_1}$ be the geodesic representative for $|q_1|$ of the homotopy class (with fixed endpoints) of $\tau$ and suppose $\tau^{q_1}$ is a concatenation of $k\ge 1$ saddle connections  $\tau_1^{q_1},\dots ,\tau_k^{q_1}$. \\
From the definition of $\sigma_2$ it follows
$$l_{q_1}(\tau_i^{q_1})\ge \hat l_{\sigma_2}(\tau_i^{q_1})$$
for every $i=1,\dots,k$. We thus obtain the following inequalities: 
$$d_{q_1}(\widetilde z_1,\gamma\cdot\widetilde z_2)=l_{q_1}(\tau^{q_1})=\sum\limits_{i=1,\dots,k} l_{q_1}(\tau_i^{q_1})\ge \sum\limits_{i=1,\dots,k}\hat l_{\sigma_2}(\tau_i^{q_1})\ge \hat l_{\sigma_2}(\tau)=d_{\sigma_2}(\widetilde z_1',\gamma\cdot\widetilde z_2').$$

Now we are going to define the function $\hat \phi$ on $\hat U$ one point at a time.\\
Let $\widetilde p_1\in P\setminus \pi^{-1}(\Sigma)$ be the first point (besides the zeroes of $\widetilde q_1$) on which we want to define $\hat\phi$: we have to find $\hat\phi(\widetilde p_1)\in \widetilde{X}_2$ such that
\begin{equation}
d_{\widetilde \sigma_2}(\hat\phi(\widetilde p_1),\gamma\cdot \widetilde x')\le d_{\widetilde {q_1}}(\widetilde p_1,\gamma\cdot \widetilde x)
\end{equation}
for every zero $\widetilde x$ of $\widetilde q_1$ contained in $P$ and for every $\gamma\in \pi_1(S_g,x_0)$. The point $\hat\phi(\widetilde p_1)$ should also satisfy the condition
\begin{equation}
d_{\widetilde \sigma_2}(\hat\phi(\widetilde p_1),\theta\cdot \hat\phi(\widetilde p_1))\le d_{\widetilde q_1}(\widetilde p_1,\theta\cdot \widetilde p_1)
\end{equation}
for every $\theta\in \pi_1(S_g,x_0)$.\\
Notice that, in order for equation $(5.3)$ to be always satisfied, it is sufficient to check only the distances of $\widetilde p_1$ from the zeroes $\gamma\cdot \widetilde x$ such that $\overline{\widetilde p_1(\gamma\cdot \widetilde x)}$ is smooth and does not contain other zeroes. Indeed, suppose  $\overline{\widetilde p_1(\gamma\cdot \widetilde x)}$ is the concatenation of the following segments
$$\overline{\widetilde p_1(\gamma\cdot \widetilde x)}=\overline{\widetilde p_1(\gamma_1\cdot \widetilde x_1)}*\widetilde \tau_1^{q_1}*\dots*\widetilde \tau_l^{q_1},$$
where:
\begin{itemize}
\item $\gamma_1\in \pi_1(S_g,x_0)$, 
\item $\widetilde x_1$ is a zero of $\widetilde q_1$ contained in $P$,
\item $\widetilde \tau_i^{q_1}$ are saddle connections for $\widetilde q_1,$
\end{itemize}
then from the inequality
$$d_{\widetilde \sigma_2}(\hat \phi(\widetilde p_1),\gamma_1\cdot \widetilde x_1')\le d_{\widetilde q_1}(\widetilde p_1,\gamma_1\cdot \widetilde x_1)$$
and the definition of $\sigma_2$ it will follow
$$d_{\widetilde q_1}(\widetilde p_1,\gamma\cdot \widetilde{x})=d_{\widetilde q_1}(\widetilde p_1,\gamma_1\cdot \widetilde x_1)+\sum\limits_{i=1,\dots,l}l_{\widetilde q_1}(\widetilde \tau_i^{q_1})\ge$$ $$\ge d_{\widetilde \sigma_2}(\hat \phi(\widetilde p_1),\gamma_1\cdot \widetilde x_1')+\sum\limits_{i=1,\dots,l}\hat l_{\widetilde \sigma_2}(\widetilde \tau_i^{q_1}) \ge d_{\widetilde \sigma_2}(\hat \phi(\widetilde p_1),\gamma\cdot \widetilde x').$$
For the same reason it suffices to verify equation $(5.4)$ only for $\theta\in \pi_1(S_g,x_0)$ such that $\overline{\widetilde p_1(\theta\cdot \widetilde p_1)}$ is smooth and does not contain zeroes of $\widetilde q_1$.\\

We define the following two sets: 
$$\mathcal{X}(\widetilde p_1):=\{\widetilde z\in \pi^{-1}(\Sigma)\text{ such that } \overline{\widetilde z\widetilde p_1} \text{ is smooth and does not contain other zeroes of } \widetilde{q_1} \},$$ 
$$\Theta(\widetilde p_1):=\{\gamma\in \pi_1(S_g,x_0) \text{ such that }  \overline{\widetilde p_1(\theta\cdot \widetilde p_1)} \text{ is smooth and does not contain zeroes of } \widetilde{q_1} \}.$$
For every $\theta\in \Theta(\widetilde p_1)$ we define the set
$$V_\theta:=\{\widetilde p'\in \widetilde X_2\enskip |\enskip d_{\widetilde \sigma_2}(\widetilde p',\theta\cdot \widetilde p')\le d_{\widetilde q_1}(\widetilde p_1,\theta\cdot \widetilde p_1) \}.$$
\begin{lem}
For every $\theta\in \Theta(\widetilde p_1)$, the set $V_\theta$ is convex in $(\widetilde X_2,d_{\widetilde \sigma_2})$.
\end{lem}
\begin{proof}
Consider any two points $\widetilde x',\widetilde y'\in V_\theta$. Since $\widetilde \sigma_2$ is invariant by covering transformations, it is possible to obtain two parametrizations $\widetilde \tau,\widetilde \tau_\theta:[0,1]\rightarrow \widetilde X_2$ respectively of $\overline{\widetilde x'\widetilde y'}$ and of $\overline{(\theta\cdot \widetilde x')(\theta\cdot \widetilde y')}$ such that $\widetilde \tau_\theta(s)=\theta\cdot \widetilde \tau(s)$.\\
The space $(\widetilde X_2,d_{\widetilde \sigma_2})$ is $Cat(0)$ and consequently Busemann-convex: this means that the function $$s\mapsto d_{\widetilde \sigma_2}(\tau(s),\tau_\theta(s))=d_{\widetilde \sigma_2}(\tau(s),\theta\cdot\tau(s))$$ is convex. From this fact we get $\tau(s)\in V_\theta$ for every $s\in [0,1]$.
\end{proof}

For every $\widetilde x\in \mathcal{X}(\widetilde p_1)$ we define the following closed ball
$$B^2_{d_{\widetilde q_1}(\widetilde x,\widetilde p_1)}(\widetilde x'):=\{\widetilde p'\in \widetilde X_2|d_{\widetilde \sigma_2}(\widetilde p',\widetilde x')\le d_{\widetilde q_1}(\widetilde x,\widetilde p_1)\}.$$

Clearly, our goal is to prove that the set $\Pi(\widetilde p_1)$,
$$\Pi(\widetilde p_1):=\left(\bigcap\limits_{\widetilde x\in \mathcal{X}(\widetilde p_1)}B^2_{d_{\widetilde q_1}(\widetilde x,\widetilde p_1)}(\widetilde x')\right)\bigcap\left(\bigcap\limits_{\theta\in \Theta(\widetilde p_1)}V_\theta\right)$$
is not empty, in order to being able to choose $\hat\phi(\widetilde p_1)\in\Pi(\widetilde p_1)$.\\ 
Since the sets $B^2_{d_{\widetilde q_1}(\widetilde x,\widetilde p_1)}(\widetilde x')$ and $V_\theta$ are convex, we can use Helly's lemma \ref{hellyiv} for uniquely geodesic spaces to prove $\Pi(\widetilde p_1)\neq \emptyset$.
There are four cases:
\begin{enumerate}
\item $B^2_{d_{\widetilde q_1}(\widetilde x_1,\widetilde p_1)}(\widetilde x_1')\cap B^2_{d_{\widetilde q_1}(\widetilde x_2,\widetilde p_1)}(\widetilde x_2')\cap B^2_{d_{\widetilde q_1}(\widetilde x_3,\widetilde p_1)}(\widetilde x_3')\neq \emptyset,$
\item $V_{\theta_1}\cap B^2_{d_{\widetilde q_1}(\widetilde x_1,\widetilde p_1)}(\widetilde x_1')\cap B^2_{d_{\widetilde q_1}(\widetilde x_2,\widetilde p_1)}(\widetilde x_2')\neq \emptyset,$
\item $V_{\theta_1}\cap V_{\theta_2}\cap B^2_{d_{\widetilde q_1}(\widetilde x,\widetilde p_1)}(\widetilde x')\neq \emptyset,$
\item $V_{\theta_1}\cap V_{\theta_2}\cap V_{\theta_3}\neq \emptyset.$
\end{enumerate}

The proofs of the four cases will be presented later, since we feel it is now best to conclude the procedure of the definition of $\hat \phi$.\\
So suppose we have proved each of the four preceding cases and we have chosen $\hat \phi(\widetilde p_1)\in \Pi(\widetilde p_1)$, we now have to find the image of a second point $\widetilde p_2\in P\setminus \pi^{-1}(\Sigma)$ in such a way that it results: 
\begin{enumerate}[label=(\roman*)]
\item  
$$d_{\widetilde \sigma_2}(\hat\phi(\widetilde p_2),\gamma\cdot \widetilde x')\le d_{\widetilde q_1}(\widetilde p_2,\gamma\cdot \widetilde x)$$
for every zero $\widetilde x$ of $\widetilde q_1$ contained in $P$ and $\gamma\in \pi_1(S_g,x_0)$ such that $\overline{\widetilde p_2(\gamma\cdot \widetilde x)}$ is smooth and does not contain other zeroes of $\widetilde q_1$, 
\item 
$$d_{\widetilde \sigma_2}(\hat\phi(\widetilde p_2),\gamma\cdot \hat\phi(\widetilde p_1))\le d_{\widetilde q_1}(\widetilde p_2,\gamma\cdot \widetilde p_1)$$
for every $\gamma\in \pi_1(S_g,x_0)$ such that $\overline{\widetilde p_2(\gamma\cdot \widetilde p_1)}$ is smooth and does not contain zeroes of $\widetilde q_1$,
\item 
 $$d_{\widetilde \sigma_2}(\hat\phi(\widetilde p_2),\theta\cdot \hat\phi(\widetilde p_2))\le d_{\widetilde q_1}(\widetilde p_2,\theta\cdot \widetilde p_2)$$
for every $\theta\in \pi_1(S_g,x_0)$ such that $\overline{\widetilde p_2(\theta\cdot \widetilde p_2)}$ is smooth and does not contain zeroes of $\widetilde q_1$.
\end{enumerate}
As we did for $\widetilde p_1$, we now define the sets $\mathcal{X}(\widetilde p_2)_\Sigma, \mathcal{X}(\widetilde p_2)_{\widetilde p_1}$ and $\Theta(\widetilde p_2)$:
$$ \mathcal{X}(\widetilde p_2)_\Sigma:=\{\widetilde x\in \pi^{-1}(\Sigma)\enskip|\enspace \overline{\widetilde p_2 \widetilde x}\text{ is smooth and does not contain other zeroes of }\widetilde{q_1}\},$$
$$ \mathcal{X}(\widetilde p_2)_{\widetilde p_1}:=\{\gamma\cdot \widetilde p_1\enspace|\enspace\gamma\in \pi_1(S_g,x_0) \text{ and } \overline{\widetilde p_2(\gamma\cdot \widetilde p_1)}\text{ is smooth and does not contain zeroes of }\widetilde{q_1}\},$$
$$\Theta(\widetilde p_2):=\{\theta\in \pi_1(S_g,x_0)\enskip|\enskip \overline{\widetilde p_2(\theta\cdot \widetilde p_2)}\text{ is smooth and does not contain zeroes of }\widetilde{q_1}\}.$$
We define the following intersections:
$$B_\Sigma:=\bigcap\limits_{\widetilde x\in \mathcal{X}(\widetilde p_2)_\Sigma}B^2_{d_{\widetilde q_1}(\widetilde x,\widetilde p_2)}(\widetilde x'),$$
$$B_{\widetilde p_1}:=\bigcap\limits_{\gamma\cdot \widetilde p_1\in \mathcal{X}(\widetilde p_2)_{\widetilde p_1}}B^2_{d_{\widetilde q_1}(\gamma\cdot\widetilde p_1,\widetilde p_2)}(\gamma\cdot \hat\phi(\widetilde p_1)),$$
$$V_{\widetilde p_2}:=\bigcap\limits_{\theta\in \Theta(\widetilde p_2)}V_\theta.$$
Again, we want to prove
$$\Pi(\widetilde p_2):=B_\Sigma\cap B_{\widetilde p_1}\cap V_{\widetilde p_2}\neq \emptyset$$
in order to pick $\hat\phi(\widetilde p_2)\in \Pi(\widetilde p_2)$. One can consider the four cases we previously deduced for $\Pi(\widetilde p_1)$, noting that this time the closed balls can also be centered in points $\gamma\cdot \hat \phi(\widetilde p_1)$.\\

We now proceed in the same way, defining $\hat\phi$ on $P$ one point at a time. \\
Suppose $\hat \phi$ is already defined on the points $\widetilde p_1,\dots,\widetilde p_n\in P\setminus \pi^{-1}(\Sigma)$ and that we wish to determine its value at $\widetilde p_{n+1}$. In order to do so we define the following sets:
$$ \mathcal{X}(\widetilde p_{n+1})_\Sigma:=\{\widetilde x\in \pi^{-1}(\Sigma)\enskip|\enskip \overline{\widetilde p_{n+1}\widetilde x}\text{ is smooth and does not contain any other zero of } \widetilde{q_1}\},$$
$$\Theta(\widetilde p_{n+1}):=\{\theta\in \pi_1(S_g,x_0)\enskip|\enskip\overline{\widetilde p_{n+1}(\theta\cdot \widetilde p_{n+1})}\text{ is smooth and does not contain zeroes of } \widetilde{q_1}\},$$
$$ \mathcal{X}(\widetilde p_{n+1})_{\widetilde p_i}:=\{\gamma\cdot \widetilde p_i\enskip|\enskip\gamma\in \pi_1(S_g,x_0) \text{ and } \overline{\widetilde p_{n+1}(\gamma\cdot \widetilde p_i)}\text{ is smooth and does not contain zeroes of }\widetilde{q_1}\},$$
for every $i=1,\dots,n$.\\
Again, we want to prove

$$\Pi(\widetilde p_{n+1}):=B_\Sigma\cap\left(\bigcap\limits_{i=1,\dots,n}B_{\widetilde p_{i}}\right)\cap V_{\widetilde p_{n+1}}\neq \emptyset,$$
where the sets $B_\Sigma,B_{\widetilde p_i},V_{\widetilde p_{n+1}}$ are defined as follows:
$$B_\Sigma:=\bigcap\limits_{\widetilde x\in \mathcal{X}(\widetilde p_{n+1})_\Sigma}B^2_{d_{\widetilde q_1}(\widetilde x,\widetilde p_{n+1})}(\widetilde x'),$$
$$B_{\widetilde p_i}:=\bigcap\limits_{\gamma\cdot \widetilde p_i\in \mathcal{X}(\widetilde p_{n+1})_{\widetilde p_i}}B^2_{d_{\widetilde q_1}(\gamma\cdot\widetilde p_i,\widetilde p_{n+1})}(\gamma\cdot \hat\phi(\widetilde p_i)),$$
$$V_{\widetilde p_{n+1}}:=\bigcap\limits_{\theta\in \Theta(\widetilde p_{n+1})}V_\theta.$$

Then we will pick $\hat\phi(\widetilde p_{n+1})\in\Pi(\widetilde p_{n+1})$: notice that even in this case there are only the same four types of intersections we pointed out for $\widetilde p_1$.\\

Since we have now fully explained our method to define $\hat \phi$ on a dense countable subset of $P$, we can now concentrate on the four types of intersections which appear in the sets $\Pi(\widetilde p_{i})$ (we will prove it for $\Pi(\widetilde p_{n+1})$, the reasoning will be the same for the other sets $\Pi(\widetilde p_{i})$).\\
The following procedure will not vary in case closed balls are centered in zeroes of $\widetilde q_1$ or in points outside $\pi^{-1}(\Sigma)$: in order to lighten the notation, given any point $\widetilde x=\gamma\cdot \widetilde p_i\in \mathcal{X}(\widetilde p_{n+1})_{\widetilde p_i}$, we will denote the corresponding point $\gamma\cdot \hat \phi(\widetilde p_i)$ simply as $\widetilde x'$.\\
From now on we will also denote the set $\mathcal{X}(\widetilde p_{n+1})_\Sigma\cup (\cup_{i=1}^n\mathcal{X}(\widetilde p_{n+1})_{\widetilde p_i})$ simply as $\mathcal X(\widetilde p_{n+1})$.\\ 

The first case concerns the intersection of three closed balls and is the most important, since it will imply all other three cases. Its proof is quite long and involves the two statements about 1-Lipschitz maps between polygons we introduced at the beginning of this section: for these reasons we feel it is best to postpone it and dedicate to it the whole next section.\\We will thus state the following theorem and take it for granted.

\begin{thm} 
If following conjecture 3.31 is true, for every $\widetilde x_1,\widetilde x_2,\widetilde x_3\in \mathcal{X}(\widetilde p_{n+1})$ it results $$B^2_{d_{\widetilde q_1}(\widetilde x_1,\widetilde p_{n+1})}(\widetilde x'_1)\cap B^2_{d_{\widetilde q_1}(\widetilde x_2,\widetilde p_{n+1})}(\widetilde x'_2)\cap B^2_{d_{\widetilde q_1}(\widetilde x_3,\widetilde p_{n+1})}(\widetilde x'_3)\neq \emptyset.$$
\end{thm}

It is important to notice that all next results will be implied by theorem 3.11: the reader is advised to keep in mind that they consequently depend on conjecture 3.31.\\

We state the following corollary, which is a consequence of theorem 3.11, Helly's lemma and some observations we already made.
\begin{cor}\label{corintersection}
Consider any finite number of points $\widetilde y_1,\dots,\widetilde y_n\in \widetilde X_1\setminus \pi^{-1}(\Sigma)$ and $\widetilde y_1',\dots,\widetilde y_n'\in \widetilde X_2\setminus \pi^{-1}(\Sigma)$ such that
$$d_{\widetilde \sigma_2}(\widetilde y_i',\widetilde y_j')\le d_{\widetilde q_1}(\widetilde y_i,\widetilde y_j) \quad \forall i,j=1,\dots,n$$
and
$$d_{\widetilde \sigma_2}(\widetilde y_i',\widetilde z')\le d_{\widetilde q_1}(\widetilde y_i,\widetilde z)$$ for every $\widetilde z\in \pi^{-1}(\Sigma)$ and $i=1,\dots,n$.\\
Then for every finite set of zeroes $\widetilde x_1,\dots,\widetilde x_m\in \pi^{-1}(\Sigma)$ and for every $\widetilde p\in \widetilde X_1$ it results
$$\left(\bigcap\limits_{i=1,\dots,n}B^2_{d_{\widetilde q_1}(\widetilde y_i,\widetilde p)}(\widetilde y_i')\right) \bigcap\left(\bigcap\limits_{i=1,\dots,m}B^2_{d_{\widetilde q_1}(\widetilde x_i,\widetilde p)}(\widetilde x_i')\right) \neq \emptyset.$$

\end{cor}
\begin{proof}
Closed balls of $d_{\widetilde \sigma_2}$ are convex, so one can use Helly's lemma \ref{hellyiv} and prove that the intersection of every triple of closed balls is not empty.\\
As we have already seen, given a point $\widetilde y_i$, if it results 
$$\overline{\widetilde y_i\widetilde p}=\overline{\widetilde p\widetilde z}*\widetilde \tau_1^{q_1}*\dots *\widetilde \tau_r^{q_1}*\overline{\widetilde w\widetilde y_i}$$
with $\widetilde w,\widetilde z\in \pi^{-1}(\Sigma)$,  $\widetilde \tau_i^{q_1}$  saddle connections and $\overline{\widetilde p\widetilde z}, \overline{\widetilde w\widetilde y_i}$ smooth, one can replace the ball $B^2_{d_{\widetilde q_1}(\widetilde y_i,\widetilde p)}(\widetilde y_i')$ in the intersection with the ball $B^2_{d_{\widetilde q_1}(\widetilde z,\widetilde p)}(\widetilde z')$. The same is true for all points $\widetilde x_i$.\\
The result then follows directly from theorem 3.11.
\end{proof}

We now want to focus ourselves on the remaining three cases. In order to do so we first need to characterize closed geodesics and flat cylinders of a semi-translation surface $(X,q)$. A proof of the following lemma can be found in \cite{St}.

\begin{lem}
Let $\theta$ be a simple closed geodesic for $|q|$ on $X$. Then $\theta$ is a cylinder curve of a flat cylinder $C$ of $(X,q)$. This means that $C$ is foliated by simple closed geodesics all parallel to $\theta$ and of the same length. The border of $C$ is composed by two components, both consisting of saddle connections of $q$ parallel to $\theta$. The length of both components equals the length of $\theta$.

\end{lem}

\begin{lem}
Consider any $\theta\in \Theta(\widetilde p_{n+1})$ and let $\widetilde C$ be the lifting to $\widetilde X_1$ of the flat cylinder of $(X_1,q_1)$ corresponding to $\theta$.\\ 
Let $\widetilde y$ be any point of $\partial \widetilde C$ and $\widetilde z_1,\widetilde z_2$ the two zeroes on $\partial \widetilde C$ such that $\overline{\widetilde z_1\widetilde z_2}$ is a saddle connection containing $\widetilde y$. Then it results
$$B^2_{d_{\widetilde q_1}(\widetilde y,\widetilde z_1)}(\widetilde z_1')\cap B^2_{d_{\widetilde q_1}(\widetilde y,\widetilde z_2)}(\widetilde z_2')\subset V_\theta.$$
\end{lem}

\begin{proof}
Let $\widetilde \tau_1^{q_1},\dots,\widetilde \tau_k^{q_1}$ be the saddle connections such that
$$\overline{\widetilde y(\theta \cdot \widetilde y)}=\overline{\widetilde y\widetilde z_1}*\widetilde \tau_1^{q_1}*\dots *\widetilde \tau_k^{q_1}*\overline{(\theta\cdot \widetilde z_2)(\theta\cdot\widetilde y)}.$$
Then for every point $\widetilde y'\in B^2_{d_{\widetilde q_1}(\widetilde y,\widetilde z_1)}(\widetilde z_1')\cap B^2_{d_{\widetilde q_1}(\widetilde y,\widetilde z_2)}(\widetilde z_2')$ it results
$$d_{\widetilde \sigma_2}(\widetilde y',\theta\cdot \widetilde y')\le d_{\widetilde \sigma_2}(\widetilde y',\widetilde z_1')+\sum\limits_{i=1,\dots,k}\hat l_{\widetilde \sigma_2}(\tau_i^q)+ d_{\widetilde \sigma_2}(\theta \cdot \widetilde y',\theta\cdot \widetilde z_2')= $$$$= d_{\widetilde \sigma_2}(\widetilde y',\widetilde z_1')+\sum\limits_{i=1,\dots,k}\hat l_{\widetilde \sigma_2}(\tau_i^{q_1})+ d_{\widetilde \sigma_2}(\widetilde y', \widetilde z_2')\le d_{\widetilde q_1}(\widetilde y,\widetilde z_1)+\sum\limits_{i=1,\dots,k} l_{\widetilde q_1}(\tau_i^{q_1})+ d_{\widetilde q_1}(\widetilde y, \widetilde z_2)=$$$$=d_{\widetilde q_1}(\widetilde y,\widetilde z_1)+\sum\limits_{i=1,\dots,k} l_{\widetilde q_1}(\tau_i^{q_1})+ d_{\widetilde q_1}(\theta\cdot\widetilde y, \theta\cdot \widetilde z_2)=d_{\widetilde q_1}(\widetilde y,\theta\cdot \widetilde y)$$
and consequently $B^2_{d_{\widetilde q_1}(\widetilde y,\widetilde z_1)}(\widetilde z_1')\cap B^2_{d_{\widetilde q_1}(\widetilde y,\widetilde z_2)}(\widetilde z_2')\subset V_\theta$. 
\end{proof}

We are now ready to prove the case of the second type of intersections.
\begin{prop}
For every $\theta\in \Theta(\widetilde p_{n+1})$ and $\widetilde x_1,\widetilde x_2\in \mathcal{X}(\widetilde p_{n+1})$ it results
$$V_\theta\cap B^2_{d_{\widetilde q_1}(\widetilde x_1,\widetilde p_{n+1})}(\widetilde x_1')\cap B^2_{d_{\widetilde q_1}(\widetilde x_2,\widetilde p_{n+1})}(\widetilde x_2')\neq \emptyset.$$ 
\end{prop}

\begin{proof}
Let $\widetilde C$ be the lifting to $\widetilde X_1$ of the flat cylinder corresponding to $\theta$.\\ 
We will first consider the case $\widetilde x_1\not\in \widetilde C$ and $\widetilde x_2\not\in \widetilde C$, since it is the more complicated one.\\  
We define the following points $\widetilde z_1,\widetilde z_2\in \widetilde X_1$:
$$\widetilde z_1:=\overline{\widetilde p_{n+1}\widetilde x_1}\cap \partial \widetilde C,\quad \quad \widetilde z_2:=\overline {\widetilde p_{n+1}\widetilde x_2}\cap \partial \widetilde C.$$
Consider the following two cases:
\begin{itemize}
\item $\overline{\widetilde x_1\widetilde x_2}$ does not traverse $\widetilde C$.

\begin{figure}[h!]
 \centering
  \includegraphics[scale=0.45]{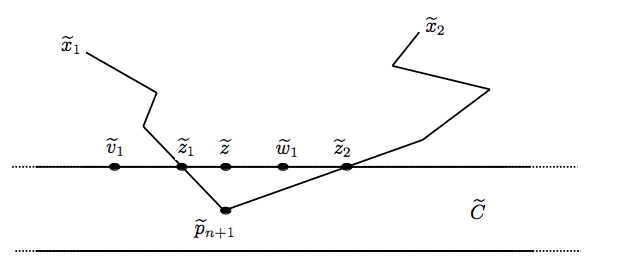}
 \caption{The case $\overline{\widetilde x_1\widetilde x_2}$ does not traverse $\widetilde C$}
 \end{figure}

There is a point $\widetilde z\in \overline{\widetilde z_1\widetilde z_2}$ (eventually equal to $\widetilde z_1$ or $\widetilde z_2$) such that $\widetilde z\in \partial \widetilde C$ and
$$d_{\widetilde q_1}(\widetilde z,\widetilde z_i)\le d_{\widetilde q_1}(\widetilde z_i,\widetilde p_{n+1}),\quad i=1,2$$  
and consequently
$$d_{\widetilde q_1}(\widetilde z,\widetilde x_i)\le d_{\widetilde q_1}(\widetilde x_i,\widetilde p_{n+1}).$$
Let $\widetilde v_1$ and $\widetilde w_1$ be the two zeroes on $\partial \widetilde C$ such that $\overline{\widetilde v_1\widetilde w_1}$ is a saddle connection containing $\widetilde z$.\\ 
From corollary \ref{corintersection} it follows
$$ \Lambda:=B^2_{d_{\widetilde q_1}(\widetilde v_1,\widetilde z)}(\widetilde v_1')\cap B^2_{d_{\widetilde q_1}(\widetilde w_1,\widetilde z)}(\widetilde w_1')\cap B^2_{d_{\widetilde q_1}(\widetilde x_1,\widetilde z)}(\widetilde x_1')\cap B^2_{d_{\widetilde q_1}(\widetilde x_2,\widetilde z)}(\widetilde x_2')\neq \emptyset.$$
The inequality $d_{\widetilde q_1}(\widetilde z,\widetilde x_i)\le d_{\widetilde q_1}(\widetilde x_i,\widetilde p_{n+1})$ grants
$$\Lambda\subset B^2_{d_{\widetilde q_1}(\widetilde v_1,\widetilde z)}(\widetilde v_1')\cap B^2_{d_{\widetilde q_1}(\widetilde w_1,\widetilde z)}(\widetilde w_1')\cap B^2_{d_{\widetilde q_1}(\widetilde x_1,\widetilde p_{n+1})}(\widetilde x_1')\cap B^2_{d_{\widetilde q_1}(\widetilde x_2,\widetilde p_{n+1})}(\widetilde x_2')$$
and applying the preceding lemma we can finally get
$$\Lambda\subset V_\theta\cap B^2_{d_{\widetilde q_1}(\widetilde x_1,\widetilde p_{n+1})}(\widetilde x_1')\cap B^2_{d_{\widetilde q_1}(\widetilde x_2,\widetilde p_{n+1})}(\widetilde x_2').$$

\item $\overline{\widetilde x_1\widetilde x_2}$ traverses $\widetilde C$.

\begin{figure}[h!]
 \centering
  \includegraphics[scale=0.45]{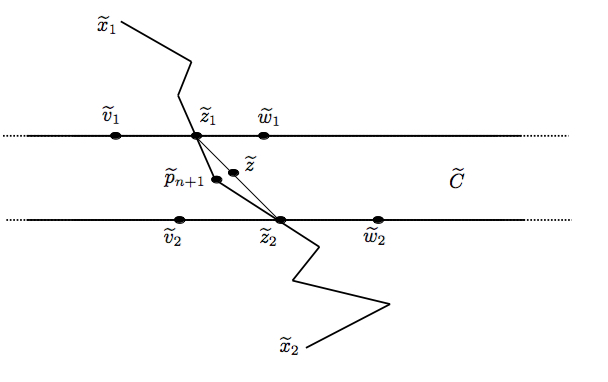}
 \caption{The case $\overline{\widetilde x_1\widetilde x_2}$ traverses $\widetilde C$.}
 \end{figure}
There is a point $\widetilde z\in \overline{\widetilde z_1\widetilde z_2}$ such that $d_{\widetilde q_1}(\widetilde z,\widetilde z_i)\le d_{\widetilde q_1}(\widetilde p_{n+1},\widetilde z_i)$, $i=1,2$.\\
For $i=1,2$, let $\widetilde v_i$ and $\widetilde w_i$ the two zeroes on $\partial \widetilde C$ such that $\overline{\widetilde v_i\widetilde w_i}$ is a saddle connection and $\widetilde z_i\in \overline{\widetilde v_i\widetilde w_i}$.\\
Denote by $\mathcal{X}(\widetilde z_1)$ the set of the zeroes of $\widetilde q_1$ joined to $\widetilde z_1$ by a smooth geodesic of $|\widetilde q_1|$: clearly $\widetilde v_1,\widetilde w_1\in \mathcal{X}(\widetilde z_1)$.\\
Corollary \ref{corintersection} and the previous lemma grant the existence of the following points $\widetilde z_i'\in \widetilde X_2$:
$$\widetilde z_1'\in \left( \bigcap\limits_{\widetilde x\in \mathcal{X}(\widetilde z_1)} B^2_{d_{\widetilde q_1}(\widetilde z_1,\widetilde x)}(\widetilde x')\right)\cap B^2_{d_{\widetilde q_1}(\widetilde z_1,\widetilde x_1)}(\widetilde x_1')\cap B^2_{d_{\widetilde q_1}(\widetilde z_1,\widetilde x_2)}(\widetilde x_2')\subset $$ $$\subset B^2_{d_{\widetilde q_1}(\widetilde z_1,\widetilde x_1)}(\widetilde x_1')\cap B^2_{d_{\widetilde q_1}(\widetilde z_1,\widetilde x_2)}(\widetilde x_2')\cap V_\theta,$$
$$\widetilde z_2'\in B^2_{d_{\widetilde q_1}(\widetilde z_2,\widetilde v_2)}(\widetilde v_2')\cap B^2_{d_{\widetilde q_1}(\widetilde z_2,\widetilde w_2)}(\widetilde w_2')\cap B^2_{d_{\widetilde q_1}(\widetilde z_2,\widetilde z_1)}(z_1')\cap B^2_{d_{\widetilde q_1}(\widetilde z_2,\widetilde x_2)}(\widetilde x_2')\subset$$ $$\subset B^2_{d_{\widetilde q_1}(\widetilde z_2,\widetilde z_1)}(\widetilde z_1')\cap B^2_{d_{\widetilde q_1}(\widetilde z_2,\widetilde x_2)}(\widetilde x_2')\cap V_\theta.$$
The set $V_\theta$ is convex, so it follows $\overline{\widetilde z_1'\widetilde z_2'}\subset V_\theta$ and since $d_{\widetilde \sigma_2}(\widetilde z_1',\widetilde z_2')\le d_{\widetilde q_1}(\widetilde z_1,\widetilde z_2)$, we can choose $\widetilde z'\in \overline{\widetilde z_1'\widetilde z_2'}$ such that $d_{\widetilde \sigma_2}(\widetilde z',\widetilde z_i')\le d_{\widetilde q_1}(\widetilde z,\widetilde z_i)$, $i=1,2$. \\
In this way one finally gets the following inequalities:
$$d_{\widetilde \sigma_2}(\widetilde z',\widetilde x_i')\le d_{\widetilde \sigma_2}(\widetilde z',\widetilde z_i')+d_{\widetilde \sigma_2}(\widetilde z_i',\widetilde x_i')\le $$ $$\le d_{\widetilde q_1}(\widetilde z,\widetilde z_i)+d_{\widetilde q_1}(\widetilde z_i,\widetilde x_i)\le d_{\widetilde q_1}(\widetilde p_{n+1},\widetilde z_i)+d_{\widetilde q_1}(\widetilde z_i,\widetilde x_i)=d_{\widetilde q_1}(\widetilde p_{n+1},\widetilde x_i).$$

\end{itemize}

The case $\widetilde x_1\in \widetilde C$ and $\widetilde x_2\not\in \widetilde C$ can be solved in the same way. Define as before $\widetilde z_2:=\overline{\widetilde p_{n+1}\widetilde x_2}\cap \partial \widetilde C$, then one just has to notice that there always is a point $\widetilde z\in \overline{\widetilde z_2\widetilde x_1}$ such that $d_{\widetilde q_1}(\widetilde z,\widetilde x_1)\le d_{\widetilde q_1}(\widetilde p_{n+1},\widetilde x_1)$ and  $d_{\widetilde q_1}(\widetilde z,\widetilde z_2)\le d_{\widetilde q_1}(\widetilde p_{n+1},\widetilde z_2)$. \\

Finally, if $\widetilde x_1\in \widetilde C$ and $\widetilde x_2\in \widetilde C$ one could notice that it results $\overline{\widetilde{x}_1'\widetilde{x}_2'}\subset V_\theta$. Since $d_{\widetilde \sigma_2}(\widetilde x_1',\widetilde x_2')\le d_{\widetilde q_1}(\widetilde x_1,\widetilde x_2)$, there is a point $\widetilde p_{n+1}'\in \overline{\widetilde{x}_1'\widetilde{x}_2'}$ such that $$d_{\widetilde \sigma_2}(\widetilde p_{n+1}',\widetilde x_i')\le d_{\widetilde q_1}(\widetilde p_{n+1},\widetilde x_i),\quad  i=1,2.$$ 

\end{proof}

\begin{cor}\label{cor529}
For every $\theta\in \Theta(\widetilde p_{n+1})$ and $\widetilde x_i\in \mathcal{X}(\widetilde p_{n+1})$, $i=1,\dots,n$, it results:
$$V_\theta\cap \bigcap\limits_{i=1,\dots,n}B^2_{d_{\widetilde q_1}(\widetilde x_i,\widetilde p_{n+1})}(\widetilde x_i')\neq \emptyset.$$ 
\end{cor}
\begin{proof}
It is a consequence of previous results and Helly's lemma for uniquely geodesic spaces.
\end{proof}

Finally, we can prove that the intersection is not empty also in the last two cases.

\begin{prop}
For every $\theta_1,\theta_2\in \Theta(\widetilde p_{n+1})$ and $\widetilde x\in \mathcal{X}(\widetilde p_{n+1})$, it follows
$$V_{\theta_1}\cap V_{\theta_2}\cap B^2_{d_{\widetilde q_1}(\widetilde x,\widetilde p_{n+1})}(\widetilde x')\neq \emptyset.$$ 
\end{prop}
\begin{proof}
Let $\widetilde C_i$ be the lifting to $\widetilde X_1$ of the flat cylinder corresponding to $\theta_i$, $i=1,2$. \\
We first consider the case $\widetilde x\not\in \widetilde C_1\cup \widetilde C_2$, since it is the more complicated one.\\
We choose the point $\widetilde z$:
$$\widetilde z:=\overline{\widetilde x\widetilde p_{n+1}}\cap \partial(\widetilde C_1\cap \widetilde C_2).$$
Notice that it results $d_{\widetilde q_1}(\widetilde p_{n+1},\widetilde x)\ge d_{\widetilde q_1}(\widetilde z,\widetilde x)$ and suppose $\widetilde z\in \partial \widetilde C_1$.\\ 
Let $\widetilde v_1$ and $\widetilde w_1$ be the two zeroes of $\widetilde q_1$ such that $\overline{\widetilde v_1\widetilde w_1}$ is the saddle connection of $\partial \widetilde C_1$ containing $\widetilde z$.\\ 
Then one gets the following inclusion of sets:
$$B^2_{d_{\widetilde q_1}(\widetilde z,\widetilde v_1)}(\widetilde v_1')\cap B^2_{d_{\widetilde q_1}(\widetilde z,\widetilde w_1)}(\widetilde w_1')\cap B^2_{d_{\widetilde q_1}(\widetilde x,\widetilde z)}(\widetilde x')\cap V_{\theta_2}\subset V_{\theta_1}\cap B^2_{d_{\widetilde q_1}(\widetilde x,\widetilde p_{n+1})}(\widetilde x')\cap V_{\theta_2}$$
and we can conclude applying corollary \ref{cor529}.\\
The case $\widetilde x\in \widetilde C_1$ and $\widetilde x\not\in \widetilde C_2$  can be solved in the same way, choosing $$\widetilde z:=\overline{\widetilde x\widetilde p_{n+1}}\cap \partial\widetilde C_1.$$
Finally, the case $\widetilde x\in \widetilde C_1\cap \widetilde C_2$ is trivial since $\widetilde x'\in V_{\theta_1}\cap V_{\theta_2}$.
\end{proof}

\begin{prop}
For every $\theta_1,\theta_2,\theta_3\in \Theta(\widetilde p_{n+1})$ it follows
$$V_{\theta_1}\cap V_{\theta_2}\cap V_{\theta_3}\neq \emptyset.$$ 
\end{prop}

\begin{proof}
As before, denote by $\widetilde C_i$ the lifting to $\widetilde X_1$ of the flat cylinder corresponding to $\theta_i$, $i=1,2,3$. Up to renumbering the indexes, we can suppose there is a point $\widetilde z\in \partial(\widetilde C_1\cap \widetilde C_2)\cap \widetilde C_3$. \\
Then, for $i=1,2$ let $\widetilde v_i,\widetilde w_i$ be the zeroes of $\widetilde q_1$ on the border of $\widetilde C_i$ such that $\overline{\widetilde v_i\widetilde w_i}$ is a saddle connection and $\widetilde z\in \overline{\widetilde v_i\widetilde w_i}$.\\
Using lemma 3.14 we just have to prove
$$B^2_{d_{\widetilde q_1}(\widetilde z,\widetilde v_1)}(\widetilde v_1')\cap B^2_{d_{\widetilde q_1}(\widetilde z,\widetilde w_1)}(\widetilde w_1')\cap B^2_{d_{\widetilde q_1}(\widetilde z,\widetilde v_2)}(\widetilde v_2')\cap B^2_{d_{\widetilde q_1}(\widetilde z,\widetilde w_2)}(\widetilde w_2')\cap \widetilde V_{\theta_3}\neq \emptyset$$
which is granted by corollary \ref{cor529}.
\end{proof}

This ends the proof of the existence of the desired function $\phi$: if conjecture 3.31 is true, we have described how to obtain the equality $L_F^a(q_1,q_2)=K_F^a(q_1,q_2)$.

\bigskip
\subsection{Proof of theorem 3.11}

The first step towards the proof of theorem 3.11 consists in the characterization of geodesic triangles in $(\widetilde X,d_{\widetilde q})$ (where as before $(X,q)$ is a semi-translation surface and $\pi:(\widetilde X,|\widetilde q|)\rightarrow (X,|q|)$ is a metric universal cover). We will use the following lemma, the proof of which can be found in \cite{St}, theorem 16.1.

\begin{lem}\label{streb}
Let $\widetilde \gamma:[0,1]\rightarrow \widetilde X$ be a locally minimizing geodesic for $d_{\widetilde q}$. It follows
$$d_{\widetilde q}(\widetilde \gamma(0),\widetilde \gamma(1))=l_{\widetilde q}(\widetilde \gamma)$$
that is, $\widetilde \gamma$ is also globally minimizing. Furthermore, $\widetilde \gamma$ is the unique geodesic with these properties.
\end{lem}

Given any triple of points $\widetilde x_1,\widetilde x_2, \widetilde x_3\in \widetilde X$ denote by $T$ the corresponding geodesic triangle for $d_{\widetilde q}$, which is the subset of $\widetilde X$ composed by the three geodesics $\overline{\widetilde x_i\widetilde x_j}$.\\ 
Since $\widetilde X\simeq \mathbb{H}$, it makes sense to define the internal part ${\mathop \Delta\limits^ \circ}$ of $T$: we call \textit{filled geodesic triangle} the set $T\cup {\mathop \Delta\limits^ \circ}$ and we denote it by $\Delta$.\\

Given any planar polygon $P$, we denote by  $d_P$ its \textit{intrinsic Euclidean metric}: for every $x_1,x_2\in P$, we define $d_P(x_1,x_2)$ as the infimum of the lengths, computed with respect to the Euclidean metric, of all paths from $x_1$ to $x_2$ entirely contained in $P$. \\
Every polygon used in the following proofs will be endowed with such intrinsic Euclidean metric.

\begin{prop}
Filled geodesic triangles of $d_{\widetilde q}$ are convex and do not contain zeroes of $\widetilde q$ in their internal part, which is connected.\\
Given a triple of points $\widetilde x_1,\widetilde x_2,\widetilde x_3\in \widetilde X$, the corresponding filled geodesic triangle $\Delta$ can have one dimensional components. For every $i=1,2,3$ we define $\widetilde v_i$ as the point on $\overline{\widetilde x_i\widetilde x_j}\cap \overline{\widetilde x_i\widetilde x_k}$, $i\neq j\neq k$ which has maximum distance with $\widetilde x_i$.\\
If ${\mathop \Delta\limits^ \circ}$ is not empty, then its border is exactly $\overline{\widetilde v_1\widetilde v_2}\cup\overline{\widetilde v_2\widetilde v_3}\cup\overline{\widetilde v_1\widetilde v_3}$ and for every $i=1,2,3$, if $\widetilde x_i\neq \widetilde v_i$, then $\overline{\widetilde x_i\widetilde v_i}$ is the only one dimensional component of $\Delta$ starting from $\widetilde x_i$.\\
The internal angles of $\overline{\mathop \Delta\limits^ \circ}$ in the three points $\widetilde v_i$ are strictly convex, while all other internal angles are concave and less than $2\pi$.\\
Finally, every filled geodesic triangle for $d_{\widetilde q}$ is isometric to a planar polygon, which eventually could be degenerate (one dimensional) or with at most three one dimensional components.
\end{prop}

\begin{figure}[h!]
 \centering
  \includegraphics[scale=0.5]{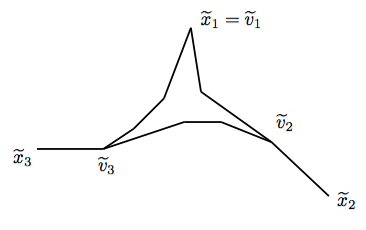}
 \caption{An example of a filled geodesic triangle $\Delta$.}
 \end{figure}

\begin{proof}

By lemma \ref{streb}, if $\overline{\widetilde x_i \widetilde x_j}$ and $\overline{\widetilde x_i\widetilde x_k}$ intersect in a point $\widetilde p\neq \widetilde x_i$, then they must coincide over all $\overline{\widetilde p\widetilde x_i}$. It follows that ${\mathop \Delta\limits^ \circ}$ is connected and its border is $\overline{\widetilde v_1\widetilde v_2}\cup\overline{\widetilde v_1\widetilde v_3}\cup\overline{\widetilde v_2\widetilde v_3}$.\\

Suppose ${\mathop \Delta\limits^ \circ}\neq \emptyset$ and denote by $\alpha_1$ the internal angle of $\overline{{\mathop \Delta\limits^ \circ}}$ in $\widetilde v_1$: we prove $\alpha_1< \pi$. \\
Let $\alpha_{12}$ be the angle in $\widetilde v_1$ determined by $\overline{\widetilde x_1\widetilde v_1}$ and $\overline{\widetilde v_1\widetilde v_2}$ completely outside $\Delta$ and let $\alpha_{13}$ be the angle in $\widetilde v_1$ determined by $\overline{\widetilde x_1\widetilde v_1}$ and $\overline{\widetilde v_1\widetilde v_3}$ completely outside $\Delta$. Clearly it results $\alpha_{12}\ge \pi$ and $\alpha_{13}\ge \pi$: if $\alpha_1\ge \pi$ then lemma \ref{streb} would imply $\widetilde v_1\in \overline{\widetilde v_2\widetilde v_3}$ and consequently $\mathop \Delta\limits^ \circ=\emptyset$. In the same way one proves that the internal angles of $\overline{{\mathop \Delta\limits^ \circ}}$ in $\widetilde v_2,\widetilde v_3$ must be strictly convex.\\

If $\widetilde v$ is a zero of $\widetilde q$ in the border of $\mathop \Delta\limits^ \circ$, $\widetilde v\neq \widetilde v_1,\widetilde v_2,\widetilde v_3$, the internal angle $\beta_{\widetilde v}$ of $\overline{{\mathop \Delta\limits^ \circ}}$ in $\widetilde v$ must be concave, and we now also prove $\beta_{\widetilde v}<2\pi$.\\
Let $\widetilde v\in \overline{\widetilde v_1\widetilde v_2}$, and suppose by contradiction $\beta_{\widetilde v}\ge 2\pi$. Let $\tau_i$, $i=1,2$, be the angle in $\widetilde v$ determined by $\overline{\widetilde v_3\widetilde v}$ and $\overline{\widetilde v\widetilde v_i}$ inside $\Delta$. Since $\beta_{\widetilde v}\ge 2\pi$, it must follow $\tau_1\ge \pi$ or $\tau_2\ge \pi$. Suppose $\tau_1\ge 1$, then $\overline{\widetilde v_3\widetilde v_1}$ would be a concatenation of $\overline{\widetilde v_1\widetilde v}$ and $\overline{\widetilde v\widetilde v_3}$, implying $\widetilde v=\widetilde v_1$. This last equality contradicts the previous assumption $\widetilde v\neq \widetilde v_1,\widetilde v_2,\widetilde v_3$.\\

Finally, suppose by contradiction that one or more zeroes $\widetilde z_j$ of $\widetilde q$ are contained in ${\mathop \Delta\limits^ \circ}$. 
Denote by $\alpha_i,\theta_j,\beta_k$ the internal angles of $\overline{{\mathop \Delta\limits^ \circ}}$ respectively in $\widetilde v_i,\widetilde z_j,\widetilde w_k$, where $\widetilde w_k$ is a zero of $\widetilde q$ on the border of $\mathop \Delta\limits^ \circ$.\\
Applying Gauss-Bonnet formula on $\overline{{\mathop \Delta\limits^ \circ}}$ one gets:
$$\sum_i(\pi-\beta_i)+(\pi-\alpha_1)+(\pi-\alpha_2)+(\pi-\alpha_{3})=2\pi+\sum_j(\theta_j-2\pi).$$
From what we have proved it follows 
$$\sum_k(\pi-\beta_k)\le 0,\quad (\pi-\alpha_1)+(\pi-\alpha_2)+(\pi-\alpha_{3})<3 \pi$$ 
and consequently we now get 
$$2\pi+\sum_j(\theta_j-2\pi)<3\pi.$$
The total angle in $\widetilde z_j\in \mathop \Delta\limits^ \circ$ must be greater than or equal to $3\pi$, but this contradicts the last inequality.\\

In order to prove that $\Delta$ is isometric to a planar polygon endowed with its intrinsic Euclidean metric it is clearly sufficient to prove that $ \overline{{\mathop \Delta\limits^ \circ}}$ is isometric to a planar polygon.\\
Let $Dev:(\widetilde X,|\widetilde q|)\rightarrow \mathbb{R}^2$ be the developing map (for a precise definition see for example \cite{Tr2}), notice that, if $Dev$ is injective on a point $\widetilde v$ on the border of $ \overline{{\mathop \Delta\limits^ \circ}}$, then the internal angle of $\overline{{\mathop \Delta\limits^ \circ}}$ in $\widetilde v$ coincides with the internal angle of $Dev(\overline{{\mathop \Delta\limits^ \circ}})$ in $Dev(\widetilde v)$.\\
We will prove that $Dev:\overline{{\mathop \Delta\limits^ \circ}}\rightarrow \mathbb{R}^2$ is injective, or equivalently that $Dev(\overline{{\mathop \Delta\limits^ \circ}})$ is a simple polygon (not self-intersecting).\\
Suppose by contradiction that $Dev$ is not injective. We divide two cases:

\begin{enumerate}
\item $Dev$ is not injective on any of the points $\widetilde v_i$. Then denote by $P_1$ be the simple polygon identified by the external border of $Dev(\overline{{\mathop \Delta\limits^ \circ}})$. Notice that internal angles of $P_1$ can correspond to internal angles of $\overline{{\mathop \Delta\limits^ \circ}}$ or can be originated by overlays on points where $Dev$ fails to be injective. Internal angles of the latter kind must be strictly concave and consequently convex internal angles of $P_1$ must correspond to convex internal angles of $\overline{{\mathop \Delta\limits^ \circ}}$. Since $P_1$ is simple, it must have at least three strictly convex internal angles. It would follow that $\overline{{\mathop \Delta\limits^ \circ}}$ must have at least six strictly convex internal angles: the three angles $\alpha_i$ plus the angles which correspond to the three strictly convex internal angles of $P_1$. This fact clearly contradicts the hypothesis.
\item $Dev$ is injective on $\widetilde v_1$. Then there is a polygon $P_2\subset Dev(\overline{{\mathop \Delta\limits^ \circ}})$ which is maximal with respect to inclusion on the set of polygons $\{Q\}$ such that 
\begin{itemize}
\item $Q\subset  Dev(\overline{{\mathop \Delta\limits^ \circ}})$,
\item $Dev(\widetilde v_1)$ is a vertex of $Q$,
\item $Dev$ is injective on $Dev^{-1}(Q)$.
\end{itemize}
Let $P_0$ be the simple polygon identified by the external border of $Dev(\overline{{\mathop \Delta\limits^ \circ}})$ and define $P_1:=\overline{P_0\setminus P_2}$ (see figure 4 for an example.). As before, convex internal angles of $P_1$ must correspond to convex internal angles of $\overline{{\mathop \Delta\limits^ \circ}}$.

\begin{figure}[h!]
 \centering
  \includegraphics[scale=0.4]{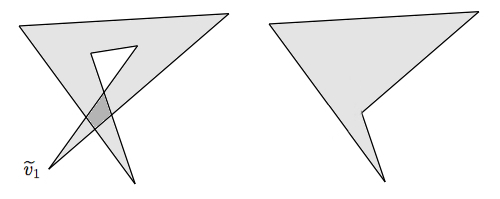}
 \caption{ On the left there is an example of $Dev(\overline{{\mathop \Delta\limits^ \circ}})$ we want to exclude. On the right there is the corresponding polygon $P_1$.}
 \end{figure}
It would follow that $\overline{{\mathop \Delta\limits^ \circ}}$ must have at least four strictly convex internal angles: $\alpha_1$ plus the angles which correspond to the three strictly convex internal angles of $P_1$. This fact clearly contradicts the hypothesis.

\end{enumerate}

Finally, convexity of $\Delta$ follows from the fact that, given any pair $\widetilde x,\widetilde y\in \Delta$, the geodesic for the intrinsic Euclidean metric connecting them is also a locally minimizing geodesic for $d_{\widetilde q}$ and consequently also globally minimizing.

\end{proof}

We now go back to consider the fundamental domain $P$ defined in the preceding section.\\
Given the point $\widetilde p_{n+1}\in P$ and the three points $\widetilde x_1,\widetilde x_2,\widetilde x_3\in \mathcal{X}(\widetilde p_n)$ corresponding to the centers of the closed balls, we consider the filled geodesic triangle $\Delta$ for $(\widetilde X_1,d_{\widetilde q_1})$ with vertices $\widetilde x_1,\widetilde x_2,\widetilde x_3$. Following the characterization of the previous proposition, we divide two cases:
\begin{enumerate}
\item $\widetilde p_{n+1}\in \Delta$, then, since the three geodesics $\overline{\widetilde p_{n+1}\widetilde x_i}$ are smooth and do not contain other zeroes of $\widetilde q_1$, it follows that $\Delta$ can not have one dimensional components.
\item $\widetilde p_{n+1}\not \in \Delta$, then $\Delta$ can have one dimensional components and even be a degenerate polygon (one dimensional). 
\end{enumerate}

Denote by $\Delta'$ the filled geodesic triangle for $(\widetilde X_2,d_{\widetilde \sigma_2})$ with vertices $\widetilde x_1',\widetilde x_2',\widetilde x_3'$. \\
Again, we divide three cases:
\begin{enumerate}[label=(\roman*)]
\item $\Delta'$ is not one dimensional, but can have at most three one dimensional components,
\item $\Delta'$ is one dimensional and it is not possible to renumber the vertices in order to obtain $x_3'\in \overline{x_1'x_2'}$,
\item $\Delta'$ is one dimensional and it is possible to renumber the vertices in order to obtain $x_3'\in \overline{x_1'x_2'}$.
\end{enumerate}

Combining them, we have a total of six cases we need to care care of.\\
In cases (1,i),(1,ii),(1,iii) our goal is to find a point $\widetilde p_{n+1}'\in \Delta'$ such that
$$d_{\widetilde q_1}(\widetilde x_i,\widetilde p_{n+1})\ge d_{\widetilde \sigma_2}(\widetilde x_i',\widetilde p_{n+1}') \text{ for } i=1,2,3.$$
In the remaining cases (2,i),(2,ii),(2,iii) we will use the orthogonal projection on convex sets in $Cat(0)$ spaces:
$$pr:(\widetilde X_1,d_{\widetilde q_1})\rightarrow \Delta,$$
where the image $pr(\widetilde x)$ of every point $\widetilde x\in \widetilde X_1$ is defined as the unique point such that 
$$d_{\widetilde q_1}(\widetilde x,pr(\widetilde x))=\inf \limits_{\widetilde y\in \Delta}d_{\widetilde q_1}(\widetilde x,\widetilde y).$$
The projection $pr$ does not increase distances (for a proof and a list of other properties of $pr$ one could see \cite{BH}, proposition 2.4, page 176) and in particular it results 
$$d_{\widetilde q_1}(\widetilde x_i,\widetilde p_{n+1})\ge d_{\widetilde q_1}(\widetilde x_i, pr(\widetilde p_{n+1})) \text{ for } i=1,2,3.$$
Then, we will look for a point $\widetilde p_{n+1}'\in \Delta'$ such that
$$d_{\widetilde q_1}(\widetilde x_i, pr(\widetilde p_{n+1}))\ge d_{\widetilde \sigma_2}(\widetilde x_i',\widetilde p_{n+1}') \text{ for } i=1,2,3.$$
We chose to confront distances with $pr(\widetilde p_{n+1})$ instead of $\widetilde p_{n+1}$ because in the following procedures it will be crucial to always consider points inside $\Delta$.\\

Cases (1,ii),(1,iii),(2,ii) and (2,iii) are easily solvable. We will prove only case (1,ii), since the others are almost identical. \\
If $\Delta'$ is one dimensional, then it always contains a vertex $\widetilde v'$ such that $$\widetilde v'\in \overline{\widetilde x_1'\widetilde x_2'}\cap \overline{\widetilde x_1'\widetilde x_3'}\cap \overline{\widetilde x_2'\widetilde x_3'}.$$

\begin{figure}[h!]
 \centering
  \includegraphics[scale=0.4]{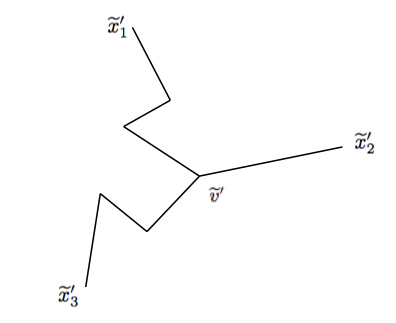}
 \caption{An example of vertex $\widetilde v'\in \Delta'$.}
 \end{figure}

If, for every index $i=1,2,3$, it results $d_{\widetilde \sigma_2}(\widetilde x_i',\widetilde v')\le d_{\widetilde q_1}(\widetilde x_i,\widetilde p_{n+1})$, then we can choose $\widetilde p_{n+1}'=\widetilde v'$.\\
If, up to renumbering the indexes, it results $d_{\widetilde \sigma_2}(\widetilde x_1',\widetilde v')>d_{\widetilde q_1}(\widetilde x_1,\widetilde p_{n+1})$, we choose $\widetilde p_{n+1}'$ to be the point on $\overline{\widetilde x_1'\widetilde v'}$ such that $d_{\widetilde q_1}(\widetilde x_1,\widetilde p_{n+1})=d_{\widetilde \sigma_2}(\widetilde x_1',\widetilde p_{n+1}')$.\\
Then it will follow $d_{\widetilde \sigma_2}(\widetilde x_i',\widetilde p_{n+1}')\le d_{\widetilde q_1}(\widetilde x_i,\widetilde p_{n+1})$ for $i=2,3$, since 
$$d_{\widetilde \sigma_2}(\widetilde x_i',\widetilde p_{n+1}')=d_{\widetilde \sigma_2}(\widetilde x_1',\widetilde x_i')-d_{\widetilde q_1}(\widetilde x_1,\widetilde p_{n+1})\le d_{\widetilde q_1}(\widetilde x_1,\widetilde x_i)-d_{\widetilde q_1}(\widetilde x_1,\widetilde p_{n+1})\le d_{\widetilde q_1}(\widetilde p_{n+1},\widetilde x_i).$$

In case (2,i) it will always be possible to suppose $\Delta$ does not have one dimensional components, since

\begin{itemize}
\item if $pr(\widetilde p_{n+1})$ is on a one dimensional component $\overline{\widetilde{x}_1\widetilde v_1}$ of $\Delta$ then it suffices to choose $\widetilde p_{n+1}'\in  \overline{\widetilde x_1'\widetilde v_1'}$ such that $$d_{\widetilde \sigma_2}(\widetilde p_{n+1}',\widetilde x_1')\le d_{\widetilde q_1}(pr(\widetilde p_{n+1}),\widetilde x_1) \text{ and } d_{\widetilde \sigma_2}(\widetilde p_{n+1}',\widetilde v_1')\le d_{\widetilde q_1}(pr(\widetilde p_{n+1}),\widetilde v_1).$$
\item Otherwise, $pr(\widetilde p_{n+1})\in \overline{{\mathop \Delta\limits^ \circ}}$, where $\overline{{\mathop \Delta\limits^ \circ}}$ corresponds to the filled geodesic triangle of vertices $\widetilde v_1,\widetilde v_2,\widetilde v_3$ (which are the vertices with strictly convex internal angle as in proposition 3.20). \\
In this case one can choose $\widetilde p_{n+1}'$ such that $d_{\Delta}(\widetilde v_i,\widetilde p_{n+1})\ge d_{\Delta'}(\widetilde v_i',\widetilde p_{n+1}')$.\\ 
In this way one obtains 
$$d_{\widetilde \sigma_2}(\widetilde p_{n+1}',\widetilde x_i')\le d_{\widetilde \sigma_2}(\widetilde p_{n+1}',\widetilde v_i')+d_{\widetilde \sigma_2}(\widetilde v_i',\widetilde x_i')\le$$ $$\le d_{\widetilde q_1}(pr(\widetilde p_{n+1}),\widetilde v_i)+d_{\widetilde q_1}(\widetilde v_i,\widetilde x_i)=d_{\widetilde q_1}(pr(\widetilde p_{n+1}),\widetilde x_i)$$ for $i=1,2,3$ as desired.
\end{itemize} 

The rest of the section will be devoted to the explanation of our method to find $\widetilde p_{n+1}$ in cases (1,i) and (2,i). As we anticipated it will depend on following theorem 3.21 and conjecture 3.31.\\

Consider the two previously defined filled geodesic triangles of vertices respectively $\widetilde x_1,\widetilde x_2,\widetilde x_3$ and $\widetilde x_1',\widetilde x_2',\widetilde x_3'$. Zeroes on the border of $\Delta$ can change position in $\Delta'$ and in particular the following things can happen:

\begin{enumerate}[label=(\roman*)]
\item if $\widetilde z\in\overline{\widetilde x_i\widetilde x_j}$, then it can happen $\widetilde z'\in \overline{\widetilde x_i'\widetilde x_k'}$,
\item a zero $ \widetilde z$ on the border of $\Delta$  can be such that $\widetilde z'\not\in \Delta'$, 
\item a zero $ \widetilde z'$ on the border of $\Delta'$  can be such that $\widetilde z\not\in \Delta$. 
\end{enumerate}

Every time case (ii) is verified, we consider the previously defined orthogonal projection on convex sets in $Cat(0)$ spaces
$$pr:(\widetilde X_2,d_{\widetilde \sigma_2})\rightarrow \Delta'$$
and take into account the point $pr(\widetilde z')\in \partial\Delta'$. Then it will follow
$$d_{\widetilde \sigma_2}(pr(\widetilde z'),\widetilde x_i')\le d_{\widetilde \sigma_2}(\widetilde z',\widetilde x_i')\le d_{\widetilde q_1}(\widetilde z,\widetilde x_i)$$
for $i=1,2,3$ and 
$$d_{\widetilde \sigma_2}(pr(\widetilde z'),\widetilde w')\le d_{\widetilde \sigma_2}(\widetilde z',\widetilde w')\le d_{\widetilde q_1}(\widetilde z,\widetilde w)$$ 
for every zero $\widetilde w'$ on the border of $\Delta'$.\\
In the following construction we will need to consider, for every point on the border of $\Delta$, a corresponding point on the border of $\Delta'$. For this reason, by abuse of notation, every time previous case (ii) is verified we will denote the point $pr(\widetilde z')$ simply by $\widetilde z'$ and consider it the point on the border of $\Delta'$ corresponding to $\widetilde z$.\\
Notice that proceeding in this way $\Delta'$ could end up having two or more coinciding vertices: this will not be a problem.\\

From now on it will be more convenient to consider filled geodesic triangles $\Delta$ and $\Delta'$ exclusively as planar polygons endowed respectively with the intrinsic Euclidean metrics $d_\Delta$ and $d_{\Delta'}$. For this reason we will consider zeroes on the border of $\Delta$ simply as vertices of the polygon. Furthermore, in order to lighten up the notation, vertices will be denoted without the  overlying tilde.\\
For every couple of points $u,v\in \Delta$ we will denote by $\overline{uv}$ the geodesic for $d_\Delta$ connecting them. Given any two points $u',v'\in \Delta'$, we will denote by $\overline{u'v'}$ the geodesic for $d_{\Delta'}$ connecting them. \\

We will initially consider the case there is a function $$\iota :Vertices(\Delta)\rightarrow Vertices(\Delta')$$ which to every vertex $z$ of $\Delta$ associates a vertex $\iota(z)=z'$ of $\Delta'$ in such a way that vertices of $\Delta$ and of $\iota (Vertices(\Delta))$ are \textit{disposed in the same order}. This means that: 
\begin{itemize}
\item for every vertex $z$ of $\Delta$, if $z\in \overline{x_ix_j}$, then $z'\in \overline{x_ix_j}$,
\item for every couple of vertices $z_1,z_2\in \overline{x_ix_j}$, if $d_\Delta(x_i,z_1)<d_{\Delta}(x_i,z_2)$, then $d_{\Delta'}(x_i',z_1')\le d_{\Delta'}(x_i',z_2')$.
\end{itemize}
We will summarize this condition on the vertices of $\Delta$ and $\Delta'$ saying that the common vertex of $\Delta$ and $\Delta'$ have the same order.\\
We noticed that, given two vertices $v_1,v_2$ of $\Delta$, it can happen that their corresponding vertices of $\Delta'$ coincide as points on $\partial\Delta'$. For a reason which will be clear in the following proofs, we will consider $v_1'$ and $v_2'$ as distinct vertices of $\Delta'$ which are at distance zero on $\partial \Delta'$: we will refer to them as multiple vertices. \\
We can thus suppose the function $\iota$ is always injective and the number of vertices of $\Delta'$ is always greater than or equal to the number of vertices of $\Delta$.\\

We underline again an important hypothesis on distances between vertices of $\Delta$ and $\Delta'$: for every pair of vertices $u,v$ of $\Delta$ such that $\overline{uv}$ is smooth it results
$$d_\Delta(u,v)\ge d_{\Delta'}(u',v').$$ 
This fact clearly implies the same inequality also in case $\overline{uv}$ is a concatenation of smooth segments.\\

The following theorem is our fundamental tool to find the desired point $\widetilde p_{n+1}'$. 
\begin{thm}\label{lemshort1}
Suppose the number of vertices of $\Delta'$ is greater than or equal to the number of vertices of $\Delta$ and that the common vertices have the same order, in the sense we explained earlier. Suppose furthermore that $\Delta'$ can have one dimensional components. \\
Then there is a 1-Lipschitz map $f:\Delta\rightarrow \Delta'$ (with respect to the intrinsic Euclidean metrics of the polygons) such that: 
$$f(z)=z'$$ 
for every vertex $z$ of $\Delta$.
\end{thm}
Clearly, given any point $\widetilde p_{n+1}\in \Delta$, we will set the point $\widetilde p_{n+1}'$ to be $f(\widetilde p_{n+1})$. \\
Instead of proving theorem \ref{lemshort1} directly, we will prove the following theorem \ref{lemshort2} which will then imply theorem \ref{lemshort1}. The reason for this choice will be made clear in the proof of theorem \ref{lemshort2} and in particular by the example of figure 11.\\

Given any planar polygon $P$ with $n\ge 3$ vertices, we will say that $P'$ is a \textit{degenerate polygon comparable with P} if $P'$ is obtained connecting planar polygons through common vertices or one dimensional components and furthermore all the following conditions are satisfied.
\begin{enumerate}[label=(\roman*)]
\item $P'$ is connected, simply connected, can be embedded in $\mathbb{R}^2$ and contains at least one planar polygon.
\item Every planar polygon of $P'$ is linked (by shared vertices or one dimensional components) to at most other two planar polygons of $P'$. The degenerate polygon $P'$ can have one dimensional components which are linked to just one planar polygon of $P'$ (as polygons $\Delta'$ corresponding to geodesic triangles of $d_{\widetilde q}$ do).
\item There is an injective function $\iota : Vertices(P)\rightarrow Vertices(P')$, which to every vertex $z$ of $P$ associates a unique vertex $z'$ of $P'$. \\
Given two vertices $z_1,z_2$ of $P$, their corresponding vertices of $P'$ can coincide as points on $\partial P'$: we will consider $z_1',z_2'$ as distinct vertices of $P'$ at distance zero on $\partial P'$ and refer to them as multiple vertices.\\ 
Consequently, the total number of vertices of $P'$ is $m\ge n$.
\item For every pair of vertices $z_1,z_2$ of $P$ it results $$d_{P}(z_1,z_2)\ge d_{P'}(z_1',z_2').$$
\item If $y'$ is a vertex of $P'$ which does not correspond to any vertex of $P$ and $y'$ does not lie on a one dimensional component, then the internal angle at $y'$ is:
\begin{itemize}
\item convex, if $y'$ is a shared vertex of two planar polygons of $P'$ or from $y'$ starts a one dimensional component, 
\item concave, otherwise.
\end{itemize}
 A vertex $y'$ which does not correspond to any vertex of $P$ can also lie on a one dimensional component, but it can not be at the extremity which is not connected to a planar polygon. 
\item The vertices of $P$ and of $\iota(Vertices(P))$ \textit{are disposed in the same order} in the following sense. There is a continuous, surjective function $\tau:[0,1]\rightarrow \partial P'$ such that $\tau(0)=z'\in \iota(Vertices(P))$ and  for every $x'\in \partial P'$ the cardinality of $\tau^{-1}(x')$ is:
\begin{itemize}
\item two, if $x'$ is a shared vertex of two planar polygons of $P'$ or $x'$ in on a one dimensional component, 
\item one, otherwise.
\end{itemize}
Then one can choose a parametrization $\gamma:[0,1]\rightarrow \partial P$ of $\partial P$ such that $\gamma(0)=z$ and $\gamma$ and $\tau$ meet respectively the vertices of $P$ and of $\iota(Verices(P))$ in the same order (up to removing one copy of the vertices which $\tau$ meets twice). 
\end{enumerate}
In figure 6 there are some example which will clarify our definition of degenerate polygons comparable with $P$ and of condition (vi).
\begin{figure}[h!]
 \centering
  \includegraphics[scale=0.5]{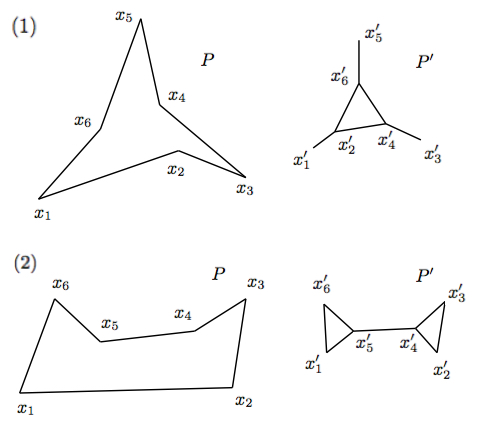}
 \caption{In example (1) one can find $\tau$ such that it encounters the vertices of $P'$ in the order $x_1',x_2',x_4',x_3',x_4',x_6',x_5',x_6',x_2'$. One then discards the first copy of $x_4'$ and $x_6'$ and the last copy of $x_2'$.
In example (2) one can find $\tau$ such that it encounters the vertices of $P'$ in the order $x_1',x_5',x_4',x_2',x_3',x_4',x_5',x_6'$. One then discards the first copy of $x_5'$ and $x_4'$.}
 \end{figure}

As it is easily verifiable, polygons $\Delta$ and $\Delta'$ as in the hypothesis of theorem \ref{lemshort1} satisfy all previous conditions.\\

If $u,v$ are vertices of $P$, $\overline{uv}$ is smooth and lies entirely on the border of $P$ then we will call $\overline{uv}$ a $side$ of $P$ and sometimes denote it simply by $\gamma$. If $w,z$ are vertices of $P$, $\overline{wz}$ is smooth and $\overline{wz}\cap \partial P=\{w,z\}$, then we will call $\overline{uv}$ a \textit{smooth diagonal} of $P$ and sometimes denote it simply by $d$. If $\overline{wz}$ is a concatenation of segments and is not entirely contained in the border of $P$ we call $\overline{wz}$ a diagonal of $P$ and sometimes denote it with the same symbol $d$. \\
We define sides $\gamma'$ of $P'$ in the same way. A diagonal $d'$ of $P'$ is a geodesic $\overline{u'v'}$ such that $\overline{uv}$ is a diagonal of $P$. In particular one should notice that: 
\begin{itemize}
\item a diagonal $d'$ of $P'$ can be entirely contained in a one dimensional component, 
\item $\overline{u'v'}$ can be a diagonal of $P'$ only if $u',v'\in \iota(Vertices(P))$.
\end{itemize}
Given sides $\gamma,\gamma'$ and diagonals $d,d'$, we will denote by $l(\gamma),l(\gamma'),l(d),l(d')$ their lengths (of which the first and the third are computed with respect to $d_{P}$ and the second and the fourth with respect to $d_{P'}$).\\

Before starting the proof, we feel it is necessary to anticipate why we decided to consider such a complicated set of degenerate polygons. The short answer is that the set of degenerate polygons $P'$ comparable to $P$ is closed with respect to the operation of \textit{cutting along a diagonal $d'$ of $P'$}, operation which is crucial in the proof of theorem \ref{lemshort2}. We will further clarify this concept in the proof.

\begin{thm} \label{lemshort2}
Let $P$ be a planar polygon with $n\ge 3$ vertices and $P'$ a degenerate polygon which is comparable with $P$ in the sense we just explained.
Then there is a 1-Lipschitz map $f:P\rightarrow P'$ (with respect to the intrinsic Euclidean metrics of the polygons) such that 
$$f(z)=z'$$ 
for every other vertex $z$ of $P$.
\end{thm}

The idea of the proof will be to turn $P'$ into the polygon $P$ through a finite number of steps, called \textit{elementary steps}, which will modify lengths of sides and diagonals of $P'$. Each elementary step will provide us of a 1-Lipschitz map: the final 1-Lipschitz map $f$ will be the composition of all intermediate 1-Lipschitz maps. Of course, all intermediate polygons will be endowed with the corresponding intrinsic Euclidean metric and the intermediate maps will have Lipschitz coefficient 1 with respect to those metrics.\\
We specify that intermediate polygons obtained through elementary steps can fail to be planar and just be \textit{generalized polygons}: a generalized polygon is a polygon which is obtained gluing planar polygons along sides of the same length and which can not be embedded in $\mathbb{R}^2$. For any generalized polygon it still makes sense to define the intrinsic Euclidean metric.\\
Given any generalized polygon $Q$ and a vertex $v$ of $Q$, in the following proofs we will denote by $\alpha_{v}$ the internal angle of $Q$ at $v$.\\

We will now define the two types of elementary steps we will use. In order to make the definition easier, we will first make the assumption $P'$ does not have one dimensional components. \\

\textbf{Elementary step of type one}: \\
If $\gamma'$ is a side of $P'$ such that $l(\gamma')<l(\gamma)$ then though an elementary step of type one on the side $\gamma'$ of $P'$ it is possible to obtain a polygon $\hat P$ and a 1-Lipschitz map $\phi: \hat P\rightarrow P'$ such that: 
\begin{itemize}
\item $l(\gamma) \ge l(\hat \gamma)>l(\gamma')$ (where $\hat \gamma$ denotes the side of $\hat P$ corresponding to $\gamma'$) and all other sides of $\hat P$ are of the same length of the corresponding sides of $P'$,
\item all the diagonals $\hat d$ of $\hat P$ are such that $l(d)\ge l(\hat d)\ge l(d')$.
\end{itemize}
We presently explain how the elementary step of type one on the side $\gamma'=\overline{x'y'}$ of $P'$ is performed.\\
Let $z'$ be another vertex of $P'$ such that $\overline{x'z'}$ and $\overline{y'z'}$ are sides or smooth diagonals of $P'$ (it is always possible to suppose the existence of such $z'$). Let $P(x',y',z')\subset P'$ be the triangle of vertices $x',y',z'$, the set $\overline{P'\setminus P(x',y',z')}$ consists of a number of polygons $Q'_i$ which varies between zero and two. \\
Denote by $P(\hat x,\hat y,\hat z)$ the triangle obtained from $P(x',y',z')$ increasing $d_{P'}(x',y')$ and without changing $d_{P'}(x',z')$ and $d_{P'}(y',z')$. \\

\begin{figure}[h!]
 \centering
  \includegraphics[scale=0.4]{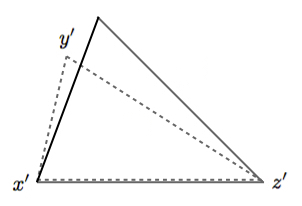}
 \caption{An example of an elementary step of type one: $P(x',y',z')$ is the triangle drawn with a dashed line, while $P(\hat x,\hat y,\hat z)$ is the triangle drawn with a continuous line.}
 \end{figure}

The polygon $\hat P$ is then obtained gluing back on the sides of $P(\hat x,\hat y,\hat z)$ the corresponding polygons $Q'_i$. \\
There is a 1-Lipschitz map $\phi_1:P(\hat x,\hat y,\hat z)\rightarrow P(x',y',z')$ which is the identity on the two sides whose length is not increased. The map $\phi_1$ can then be extended to a 1-Lipschitz map $\phi:\hat P\rightarrow P'$ by defining it as the identity on $Q'_i$.\\

We will also consider degenerate elementary steps of type one, in which $P(x',y',z')$ is one dimensional and is then turned into a triangle. This will happen for example in case of coinciding vertices, which correspond to sides of length zero.\\

\textbf{Elementary step of type two}: \\
If $d'$ is a smooth diagonal of $P'$ such that $l(d')<l(d)$ then through an elementary step of type two on the diagonal $d'$ of $P'$ it is possible to obtain a polygon $\hat P$ and a 1-Lipschitz map $\psi:\hat P\rightarrow P'$ such that: 
\begin{itemize}
\item all sides of $\hat P$ have the same length of the corresponding sides of $P'$,
\item all diagonals $\hat d$ of $\hat P$ are such that $l(d)\ge l(\hat d)\ge l(d')$.
\end{itemize}
We presently explain how the elementary step of type two on the smooth diagonal $d'=\overline{x'y'}$ of $P'$ is performed.\\

Unlike elementary steps of type one, it is possible to perform an elementary step on a smooth diagonal $d'=\overline{x'y'}$ of $P'$ only if there are other vertices $u',v'$ of $P'$ such that:
\begin{itemize}
\item all four geodesics $\overline{u'x'},\overline{x'v'},\overline{v'y'},\overline{y'u'}$ are smooth and thus define a quadrilateral \\ $P(x',y',u',v')\subset P'$,
\item $\overline{x'y'}$ is a smooth diagonal of $P(x',y',u',v')$,
\item $P(x',y',u',v')$ has only one strictly concave internal angle, which is in $x'$ or $y'$. Consequently all other three internal angles of $P(x',y',u',v')$ are strictly convex.
\end{itemize} 

We allow the quadrilateral $P(x',y',u',v')$ to be $degenerate$ in the sense that one of the internal angles of $P(x',y',u',v')$ in $u'$ or $v'$ can be zero.\\
The set $\overline{P'\setminus P(x',y',u',v')}$ consists of a number of polygons $Q_i'$ which varies between zero and four. It is possible to obtain another quadrilateral $P(\hat x,\hat y,\hat u,\hat v)$ from $ P(x',y',u',v')$ increasing $d_{P'}(x',y')$, decreasing the strictly concave angle of the quadrilateral $P(x',y',u',v')$ and leaving unchanged the lengths of its sides. The polygon $\hat P$ is then obtained gluing back the polygons $Q_i'$ on the corresponding sides of $P(\hat x,\hat y,\hat u,\hat v)$.\\

\begin{figure}[h!]
 \centering
  \includegraphics[scale=0.45]{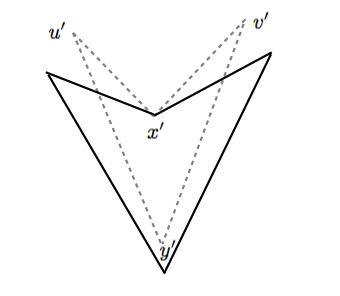}
 \caption{An example of an elementary step of type two: $P(x',y',u',v')$ is the quadrilateral drawn with a dashed line, while $P(\hat x,\hat y,\hat u,\hat v)$ is the quadrilateral drawn with a continuous line.}
 \end{figure}

There is a 1-Lipschitz map $ \psi_1:P(\hat x,\hat y,\hat u,\hat v)\rightarrow P(x',y',u',v')$, with respect to the intrinsic Euclidean metrics of the polygons, which is the identity on the sides of the quadrilaterals. It can be extended to a 1-Lipschitz map $\psi:\hat P\rightarrow P'$ defining it as the identity on the polygons $Q_i'$. \\

Notice that both types of elementary steps do not change the sum of the internal angles of polygons $Q'$ on which they are performed.\\

Now that we have defined the two types of elementary steps, we can go back to explaining how to obtain the desired 1-Lipschitz map $f:P\rightarrow P'$.\\
We will use the following lemma regarding generalized polygons. Notice that any generalized polygon $Q'$ with $m$ vertices has sum of internal angles equal to $\pi(m-2)$. Indeed, suppose $Q$ is obtained gluing two planar polygons $Q_1$ and $Q_2$ along a side $\overline{vw}$: denote by $m_1,m_2$ the number of vertices of $Q_1$ and $Q_2$, then it must follow $m_1+m_2=m+2$ (since gluing $Q_1$ and $Q_2$ we lose two vertexes which are identified together). Consequently the sum of internal angles of $Q$ is equal to $\pi(m_1-2)+\pi(m_2-2)=\pi(m-2)$.

\label{conv}\begin{lem}
Let $Q'$ be a generalized polygon with $n$ vertices. Then it is possible to apply a finite sequence of elementary steps of type two on $Q'$ turning it into a convex polygon $\hat Q$ such that  all sides of $\hat Q$ are of the same length of the corresponding sides of $Q'$.
\end{lem}
\begin{proof}
We proceed by induction on the number $m$ of vertices of $Q'$. \\
If $m=4$ then the result is trivial. 
Suppose the thesis is true for all polygons $Q'$ with number of vertices between 4 and $m>4$, then we will prove it for polygons $Q'$ with $m+1$ vertices.\\
We cut $Q'$ along a smooth diagonal $\overline{v'w'}$ obtaining two generalized polygons $Q_i$ on which we can apply the inductive hypothesis thus turning them into two convex polygons $\hat Q_i$: we glue $\hat Q_1,\hat Q_2$ back together along $\overline{\hat v\hat w}$ obtaining a polygon $\hat Q$ which can have strictly concave internal angles only in $\hat v$ and $\hat w$. If $\alpha_{\hat v}>\pi$ then one performs an elementary step of type two on $P(\hat v_1,\hat v,\hat v_2,\hat w)$ (where $\hat v_1,\hat v_2$ are the vertices next to $\hat v$) stretching $\overline{\hat v\hat w}$ until $\alpha_{\hat v}=\pi$. Finally, only the angle $\alpha_{\hat w}$ can be strictly concave. Notice that all diagonals $\overline{\hat w\hat z}$ must be smooth, where $\hat z$ is any vertex of $\hat Q$ not adjacent to $\hat v$: using this fact and the hypothesis on the internal angles of $Q'$ one gets that it is always possible to flatten the angle $\alpha_{\hat w}$ performing elementary steps of type 2 stretching $\overline{\hat w\hat x}$, where $\hat x$ is a vertex of $\hat Q$ such that $\alpha_{\hat x}<\pi$, without making any angle $\alpha_{\hat x}$ strictly concave.
\end{proof}

\begin{oss}Notice that, given polygons $Q'$ and $\hat Q$ as in the previous lemma, if $x'$ is a vertex of $Q'$ such that $\alpha_{x'}\ge \pi$, then it can not result $\alpha_{\hat x}<\pi$.\\
To see this, denote by $x_1',x_2'$ the two vertices of $Q'$ adjacent to $x'$ and by $\hat x_1,\hat x_2$ the two corresponding vertices of $\hat Q$. If $\alpha_{\hat x}<\pi$ then $\overline{\hat x_1\hat x_2}$ is a segment of length strictly smaller than $d_{Q'}(x_1',x')+d_{Q'}(x',x_2')=d_{Q'}(x_1',x_2')$ and consequently it would result $d_{Q'}(x_1',x_2')>d_{\hat Q}(\hat x_1,\hat x_2)$.\\
This inequality would contradict the fact that, since $\hat Q$ is obtained from $Q'$ through a sequence of elementary steps of type two, there is a 1-Lipschitz map $f:\hat Q\rightarrow Q'$ which sends vertices to corresponding vertices.
\end{oss}

We can now start the proof of theorem \ref{lemshort2}, using induction on the number $n$ of vertices of $P$. In order to make the proof more easily readable, we will divide the following arguments in succeeding lemmas. \\
Suppose $n=3$, then $P$ is an Euclidean triangle, while $P'$ can have many more vertices than $P$. In order to satisfy condition (v) of the definition of degenerate polygons comparable to $P$, $P'$ can have only one planar subpolygon and at most three one dimensional components ending in points of $\iota(Vertices(\Delta))$.\\
Consequently, for $n=3$, $P$ and $P'$ will be polygons of the type described in theorem \ref{lemshort1}: for this reason we will denote them by $\Delta,\Delta'$. 

\begin{lem}
If $n=3$, it is possible to turn $\Delta'$ into $\Delta$ using only elementary steps of type one and two and consequently get a 1-Lipschitz map $f:\Delta\rightarrow \Delta'$ obtained composing all intermediate 1-Lipschitz maps between intermediate polygons.
\end{lem}
\begin{proof}
We first get rid of the one dimensional components of $\Delta'$, turning them into part of $\overline{\mathop \Delta\limits^ \circ}$ using elementary steps as thus explained (we will explain the procedure only for the one dimensional component starting at $x_1'$, the other two will be treated in the same way). \\
Suppose there is a one dimensional component $\overline{x_1'v_1'}$ starting at $x_1'$ and $v_1'\in\overline{\mathop \Delta\limits^ \circ}$ is the vertex such that the corresponding internal angle of $\overline{\mathop \Delta\limits^ \circ}$ must be strictly convex. Let $v_2'$ be the vertex of $\overline{x_1'v_1'}$ closer to $v_1'$ and let $w_1',u_1'$ be the two vertices of $\overline{\mathop \Delta\limits^ \circ}$ adjacent to $v_1'$. We perform an elementary step of type two on $P(u_1',v_1',v_2',w_1')$ (which is a degenerate polygon, since the internal angle in $v_2'$ is zero) until $P(u_1',v_1',v_2',w_1')$ is no longer degenerate. Notice that there are two ways of performing an elementary step of type two on a degenerate quadrilateral $P(u_1',v_1',v_2',w_1')$ (as it is showed in figure 9): in one way $\overline{u_1'v_1'}$ is stretched and it will result $\hat v_1\in \overline{\hat w_1 \hat v_2}$ and in the other way $\overline{v_1'w_1'}$ is stretched and it will result $\hat  v_1\in \overline{\hat u_1\hat v_2}$. Since we do not care on which side the vertex $\hat v_1$ will end up, we can choose either way.\\

\begin{figure}[h!]
 \centering
  \includegraphics[scale=0.5]{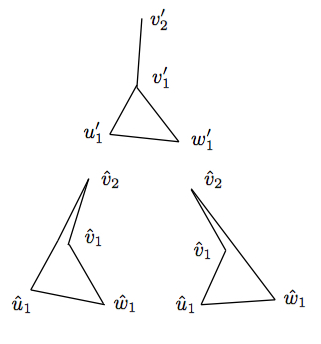}
 \caption{Two ways of performing an elementary step of type two on $P(u_1',v_1',v_2',w_1')$}
 \end{figure}

We then proceed in the same way considering the vertex $v_3'$ of the one dimensional component $\overline{v_2'x_1'}$ closer to $v_2'$.\\
Having done so, we obtain a polygon without one dimensional components and with concave internal angles in all vertices which are not in $\iota(Vertices(\Delta))$: we turn it into an Euclidean triangle $\widehat \Delta$ of vertices $\hat x_i$, $i=1,2,3$ using lemma \ref{conv} through elementary steps of type two. Finally, we perform a finite sequence of elementary steps of type one on the three sides $\overline{\hat x_i\hat x_j}$ of $\widehat \Delta$ in order to make them of the same length of the corresponding sides of $\Delta$.\\
Suppose all three sides of $\widehat \Delta$ are such that $d_{\widehat \Delta}(\hat x_i,\hat x_j)<d_\Delta(x_i,x_j)$. We start by stretching the length of $\overline{\hat x_1\hat x_2}$ until the angle in $\hat x_3$ is equal to $\pi-\epsilon$, with $\epsilon>0$ very small: in this way it results $l(\overline{\hat x_1\hat x_3})^2+l(\overline{\hat x_2\hat x_3})^2=l(\overline{\hat x_1\hat x_2})^2+\psi(\epsilon)$ with $\lim\limits_{\epsilon\rightarrow 0}\frac{\psi(\epsilon)}{\epsilon}=0$. Performing again an elementary step of type one stretching $\overline{\hat x_2\hat x_3}$ until the angle in $\hat x_1$ is equal to $\pi-\epsilon$  one gets $2l(\overline{\hat x_1\hat x_3})^2+l(\overline{\hat x_2\hat x_3})^2=l(\overline{\hat x_2\hat x_3})^2+\psi_1(\epsilon)$ with $\lim\limits_{\epsilon\rightarrow 0}\frac{\psi_1(\epsilon)}{\epsilon}=0$. Consequently, proceeding in this way, after a finite number of steps one side must reach its maximum length. Then we proceed in the same way until all sides of $\widehat \Delta$ are of the same length of the corresponding sides of $\Delta$.
\end{proof}

Now suppose the inductive hypothesis is verified if the number of vertices of $P$ is not greater than $n$, then we shall find the 1-Lipschitz map if $P$ has $n+1$ vertices.
\begin{lem}
If there is a diagonal $\overline{v'w'}$ of $P'$ such that $l(\overline{v'w'})=l(\overline{vw})$, then it is possible to apply the inductive hypothesis to obtain the 1-Lipschitz map $f:P\rightarrow P'$.
\end{lem}
\begin{proof}
We divide two cases.
\begin{itemize}
\item if $\overline{vw}$ is smooth then we cut the polygons $P$ and $P'$ in the following way:
\begin{itemize}
\item we cut $P$ along $\overline{vw}$ obtaining $P_1$ and $P_2$,
\item we cut $P'$ along $\overline{v'w'}$ obtaining $P_1'$ and $P_2'$.
\end{itemize}
Notice that if $\overline{v'w'}$ is not smooth then the operation of \textit{cutting along $\overline{v'w'}$} must be further clarified. If $d'$ passes through a side of $P'$ (resp. a one dimensional component), then such side (resp. one dimensional component) will appear on both polygons $P_i'$. Notice that in this way the polygons $P_i'$ could acquire new one dimensional components and new vertices. We will follow this rule to name the new vertices: if  $u'$ is a vertex of $\iota(Vertices(P))$ on $\overline{v'w'}$ and $u\in \Delta_1$ (resp. $u\in \Delta_2$), then $u'$ will be a vertex only of $\Delta_1'$ (resp. $\Delta_2'$).

\begin{figure}[h!]
 \centering
  \includegraphics[scale=0.5]{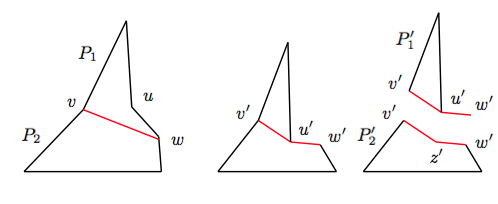}
 \caption{An example of cutting in case $d'$ is not smooth: notice that the points $v',w'$ appear on both $P_1'$ and $P_2'$, while $u'$ appears only on $P_1'$, since $u\in P_1$. On $P_2'$ there is a new vertex $z'\not\in \iota(Vertices(P_2))$.}
 \end{figure}

Sometimes a polygon $P_i'$ could be entirely degenerate (i.e. one dimensional): in that case we perform a degenerate elementary step of type one on $P_i'$ turning it into a degenerate polygon which includes at least one planar polygon.\\  
We can thus suppose both newly obtained polygons $P_1'$ and $P_2'$ are degenerate polygons comparable respectively with $P_1$ and $P_2$. Indeed, condition (v) of the definition is verified since, if $z'\in\overline{v'w'}$, $z'\not\in \iota(Vertices(P))$ is a vertex of a planar polygon of $P'$, a corresponding vertex $z_i'\in P_i'$ can have strictly convex internal angle only if from $z_i'$ starts a one dimensional component. 

This is the crucial property we were looking for: we can now apply the inductive hypothesis and obtain two 1-Lipschitz maps $f_i:\Delta_i\rightarrow \Delta_i'$ which must agree on $\overline{vw}$: we will define $f:\Delta\rightarrow \Delta'$ to be such that $f|_{\Delta_i}:=f_i$.\\
Notice that the same reasoning could not have been done considering polygons $\Delta$ and $\Delta'$ of the hypothesis of theorem \ref{lemshort1}, as explained in picture 5.11.
\begin{figure}[h!]
 \centering
  \includegraphics[scale=0.55]{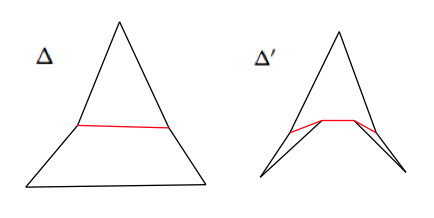}
 \caption{The diagonals $d$ and $d'$ are drawn in red. One clearly sees that the bottom half of $\Delta$ has four strictly convex angles, while the bottom half of $\Delta'$ is composed by two triangles connected by a one dimensional component.}
 \end{figure}

\item if  $\overline{vw}$ is not smooth, then suppose $\overline{vw}$ is the concatenation of segments $\overline{vv_1}*\overline{v_1v_2}*\cdots *\overline{v_{m}w}$: at least one of them must be a smooth diagonal, so suppose $\overline{vv_1}$ is. Notice that if $l(\overline{vw})=l(\overline{v'w'})$ then it must follow $\overline{v'w'}=\overline{v'v_1'}*\overline{v_1'v_2'}*\cdots *\overline{v_{m}'w'}$, otherwise one would get
$$l(\overline{vv_1})+l(\overline{v_1v_2})+\cdots +l(\overline{v_mw})< l(\overline{v'v_1'})+l(\overline{v_1'v_2'})+\cdots +l(\overline{v_m'w'})$$
which contradicts the hypothesis on the distances in $P$ and $P'$.\\
Now one can just consider the diagonals $\overline{vv_1}$ (which is smooth) and $\overline{v'v_1'}$ and fall into the previous case.
\end{itemize}

\end{proof}

After these considerations we can always suppose all diagonals of $P'$ are strictly shorter than the corresponding diagonals of $P$. We will now deal with one dimensional components of $P'$.
\begin{lem}
Suppose all diagonals of $P'$ are strictly shorter than the corresponding diagonals of $P$. Then, using elementary steps, it is possible to turn $P'$ in a degenerate polygon comparable to $P$ without one dimensional components. If in doing so one diagonal of $P'$ reaches its maximum length (i.e. the length of the corresponding diagonal of $P$) it is possible to apply the inductive hypothesis to obtain the desired 1-Lipschitz map $f:P\rightarrow P'$.
\end{lem}
\begin{proof}
We will proceed in a way which is almost identical to the one applied in the previous case $n=3$. Let $\overline{v_1'x_1'}$ be a one dimensional component of $P'$ and $v_2'$ the vertex of $\overline{v_1'x_1'}$ closer to $v_1'$, then we will apply an elementary step of type two on $P(u_1',v_1',v_2',w_1')$ (where, as before, $u_1'$ and $w_1'$ are vertices of a planar subpolygon of $P'$ adjacent to $v_1'$) in such a way that the newly obtained degenerate polygon $\hat P$ satisfies axiom (vi). In particular, if $v_1',u_1',w_1',v_2'$ are all vertices of $\iota(Vertices(P))$ then we will apply the elementary step which gives $\hat v_1\in \overline{\hat w_1\hat v_2}$ (resp. $\hat v_1\in \overline{\hat u_1\hat v_2}$) if $v_1\in \overline{ w_1 v_2}$ (resp. $ v_1\in \overline{ u_1v_2}$). If $v_1'$ is not a vertex of $\iota(Vertices(P'))$, then it is possible to perform both types of elementary step of type one. \\
Proceeding in this way one could end up with a vertex $x_1'$ which connects two planar polygons of $P'$: it is possible to get rid of this "pathology" with another elementary step of type two as it is explained in figure 12.\\

\begin{figure}[h!]
 \centering
  \includegraphics[scale=0.5]{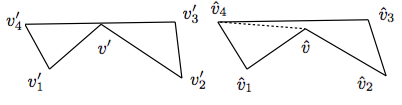}
 \caption{If $v\in \overline{v_1v_2}$ then one performs an elementary step of type two on $P(v_4',v',v_2',v_3')$.}
 \end{figure}

In this way we explained also how to get rid of vertices of $P'$ which link two different planar polygons. Clearly, if at any point during this procedure of elimination of one dimensional components, one ends up with a diagonal $d'$ of $P'$ such that $l(d')=l(d)$ then the 1-Lipschitz map $f:P\rightarrow P'$ is obtained as explained before.
\end{proof}

At this point, we can suppose $P'$ does not have one dimensional components, but it can stil have more vertices than $P$. Notice that a straightforward consequence of the definition of elementary steps of type two and of condition (v) of the definition of degenerate polygons comparable to $P$ is that all internal angles in vertices of $P'$ which are not in $\iota(Vertices(P))$ will have concave internal angle.
\begin{lem}
Suppose all diagonals of $P'$ are strictly shorter than the corresponding diagonals of $P$ and $P'$ does not have one dimensional components. Then, using elementary steps, it is possible to turn $P'$ in a degenerate polygon comparable to $P$ with the same vertices of $P$. If in doing so one diagonal of $P'$ reaches its maximum length (i.e. the length of the corresponding diagonal of $P$) it is possible to apply the inductive hypothesis to obtain the desired 1-Lipschitz map $f:P\rightarrow P'$.
\end{lem}
\begin{proof}
One just has to apply lemma \ref{conv} (and its following observation), turning $P'$ into a convex polygon $\hat P$. The polygon $\hat P$ will have flat internal angles at vertices $\hat z$ such that the corresponding vertex $z'$ of $P'$ is not in $\iota(Vertices(\Delta))$. At this point one simply "forgets" about $\hat z$ and removes it from the set of vertices of $\hat P$.\\
Again, if, performing any of the elementary steps of type two of lemma \ref{conv}, one diagonal $d'$ of $P'$ is stretched until $l(\hat d)=l(d)$, then  the procedure is finished as we already explained.
\end{proof}
We will now stretch all sides of $P'$ until they become of the same length of the corresponding sides of $P$.
\begin{lem}
Suppose all diagonals of $P'$ are strictly shorter than the corresponding diagonals of $P$, $P'$ has the same vertices of $P$ and $P'$ does not have one dimensional components. Then, using elementary steps of type one, it is possible to stretch all sides of $P'$ until they become of the same length of the corresponding sides of $P$. If in doing so one diagonal of $P'$ reaches its maximum length (i.e. the length of the corresponding diagonal of $P$) it is possible to apply the inductive hypothesis to obtain the desired 1-Lipschitz map $f:P\rightarrow P'$.
\end{lem}

\begin{proof}
First, notice that it is not always possible to stretch a side of $\Delta'$ with just one elementary step of type one until it reaches its maximum length. Indeed, let $\overline{x_1'x_2'}$ be a side of $P'$ such that $l(\overline{x_1'x_2'})<l(\overline{x_1x_2})$ and $P(x_1',x_2',z')$ a triangle as in the definition of elementary step of type one. Notice that the upper limit of the length of the side $\overline{x_1'x_2'}$ obtainable through an elementary step of type one on $P(x_1',x_2',z')$ is $l(\overline{x_1'z'})+l(\overline{x_2'z'})$ (at whose length $P(x_1',x_2',z')$ becomes a segment). \\

In order to overcome this difficulty, we number the vertices of $P'$ in an increasing order starting from from $x_1'$, in such a way that its adjacent vertices are $x_2'$ and $x_m'$. Then we will explain how to turn $P'$ into a triangle with convex angles in $x_1',x_2'$ and internal angle in $x_m'$ equal to $\pi-\epsilon$. In this way it will result that $l(\overline{x_1'x_2'})^2$ is equal to the sum of the squares of the lengths of all other sides of $P'$ minus a term $\psi(\epsilon)$ such that $\lim\limits_{\epsilon\rightarrow 0}\frac{\psi(\epsilon)}{\epsilon}=0$: the conclusion will follow in the same way of the case of the proof of lemma 3.25. Clearly, if doing so one diagonal $d'$ of $P'$ is stretched until $l(d')=l(d)$ then the procedure is finished as explained before. One should notice that coinciding vertices do not constitute a problem, since they just correspond to sides of length zero (and will now be stretched by degenerate elementary steps of type one).\\

We now explain how to turn $P'$ into a triangle with convex angles in $x_1',x_2'$ and internal angle in $x_m'$ equal to $\pi-\epsilon$: first of all we apply lemma \ref{conv} and turn $P'$ into a convex polygon $\hat P$.\\ 
Denote by $\alpha_{\hat x_i}$ the internal angle of $\hat P$ in $\hat x_i$. If $\alpha_{\hat x_2}<\pi$ and $\alpha_{\hat x_j}=\pi$ for $j=3,\dots,l-1$, we perform an elementary step of type one on $P(\hat x_1,\hat x_2,\hat x_l)$ until $\alpha_{\hat x_j}=\pi$.\\
If $\alpha_{\hat x_i}=\pi$ for $i=2,\dots, k-1$, we perform an elementary step of type one on $P(\hat x_1,\hat x_2,\hat x_k)$ until $\alpha_{\hat x_k}=\pi$ and then perform an elementary step of type one on $P(\hat x_1,\hat x_2,\hat x_{k-1})$ until $\alpha_{\hat x_{k-1}}=\pi$. \\
Proceeding in this way one can flatten all angles $\alpha_{\hat x_i}$, $i=3,\dots,m-1$, until $\hat P$ becomes a triangle with convex angles only in $\hat x_1,\hat x_2,\hat x_m$. Finally, one performs an elementary step of type one until $\alpha_{\hat x_m}=\pi-\epsilon$.
\end{proof}

The following lemma concludes the proof of theorem \ref{lemshort2}.
\begin{lem}
Suppose all diagonals of $P'$ are strictly shorter than the corresponding diagonals of $P$, $P'$ does not have one dimensional components, $P'$ has the same vertices of $P$ and all sides of $P'$ have the same length of the corresponding sides of $P$. Then it is possible to obtain the desired 1-Lipschitz map $f:P\rightarrow P'$.
\end{lem}
\begin{proof}
We will prove that it will always be possible to obtain a diagonal $d'$ of $P'$ of maximum length, applying a finite number of elementary steps of type two on $P'$: then the conclusion will follow as in the previous lemmas.\\
Once again, we turn $P'$ into a convex polygon $\hat P$ using lemma \ref{conv}. In case $\hat P\neq P$, there must be a vertex $\hat x$ of $\hat P$ such that $\alpha_{\hat x}>\alpha_x$. If we can prove that this implies the existence of a diagonal $\hat d$ of $\hat P$ such that $l(\hat d)\ge l(d)$ then the proof is finished, since this means that at some point during the sequence of elementary steps of type two which turns $P'$ into $\hat P$ one gets $l(\hat d)=l(d)$.\\
We prove the equivalent statement that if all diagonals of $\hat P$ are strictly shorter than the corresponding diagonals of $P$, then all convex angles of $P$ must be greater than the corresponding angles of $\hat P$.\\
Denote by $y$ and $z$ the vertices of $P$ next to $x$: suppose $\overline{yz}$ is the concatenation of the smooth segments $\overline{yx_1}*\overline{x_1x_2}*\cdots *\overline{x_kz}$ for $k\ge 0$.\\
Denote by $Q$ the polygon delimited by $\overline{xy},\overline{xz}$ and $\overline{yz}$: all internal angles of $Q$ are concave except for the ones in $x,y,z$ and $\overline{xx_i}$, $i=1,\dots,k$ are smooth diagonals contained in $Q$. We claim that decreasing the length of all diagonals $\overline{xx_i}$ without increasing the length of the sides of $Q$ and without changing the lengths of $\overline{xy}$ and $\overline{xz}$, the angle $\alpha_x$ will decrease: this can be proved modifying the lengths of sides of $Q$ one at a time. \\
Indeed, if only $d_Q(x_i,x_{i+1})$ decreases, then $\alpha_x$ must decrease: this can be easily seen shortening the side $\overline{x_ix_{i+1}}$ of the triangle $P(x,x_i,x_{i+1})$ of vertices $x,x_i,x_{i+1}$ without changing the lengths of the other two sides of $P(x,x_i,x_{i+1})$.  In the same way, if only $d_Q(x,x_{i})$ decreases, then $\alpha_x$ must decrease: this can be easily seen shortening the diagonal $\overline{xx_i}$ of the quadrilateral $P(x,x_{i-1},x_i,x_{i+1})$ of vertices $x,x_{i-1},x_i,x_{i+1}$ without changing the lengths of the sides of $P(x,x_{i-1},x_i,x_{i+1})$.
\end{proof}

As we said, this ends the proof of theorem \ref{lemshort2} and consequently also theorem 3.21 is proved.\\

We are now left with the case common vertices of $\Delta$ and of $\Delta'$ are not disposed in the same order.\\
From now on, we will denote the vertices of $\Delta$ with concave internal angle in the following way, which will be useful in the succeeding reasonings.
\begin{itemize}
\item Denote by $w_j$ the vertices on $\overline{x_1x_2}$, ordered in increasing order from $x_1$ to $x_2$.
\item Denote by $u_k$ the vertices on $\overline{x_2x_3}$, ordered in increasing order from $x_3$ to $x_2$.
\item Denote by $v_l$ the vertices on $\overline{x_1x_3}$, ordered in increasing order from $x_3$ to $x_1$.
\end{itemize}

As before, we denote by  $w_j',u_k',v_l'$ the corresponding vertices of $\Delta'$.\\
We say a vertex $w_j$ has \textit{changed side} on $\Delta'$ if $w_j'\not\in \overline{x_1'x_2'}$.\\
Two vertices $w_m',w_n' \in \overline{x_1'x_2'}$ have \textit{changed their order} if $m<n$ and it results $d_{\Delta'}(w_n',x_1')<d_{\Delta'}(w_m',x_1')$.\\
Changes of side and order of vertices $u_k$ and $v_l$ are defined in the same way.\\
Common vertices of $\Delta$ and of $\Delta'$ are not disposed in the same order if there is at least one change of side or one change of order.\\
As we anticipated, we are not able to prove a statement similar to the one of theorem 3.21 in case common vertices of $\Delta$ and $\Delta'$ are not disposed in the same order. So we can only state the following conjecture.

\begin{conj}
Suppose the number of vertices of $\Delta'$ can be greater than the number of vertices of $\Delta$, $\Delta'$ can have one dimensional components and the common vertices of $\Delta$ and $\Delta'$ are not disposed in the same order. \\
Then for every $p\in \Delta$ there is a corresponding point $p'\in \Delta'$ such that $$d_{\Delta'}(p',x_i')\le d_\Delta(p,x_i),\quad i=1,2,3.$$
\end{conj}
Clearly, it is not possible to adapt the proof of theorem \ref{lemshort2} to prove conjecture 5.31, since the method consisting of elementary steps would only work if common vertices of $\Delta$ and $\Delta'$ have the same order.\\
Nonetheless, we are quite confident conjecture 5.31 must be true: this is because changes of side or order of vertices force the polygon $\Delta'$ to become smaller.\\
Indeed, if two vertices $w_m,w_n$ of $\overline{x_1x_2}$ change order in $\Delta'$, then it must result $d_{\Delta'}(x_1',x_2')\le d_\Delta(x_1,x_2)-d_\Delta(w_m,w_n)$. Since each change of order of the vertices contributes to the shortening of $\overline{x_1'x_2'}$, as the number of changes of order of vertices of $\overline{x_1x_2}$ increases, the shortening of $\overline{x_1'x_2'}$ also increases.  \\
In a similar way, if a vertex of $\Delta$ changes side and for example it is $u_{k_0}'\in \overline{x_1'x_3'}$, then, since the distances $d_{\Delta'}(u_k',u_{k_0}')$ can not be greater than the corresponding distances $d_{\Delta}(u_k,u_{k_0})$, all other vertices $u_k'$ are forced to "follow" $u_{k_0}'$ and become closer to vertices of $\overline{x_1'x_3'}$. This fact will force some distances inside $\Delta'$ to become smaller than the corresponding distances in $\Delta$. \\
In light of these observations, one could even consider the case common vertices of $\Delta$ and $\Delta'$ are disposed in the same order as the worst one to prove the existence of $p'$, since no distance inside $\Delta'$ is forced to decrease.\\

The following two propositions should support our intuition. Indeed, they show some cases where the change of side of one or more vertices of $\Delta'$ directly implies the existence of $p'$.

\begin{prop}
Suppose there is at least one vertex of $\Delta'$ which changes side, for example $u'\in \overline{x_1'x_3'}$. Then for every $p\in \Delta$ such that $d_\Delta(p,x_3)\le d_{\Delta'}(u',x_3')$ there is a point $p'\in \Delta'$ such that $$d_{\Delta'}(p',x_i')\le d_\Delta(p,x_i),\quad i=1,2,3.$$
\end{prop}
\begin{proof}
Choose $p'\in \overline{u'x_3'}$ at distance $d_\Delta(p,x_3)$ from $x_3'$. Then it results
$$d_{\Delta'}(p',x_1')=d_{\Delta'}(x_1',x_3')-d_{\Delta'}(p',x_3')\le d_{\Delta}(x_1,x_3)-d_\Delta(p,x_3)\le d_\Delta(p,x_1),$$
$$d_{\Delta'}(p',x_2')\le d_{\Delta'}(p',u')+d_{\Delta'}(u',x_2')\le d_\Delta(x_2,x_3)-d_\Delta(p,x_3)\le d_\Delta(p,x_2).$$
\end{proof}

\begin{prop}
Suppose one of the following three conditions is satisfied:
\begin{enumerate}[label=(\roman*)]
\item there are vertices $u',v'\in \overline{x_1'x_2'}$ such that $d_{\Delta'}(x_1',v')>d_{\Delta'}(x_1',u')$,
\item there are vertices $u',w'\in \overline{x_1'x_3'}$ such that $d_{\Delta'}(x_1',w')>d_{\Delta'}(x_1',u')$,
\item there are vertices $v',w'\in \overline{x_2'x_3'}$ such that $d_{\Delta'}(x_2',w')>d_{\Delta'}(x_2',v')$.
\end{enumerate}
Then for every $p\in \Delta$ there is a corresponding point $p'\in \Delta'$ such that $$d_{\Delta'}(p',x_i')\le d_\Delta(p,x_i),\quad i=1,2,3.$$
\end{prop}
\begin{proof}
We will prove the proposition only for case $(i)$, since the proof is identical for the other two cases.\\ 
One can find $p'$ as follows.
\begin{enumerate}
\item If $d_\Delta(p,x_2)\le d_{\Delta'}(u',x_2')$ let $p'$ be the point on $\overline{u'x_2'}$ at distance $d_\Delta(p,x_2)$ from $x_2$. It then results
$$d_{\Delta'}(p',x_1')=d_{\Delta'}(x_1',x_2')-d_{\Delta'}(p',x_2')\le d_\Delta(x_1,x_2)-d_\Delta(p,x_2)\le d_\Delta(p,x_1),$$
$$d_{\Delta'}(p',x_3')\le d_{\Delta'}(x_3',u')+d_{\Delta'}(u',x_2')-d_{\Delta'}(p',x_2')\le d_\Delta(x_2,x_3)-d_\Delta(p,x_2)\le d_\Delta(p,x_3).$$
\item If $d_\Delta(p,x_3)\le d_{\Delta'}(u',x_3')$ let $p'$ be the point on $\overline{x_3'u'}$ at distance $d_\Delta(p,x_3)$ from $x_3'$. It results
$$d_{\Delta'}(p',x_2')\le d_{\Delta'}(x_3',u')+d_{\Delta'}(u',x_2')-d_{\Delta'}(p',x_3')\le d_\Delta(x_2,x_3)-d_\Delta(p,x_3)\le d_\Delta(p,x_2),$$
$$d_{\Delta'}(p',x_3')+d_{\Delta'}(p',x_1')\le d_{\Delta'}(v',x_3')+d_{\Delta'}(v',x_1')\le d_\Delta(x_1,x_3),$$
$$d_{\Delta'}(p',x_1')\le d_\Delta(x_1,x_3)-d_{\Delta'}(p',x_3')=d_\Delta(x_1,x_3)-d_\Delta(p,x_3)\le d_\Delta(p,x_1)$$
\item If $d_\Delta(p,x_2)> d_{\Delta'}(u',x_2')$ and $d_\Delta(p,x_3)> d_{\Delta'}(u',x_3')$ then there is always a point $p'\in \overline{x_1'u'}$ such that one of the following two conditions is satisfied:
\begin{itemize}
\item $d_{\Delta'}(p',x_2')=d_\Delta(p,x_2)$ and $d_{\Delta'}(p',x_3')\le d_\Delta(p,x_3)$, then one can proceed as in previous case (1),
\item $d_{\Delta'}(p',x_3')=d_\Delta(p,x_3)$ and $d_{\Delta'}(p',x_2')\le d_\Delta(p,x_2)$, then one can proceed as in previous case (2).
\end{itemize}
\end{enumerate}
\end{proof}

One could try to prove conjecture 5.31 using the following approach. \\
Consider a subpolygon $\widehat \Delta\subset \Delta'$ such that to every vertex $x_i,w_j,u_k, v_l$ of $\Delta$ there is a unique corresponding vertex $\hat x_i,\hat w_j,\hat u_k,\hat v_l$ of $\widehat \Delta$.\\ 
We say that $\widehat \Delta$ is a \textit{subpolygon of $\Delta'$ comparable to $\Delta$} if $\hat x_i=x_i'$, $i=1,2,3$ and $\Delta,\widehat \Delta$ satisfy the hypothesis of theorem 5.3.4. In particular, this condition implies that :
\begin{itemize}
\item common vertices of $\widehat \Delta$ and of $\Delta$ are disposed in the same order,
\item the distance between any two vertices of $\Delta$ is greater than or equal to the distance between the corresponding two points of $\widehat \Delta$.
\end{itemize}

If such polygon $\widehat \Delta$ exists, then preceding theorem \ref{lemshort2} will grant the existence of a 1-Lipschitz map $\phi:\Delta\rightarrow \widehat \Delta$ which sends vertices of $\Delta$ to corresponding vertices of $\widehat \Delta$. Since for every couple of points $x',y'\in \widehat \Delta$ it results $d_{\widehat \Delta}(x',y')\ge d_{\Delta'}(x',y')$, one will conclude that $\phi$ is also a 1-Lipschitz map from $\Delta$ to $\Delta'$ such that $\phi(x_i)=x_i'$, $i=1,2,3$.\\
Notice that there is no need to require the polygon $\widehat \Delta$ to have exactly three strictly convex internal angles, since it is not required in the hypothesis of theorem \ref{lemshort2}.\\
Unfortunately we were not able to develop a method which always produces such polygon $\widehat \Delta$ for every $\Delta,\Delta'$. Indeed, we can only make the following conjecture.

\begin{conj}
Suppose the number of vertices of $\Delta'$ can be greater than the number of vertices of $\Delta$, $\Delta'$ can have one dimensional components and the common vertices of $\Delta$ and $\Delta'$ are not disposed in the same order. \\
Then there always is a subpolygon $\widehat \Delta$ of $\Delta'$ comparable to $\Delta$.
\end{conj}

As we just explained, conjecture 3.34 implies conjecture 5.31. \\
We feel conjecture 3.34 must be true for the same reasons we explained to justify conjecture 5.31: each vertex which changes side or order in $\Delta'$ forces some distances to decrease.\\
We are only able to prove conjecture 3.34 in two simple cases, which we now illustrate.  

\begin{prop}
If there is only one point $u_{k_0}'$ of $\Delta'$ such that $u_k'\in \overline{x_1'x_3'}$ and no other vertex changes order or side, then there is a subpolygon $\widehat \Delta$ of $\Delta'$ comparable to $\Delta$.
\end{prop}

\begin{proof}
In this case it is possible to obtain $\widehat \Delta$ in the following way.\\ 
One replaces $\overline{x_2'x_3'}$ with $\overline{x_3'u_{k_0}'}*\overline{u_{k_0}'x_2'}$ and evaluates if it is possible to find points \\ $\hat u_1,\dots,\hat u_{k_0-1},\hat u_{k_0+1},\dots,\hat u_k\in \overline{x_3'u_{k_0}'}*\overline{u_{k_0}'x_2'}$ corresponding to $u_1',\dots,u_{k_0-1}',u_{k_0+1}',\dots,u_k'$ such that the polygon identified by $\overline{x_1'x_2'},\overline{x_3'u_{k_0}'}*\overline{u_{k_0}'x_2'}, \overline{x_1'x_3'}$ is a subpolygon of $\Delta'$ comparable to $\Delta$.\\ 
If so, the proof is concluded, since we have found the desired polygon $\widehat \Delta$. \\
If not, consider the orthogonal projection $pr:\Delta'\rightarrow \overline{x_2'x_3'}$: one moves the point $u_{k_0}$ in $\hat u_{k_0}$ on $\overline{u_{k_0}'pr(u_{k_0}')}$ towards $pr(u_{k_0}')$ until one of the following events happens.
\begin{enumerate}[label=(\roman*)]
\item Replacing $\overline{x_2'x_3'}$ with $\overline{x_3'\hat u_{k_0}}*\overline{\hat u_{k_0}x_2'}$ it is possible to find points $\hat u_1,\dots,\hat u_{k_0-1},$ $\hat u_{k_0+1},\dots,\hat u_k\in \overline{x_3'\hat u_{k_0}}*\overline{\hat u_{k_0}x_2'}$ corresponding to $u_1',\dots,u_{k_0-1}',u_{k_0+1}',\dots,u_k'$ such that the polygon identified by $\overline{x_1'x_2'},\overline{x_3'\hat u_{k_0}}*\overline{\hat u_{k_0}x_2'}, \overline{x_1'x_3'}$ is a subpolygon of $\Delta'$ comparable to $\Delta$. Then the polygon $\widehat \Delta$ is found.

\item The distance of $\hat u_{k_0}$ with one of the points $x_1', v_j', w_l'$ becomes equal to the distance between corresponding points of $\Delta$ (suppose for example $d_{\Delta'}(\hat u_{k_0},v_j')=d_\Delta(u_{k_0},v_j)$). \\
Define $\widehat \Delta$ as the polygon obtained from $\Delta'$ replacing $\overline{x_2'x_3'}$ with $\overline{x_3'\hat u_{k_0}}*\overline{\hat u_{k_0}x_2'}$. The vertices $\hat u_1,\dots, \hat u_{k_0-1}$ of $\widehat \Delta$ corresponding to $ u_1',\dots,  u_{k_0-1}'$ will be their orthogonal projection on $\overline{x_3'\hat u_{k_0}}$, while the vertices $\hat u_{k_0+1},\dots, \hat u_{k}$ of $\widehat \Delta$ corresponding to $ u_{k_0+1}',\dots,  u_{k}'$ will be their orthogonal projection on $\overline{\hat u_{k_0}x_2'}$. One then cuts $\widehat \Delta$ along $\overline{\hat u_{k_0}v_j'}$ obtaining $\widehat \Delta_1,\widehat \Delta_2$ and cuts $\Delta$ along $\overline{u_{k_0}v_j}$ obtaining $\Delta_1,\Delta_2$.\\
Since both $\Delta_1,\widehat \Delta_1$ and $\Delta_2,\widehat \Delta_2$ satisfy the hypothesis of theorem \ref{lemshort2}, one can conclude there are two 1-Lipschitz maps $\phi_i:\Delta_i\rightarrow \widehat \Delta_i$, $i=1,2$. From the equality $d_{\Delta'}(\hat u_{k_0},v_j')=d_\Delta(u_{k_0},v_j)$ one gets that it is possible to obtain a 1-Lipschitz map $\phi:\Delta\rightarrow \widehat \Delta$ sending vertices to corresponding vertices and such that $\phi(p):=\phi_i(p)$ if $p\in \Delta_i$. This proves that the distance between any two vertices of $\Delta$ is greater than or equal to the distance between the two corresponding vertices of $\widehat \Delta$.
\end{enumerate}
\end{proof}

\begin{prop}
If only two adjacent vertices $w_m,w_{m+1}$ of $\Delta$ change order in $\Delta'$ and no vertex of $\Delta$ changes side, then there is a subpolygon $\widehat \Delta$ of $\Delta'$ comparable to $\Delta$.
\end{prop}

\begin{proof}
In figure 13 is represented an example of this situation with $m=1$.
\begin{figure}[h!]
 \centering
  \includegraphics[scale=0.5]{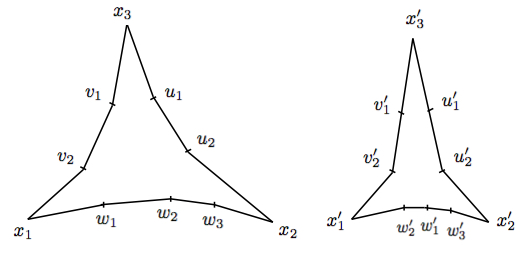}
 \caption{The order of $w_1'$ and $w_2'$ is changed.}
 \end{figure}

We will move only one vertex between $w_m'$ and $w_{m+1}'$ and prove it will always be possible to set $\hat w_m=\hat w_{m+1}:=w_{m+1}'$ or $\hat w_m=\hat w_{m+1}:=w_{m}'$.\\
Clearly, in case one sets $\hat w_m=\hat w_{m+1}:=w_{m+1}'$ then $w_{m+1}'$ will become a multiple vertex of $\widehat \Delta$ and $w_m'$ will be a vertex of $\widehat \Delta$ not in $\iota(Vertices(\Delta))$.\\
All other vertices of $\widehat \Delta$ will coincide with the corresponding vertices of $\Delta'$. As it will be clear, all our considerations will not change in case $\Delta'$ has one dimensional components, more vertices than $\Delta$ or multiple vertices.\\
We define the following subsets of the sets of vertices of $\Delta$ and $\Delta'$:
$$T_{w_m}:=\{p \text{ is a vertex of } \Delta \text{ such that } d_{\Delta}(p,w_m)\le d_{\Delta}(p,w_{m+1})\},$$
$$T_{w_{m+1}}:=\{p \text{ is a vertex of } \Delta \text{ such that } d_{\Delta}(p,w_{m+1})\le d_{\Delta}(p,w_m)\},$$
$$T_{w_m}':=\{p' \text{ is a vertex of } \Delta' \text{ such that } p\in T_{w_m}\},$$
$$T_{w_{m+1}}':=\{p' \text{ is a vertex of } \Delta' \text{ such that } p\in T_{w_{m+1}}\}.$$
The meaning of these sets is that in order to being able to impose $\hat w_m:=w_{m+1}'$ one has to check only the distances of $w_{m+1}'$ with the points of $T_{w_m}'$, since for every $p'\in T_{w_{m+1}}'$ it results 
$$d_{\Delta'}(p',w_{m+1}')\le d_{\Delta}(p,w_{m+1})\le d_{\Delta}(p,w_m).$$
For the same reason, in order to set $\hat w_{m+1}:=w_{m}'$ one has to check only the distances of $w_m$ with the points of $T_{w_{m+1}}'$.\\
It is possible to further develop this reasoning defining the following two sets:
$$\hat T_{w_m}:=\{p' \in T_{w_m}'  \text{ such that } d_{\Delta'}(p',w_m')\le d_{\Delta'}(p',w_{m+1}')\},$$
$$\hat T_{w_{m+1}}:=\{p' \in T_{w_{m+1}}'  \text{ such that } d_{\Delta'}(p',w_{m+1}')\le  d_{\Delta'}(p',w_m')\}.$$
Following our previous idea, only points of $\hat T_{w_m}$ (resp. of $\hat T_{w_{m+1}}$) can prevent one from imposing $\hat w_m:=w_{m+1}'$ (resp. $\hat w_{m+1}:=w_{m}'$).\\
The example of figure 13 is particularly simple, since the sets $\hat T_{w_1},\hat T_{w_2}$ are empty and consequently it is possible to impose both $\hat w_1=\hat w_2:=w_2'$ and $\hat w_1=\hat w_2:=w_1'$.

\begin{figure}[h!]
 \centering
  \includegraphics[scale=0.5]{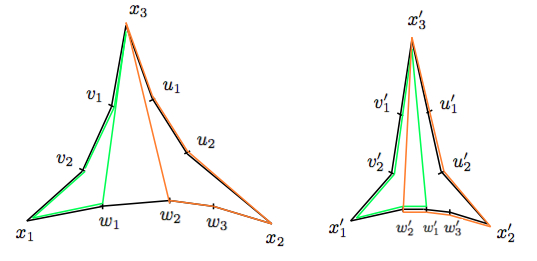}
 \caption{The convex envelope of $T_{w_1}$ and $T_{w_1}'$ is drawn in green and the convex envelope of $T_{w_2}$ and $T_{w_2}'$ is drawn in orange.}
 \end{figure}

One will not always be so lucky: we will use the following simple lemma to study the general case.
\begin{lem}
Consider $w_m',w_{m+1}'\in \overline{x_1'x_2'}$ such that $d_{\Delta'}(x_1',x_{m+1}')<d_{\Delta'}(x_1',x_m')$. \\
If there is a vertex $v'\in \overline{x_1'x_3'}$ such that  $d_{\Delta'}(v',w_m') \le d_{\Delta'}(v',w_{m+1}')$, then for every $p'\in \overline{v'x_3'},\overline{w_m'x_2'},\overline{x_2'x_3'}$ it will also follow $d_{\Delta'}(p',w_m') \le d_{\Delta'}(p',w_{m+1}')$.
\end{lem}
\begin{proof}
We will first prove $d_{\Delta'}(p_1',w_m') \le d_{\Delta'}(p_1',w_{m+1}')$ for every $p_1'\in \overline{v'w_{m}'}$. Suppose by contradiction there is a point $p_2'\in \overline{v'w_{m}'}$ such that $d_{\Delta'}(p_2',w_m') > d_{\Delta'}(p_2',w_{m+1}')$ and denote by $pr:\Delta'\rightarrow \overline{v'w_{m+1}'}$ the orthogonal projection on $\overline{v'w_{m+1}'}$. Then it would follow:
$$d_{\Delta'}(v',w_{m+1}')=d_{\Delta'}(v',pr(p_2'))+d_{\Delta'}(pr(p_2'),w_{m+1}')\le d_{\Delta'}(v',p_2')+d_{\Delta'}(p_2',w_{m+1}')<$$ $$<d_{\Delta'}(v',p_2')+d_{\Delta'}(p_2',w_{m}')=d_{\Delta'}(v',w_m')$$
which contradicts the hypothesis.\\
For every point  $p'\in \overline{v'x_3'},\overline{w_m'x_2'},\overline{x_2'x_3'}$ then define $\widetilde p:=\overline{p'w_{m+1}'}\cap\overline{v'w_m'}$ (notice that if $v'w_m'$ is not smooth then it can happen $p'=\widetilde p$ if $p'\in \overline{x_2'x_3'}$). It results:
$$d_{\Delta'}(p',w_{m+1}')=d_{\Delta'}(p',\widetilde p)+d_{\Delta'}(\widetilde p,w_{m+1}')\ge d_{\Delta'}(p',\widetilde p)+d_{\Delta'}(\widetilde p,w_{m}')\ge d_{\Delta'}(p',w_{m}').$$
\end{proof}

Notice that there can not be points $w_j'$ in $\hat T_{w_m}$ or $\hat T_{w_{m+1}}$, otherwise there would be another change of order of the vertices.\\
One can apply the preceding lemma to make the following inferences.
\begin{itemize}
\item There can not be a point $u_j'\in \hat T_{w_{m+1}}$ and a point $v_i'\in \hat T_{w_m}$, since if it results $d_{\Delta'}(v_i',w_{m}')\le d_{\Delta'}(v_i',w_{m+1}')$ then it must follow $d_{\Delta'}(u_j',w_{m}')\le d_{\Delta'}(u_j',w_{m+1}')$.
\item There can not be a point $v_i'\in \hat T_{w_{m+1}}$ and a point $u_j'\in \hat T_{w_m}$, since if it results $d_{\Delta}(v_i,w_{m+1})\le d_{\Delta}(v_i,w_m)$ then it must follow $d_{\Delta}(u_j,w_{m+1})\le d_{\Delta}(u_j,w_m)$ (since clearly there is an analogue version of the preceding lemma on $\Delta$).
\item There can not be a point $v_i'\in  \hat T_{w_{m+1}}$ and a point $v_j'\in \hat T_{w_m}$, since they would be forced to have inverted order.
\end{itemize}

One can conclude that the two sets $\hat T_{w_{m+1}},\hat T_{w_{m}}$ can not be both non-empty and consequently it is always possible to set $\hat w_m:=w_{m+1}'$ or $\hat w_{m+1}:=w_m'$.

\end{proof}

\end{document}